\documentclass[11pt]{article}

\usepackage[english]{babel}
\usepackage[T1]{fontenc}
\usepackage[utf8]{inputenc}
\usepackage{lmodern}
\usepackage{fullpage}
\usepackage{amsmath}
\usepackage{amssymb}
\usepackage{amsthm}
\usepackage{extarrows}
\usepackage{booktabs}
\usepackage{tabularx}
\usepackage{cite}
\usepackage{tikz-cd}
\usepackage{mathtools}
\usepackage{adjustbox}
\usepackage[mathcal]{eucal}
\usepackage{hyperref}
\hypersetup{
	colorlinks,
	linkcolor={red!55!black},
	citecolor={green!55!black},
	urlcolor={blue!80!black}
}
\usepackage{cleveref}

\mathchardef\mhyphen="2D
\def\on{\operatorname}

\newtheorem{theorem}{Theorem}[section]
\newtheorem{lemma}[theorem]{Lemma}
\newtheorem{conjecture}[theorem]{Conjecture}
\newtheorem{proposition}[theorem]{Proposition}
\newtheorem{corollary}[theorem]{Corollary}
\newtheorem*{corollary*}{Corollary}
\newtheorem{introthm}{Theorem}

\newtheorem{introcor}{Corollary}

\theoremstyle{definition}
\newtheorem{notation}[theorem]{Notation}
\newtheorem{construction}[theorem]{Construction}
\newtheorem{definition}[theorem]{Definition}
\newtheorem{remark}[theorem]{Remark}
\newtheorem{example}[theorem]{Example}
\newtheorem*{example*}{Example}

\title{Spherical monadic adjunctions of stable infinity categories}
\author{Merlin Christ}
\date{\today}

\begin{document}

\maketitle
\abstract{ 
This paper concerns spherical adjunctions of stable $\infty$-categories and their relation to monadic adjunctions. We begin with a proof of the 2/4 property of spherical adjunctions in the setting of stable $\infty$-categories. The proof is based on the description of spherical adjunctions as \mbox{$4$-periodic} semiorthogonal decompositions given by Halpern-Leistner, Shipman \cite{HLS16} and Dyckerhoff, Kapranov, Schechtman, Soibelman \cite{DKSS19}. We then describe a class of examples of spherical adjunctions arising from local systems on spheres. The main result of this paper is a characterization of the sphericalness of a monadic adjunctions in terms of properties of the monad. Namely, a monadic adjunction is spherical if and only if the twist functor is an equivalence and commutes with the unit map of the monad. This characterization is inspired by work of Ed Segal \cite{Seg18}.
}

\tableofcontents

\newpage

\section*{Introduction}
Seidel and Thomas \cite{ST01} introduced the notion of a \textit{spherical object} for the construction of autoequivalences of the bounded derived category of a smooth complex projective variety. Their approach was inspired by homological mirror symmetry; the symplectic analogue of the constructed autoequivalence is a Dehn twist associated to a Lagrangian sphere. Spherical objects were soon generalized, leading to the notion of a \textit{spherical functor}, applicable to the more general construction of autoequivalences of triangulated categories. A functor $F:\mathcal{A}\rightarrow \mathcal{B}$ between (suitably enhanced) triangulated categories that admits a right adjoint $G$ is called spherical if
\begin{itemize} 
\item the endofunctor $T_\mathcal{A}=\on{cone}(id_\mathcal{A}\xrightarrow{u}GF)$ of $\mathcal{A}$ given by the pointwise cone of the unit transformation $u$ of the adjunction $F\dashv G$ is an equivalence and
\item the endofunctor $T_\mathcal{B}=\on{cone}(FG\xrightarrow{cu}id_\mathcal{B})[-1]$ of $\mathcal{B}$ given by the pointwise shifted cone of the counit transformation of the adjunction $F\dashv G$ is an equivalence.
\end{itemize} 
The functors $T_\mathcal{A}$ and $T_\mathcal{B}$ are called twist and cotwist functor, respectively. Following \cite{DKSS19}, we adopt a different convention and call the adjunction $F\dashv G$ a \textit{spherical adjunction}, emphasizing that sphericalness is a property of an adjunction. Currently, much interest in spherical adjunctions arises because spherical adjunctions appear as local data of the conjectured concept of perverse schober on surfaces, cf.~\cite{KS14}.\\

The cone in a triangulated category is infamous for not being functorial, so that one needs to choose an enhancement to define the twist and cotwist functors. Common choices are pretriangulated dg-categories and pretriangulated $\mathbb{A}_\infty$-categories. We choose stable $\infty$-categories as the enhancement. This choice provides us access to the powerful framework developed by Lurie in \cite{HTT,HA}. 

We will begin in \Cref{sec3} by extending basic aspects of the theory of spherical adjunctions to the setting of stable $\infty$-categories, most notably the so called 2/4 property, appearing in the setting of dg-categories in \cite{AL17}. Our proof of the 2/4 property is based on the correspondence between spherical adjunctions and $4$-periodic semiorthogonal decompositions due to \cite{HLS16}, which was extended to the setting of stable $\infty$-categories in \cite{DKSS19}. In \Cref{sec4}, we study the following family of spherical adjunctions.

\begin{example*} 
Let $f:X\rightarrow Y$ be a spherical fibration between Kan complexes, i.e.~the simplicial analogue of a Serre fibration between spaces whose fibers are homotopy equivalent to $n$-spheres. Let $\mathcal{D}$ be any stable $\infty$-category. We call the functor categories $\on{Fun}(Y,\mathcal{D})$ and $\on{Fun}(X,\mathcal{D})$ the stable $\infty$-categories of local systems with values in $\mathcal{D}$ on $Y$ and $X$, respectively. Consider the pullback functor 
\[f^*:\on{Fun}(Y,\mathcal{D})\longrightarrow \on{Fun}(X,\mathcal{D})\]
along $f$ with right adjoint $f_*$. The adjunction $f^*\dashv f_*$ is spherical.
\end{example*}

For $\mathcal{D}=\mathcal{D}^b(k)$ the bounded derived category of a field $k$, the above examples also appear in \cite[1.11]{KS14} in the setting of pretriangulated dg-categories. 

These examples of spherical adjunctions can be seen as arising from a family of spherical objects. By allowing any stable $\infty$-category $\mathcal{D}$ as the target for the local systems, we show that such families of spherical objects also exist in the setting of spectrally enriched $\infty$-categories which cannot be treated as dg- or $A_\infty$-categories, such as the $\infty$-category of spectra. 

In the remaining \Cref{sec5} we turn towards spherical monadic adjunctions. Let $\mathcal{D}$ be a stable $\infty$-category. A monad $M:\mathcal{D}\rightarrow \mathcal{D}$ on $\mathcal{D}$ is an algebra object in the $\infty$-category of endofunctors, i.e~equipped with a multiplication map $m:M^2\rightarrow M$, a unit map $u:id_{\mathcal{D}}\rightarrow M$ and further data exhibiting associativity and unitality. An important source of monads is given by adjunctions. Every adjunction $F\dashv G$ determines a monad $GF$ called the adjunction monad, with unit given by the adjunction unit. Every monad $M$ is equivalent to the adjunction monad of its associated monadic adjunction
\[ F:\mathcal{D}\leftrightarrow \on{LMod}_M(\mathcal{D}):G\,.\]
Here $\on{LMod}_M(\mathcal{D})$ denotes the stable $\infty$-category of left modules in $\mathcal{D}$ over the monad $M$, see \Cref{sec1.4} below, and is also sometimes called the Eilenberg-Moore $\infty$-category of the monad. The stable Kleisli $\infty$-category  $\overline{\on{LMod}^{\on{free}}_M(\mathcal{D})}\subset \on{LMod}_\mathcal{D}(D)$ is defined as the smallest stable, full subcategory containing all free $M$-modules. The monad $M$ on $\mathcal{D}$ is also equivalent to the adjunction monad of the stable Kleisli adjunction
\[ F:\mathcal{D}\leftrightarrow \overline{\on{LMod}^{\on{free}}_M(\mathcal{D})}:G\,,\]
which is defined as the restriction of the monadic adjunction. We show in \Cref{sec1.4} that this adjunction is the minimal adjunction of stable $\infty$-categories with adjunction monad $M$. If the monadic adjunction is spherical, then the stable Kleisli adjunction is as a restriction also spherical. The main result of this paper is the following characterization of the sphericalness of a monadic adjunction in terms of the properties of the adjunction monad.

\begin{introthm}[\Cref{sphmndthm}]
\label{thm1}
Let $\mathcal{D}$ be a stable $\infty$-category and let $M:\mathcal{D}\rightarrow \mathcal{D}$ be a monad with unit $u:\on{id}_{\mathcal{D}}\rightarrow M$. Consider the endofunctor $T_{\mathcal{D}}=\on{cone}(\on{id}_{\mathcal{D}}\xrightarrow{u}M)\in \on{Fun}(\mathcal{D},\mathcal{D})$. The following conditions are equivalent.
\begin{enumerate}
\item The endofunctor $T_{\mathcal{D}}$ is an equivalence and the unit $u$ satisfies $T_\mathcal{D}u\simeq uT_\mathcal{D}$.
\item The monadic adjunction $F:\mathcal{D}\leftrightarrow \on{LMod}_M(\mathcal{D}):G$ is spherical.
\end{enumerate}
\end{introthm}

The main ingredient in the proof of \Cref{thm1} is Lurie's far reaching $\infty$-categorical Barr-Beck theorem. \Cref{thm1} extends and completes a discussion in \cite[Section 3.2]{Seg18}. Using \Cref{thm1}, we can also extend the main result of \cite{Seg18} to the setting of stable $\infty$-categories.

\begin{introcor}[\Cref{cor:sph}]\label{cor1}
Every autoequivalence of a stable $\infty$-category arises as the twist functor of a spherical adjunction. 
\end{introcor}

For the proof of \Cref{cor1}, we consider the square-zero extension monad, whose monadic adjunction is spherical and whose twist functor is equivalent to the autoequivalence. In \cite{Seg18}, Segal recovers the autoequivalence via the Kleisli adjunction of the square-zero extension monad, which in the setting of dg-categories is Morita equivalent to the stable Kleisli adjunction of the monad. We can thus describe any autoequivalence as the twist functor of both a monadic adjunction and a stable Kleisli adjunction. 

\Cref{thm1} has further implications for all spherical adjunctions. We describe in \Cref{sphmndprop} a characterization of the sphericalness of an adjunction (not necessarily monadic) similar to \Cref{thm1}. 

We conclude this work by providing a counterexample to an expectation raised by Segal, see "Proposition" 3.10 in \cite{Seg18}. Given a spherical adjunction $F:\mathcal{D}\leftrightarrow \mathcal{C}:G$, such that $F$ is essentially surjective, the expectation states that the adjunction can be recovered from the twist functor $T_{\mathcal{D}}$ and its section $s:T_{\mathcal{D}}[-1]\rightarrow \on{id}_{\mathcal{D}}$, see \Cref{sec4.1}. Such an adjunction is by \Cref{stbKleisli} equivalent to the stable Kleisli adjunction of the adjunction monad $M=GF$ and thus determined by the adjunction monad. We provide the following counter-example to the expectation, that the monad can be recovered from the section of the twist functor. Given a field $k$, there are two algebra structures on $k\oplus k$, the square-zero algebra structure and the product algebra structure. These determine different monads with underlying endofunctor $(k\oplus k)\otimes_k \mhyphen:\mathcal{D}(k)\rightarrow \mathcal{D}(k)$. The two different arising stable Kleisli adjunctions are spherical and have equivalent twist functors and sections.

\subsection*{Acknowledgements}
I deeply thank my supervisor Tobias Dyckerhoff for his availability and guidance. I further wish to thank Ed Segal for valuable comments on a draft of this paper and in particular for suggesting a simplification of the \Cref{ex1,ex2}. Finally, I also wish to thank two anonymous referees for helping improve the readability of the paper. This paper is based on the author's Master's thesis. The author acknowledges support by the Deutsche Forschungsgemeinschaft under Germany’s Excellence Strategy – EXC 2121 “Quantum Universe” – 390833306.

\section{Preliminaries}\label{sec2}
In this paper we assume familiarity with the basic notions of $\infty$-category theory as developed by Joyal and Lurie. In \Cref{in1}, we provide an informal account of the role of stable $\infty$-categories in the remainder of the text. In particular, we discuss how stable $\infty$-categories compare to other notions of enhancements of triangulated categories. In the \Cref{sec1.2,sec2.2,sec1.4} we introduce some concepts and results from the theory of $\infty$-categories. In \Cref{sec1.5} we recall the relationship between semiorthogonal decompositions and adjunctions of stable $\infty$-categories, as appearing in \cite{DKSS19}. For an extensive treatment of the theory of $\infty$-categories and stable $\infty$-categories we refer to \cite{HTT} and \cite{HA}, respectively.
The contents of this section are not original, except for the discussion of the stable Kleisli $\infty$-category in \Cref{sec1.4}.

\subsection{Stable \texorpdfstring{$\infty$}{infinity}-categories}\label{in1}
For any $\infty$-category, there is an associated $1$-category called the homotopy category. A stable $\infty$-category is an $\infty$-category with additional properties which ensure that the homotopy category can be given the structure of a triangulated category. Practically all examples of interest of triangulated categories appear as the homotopy category of a stable $\infty$-category. We can thus use stable $\infty$-categories as an enhancement of triangulated categories. Stable $\infty$-categories posses a number of convenient features, which distinguishes them from other choices of enhancements.
\begin{itemize}
\item There is an intrinsic notion of limits and colimits in any $\infty$-category. We can thus characterize many appearing mathematical objects via universal properties. 
\item By using universal properties many definitions can be stated in a simpler form and many theorems become more general. 
\item There is an intrinsic notion of Kan extension. They allow for much flexibility in constructing functors and $\infty$-categories.
\item Any pretriangulated dg- or $\mathbb{A}_\infty$-category can be regarded as a stable $\infty$-category via the nerve construction. There are further examples of stable $\infty$-categories such as the stable $\infty$-category of spectra.
\end{itemize}
We present a small dictionary for translating between triangulated categories and stable $\infty$-categories in \Cref{table1}.

\begin{table}[h!]
\setlength{\extrarowheight}{9pt} \setlength{\tabcolsep}{12pt}
\centering
\begin{tabular}{cc}
triangulated categories & stable $\infty$-categories\\ 
\midrule
object & $0$-simplex or vertex or object \\
morphism & $1$-simplex or edge or morphism \\
shift functor & suspension functor \\
inverse shift functor & loop functor \\
distinguished triangle & fiber and cofiber sequence\\
mapping cone  & cofiber \\
mapping cone shifted by $[-1]$ & fiber \\
adjunction & biCartesian fibration
\end{tabular}
\caption{A small dictionary for translating between triangulated categories and stable $\infty$-categories.}
\label{table1}
\end{table}

The concept of stable $\infty$-categories first appeared in \cite{Lur06}, building on the idea of a stable model category, originating in \cite{Hov99}. Applying the theory of stable $\infty$-categories has become feasible after the foundational works \cite{HTT} and \cite{HA}. Our hope is that this paper exemplifies that the theory of stable $\infty$-categories may provide new and efficient tools for studying spherical adjunctions, both for theoretical considerations and practical computations.

\subsection{Adjunctions and the Grothendieck construction}\label{sec1.2}
We begin by recalling the notion of an adjunction of $\infty$-categories. An adjunction of 1-categories can be defined as a pair of functors with unit and counit transformations satisfying the triangle identities. For adjunctions of $\infty$-categories, one needs to keep track of further data. There are specific 2-simplicies exhibiting the triangle identities and there are further simplicies exhibiting further compatibility properties of the $2$-simplicies and so on. To avoid making any choice of such data, one adopts a different approach. One encodes the data of an adjunction in a functor $\mathcal{M}\rightarrow \Delta^1$ that is at the same time a Cartesian fibration as well as a coCartesian fibration, i.e.~a functor with certain lifting properties. A Cartesian and coCartesian fibration is called a biCartesian fibration.

\begin{definition}
An adjunction between $\infty$-categories $\mathcal{A}$ and $\mathcal{B}$ is a biCartesian fibration $p:\mathcal{M}\rightarrow \Delta^1$ with fibers $\mathcal{A}$ and $\mathcal{B}$ over $0$ and $1$. 
\end{definition}

Given a biCartesian fibration $p:\mathcal{M}\rightarrow \Delta^1$ with fibers $\mathcal{A}$ and $\mathcal{B}$, we can associate a functor $G:\mathcal{B}\rightarrow \mathcal{A}$ to the the Cartesian fibration $p$ and a second functor $F:\mathcal{A}\rightarrow \mathcal{B}$ to the coCartesian fibration $p$, see \cite[Section 5.2.1]{HTT}. We call any pair of functors $F,G$ arising in this way adjoint and write $F\dashv G$. The functor $F$ is called the left adjoint and the functor $G$ the right adjoint. The unit, counit and further coherence data can also be recovered using the lifting properties of $p$.

Let $\mathcal{M}\rightarrow \Delta^1$ be a biCartesian fibration as above. We are interested in distinguished edges $e:a\rightarrow b$ in $\mathcal{M}$ with $a\in \mathcal{A}$ and $b\in \mathcal{B}$, called coCartesian and Cartesian edges, defined via certain lifting properties in \cite[Section 2.4.1]{HTT}. The edge $e$ is coCartesian if, informally, it describes the application of the functor $F$ to $a$ and is thus in particular equivalent to an edge of the form $a\rightarrow F(a)$. The edge $e$ is Cartesian if, informally, it describes the application of the functor $G$ to $b$ and is thus in particular equivalent to an edge of the form $G(b)\rightarrow b$. 

\begin{notation}
Let $e$ be an edge as above. We write $e:a\xrightarrow{\ast}b$ if $e$ is a Cartesian edge and $e:a\xrightarrow{!}b$ if $e$ is a coCartesian edge.
\end{notation}

\begin{definition}
Let $p:\mathcal{M}\rightarrow \Delta^1$ be an adjunction between $\mathcal{A}$ and $\mathcal{B}$. An edge $e:a\rightarrow a'$ in $\mathcal{A}$ is called a unit map if there exists a diagram in $\mathcal{M}$ (i.e.~object of $\on{Fun}(\Delta^2,\mathcal{M})$) of the form 
\[
\begin{tikzcd}
                                 & b                     \\
a \arrow[r, "e"] \arrow[ru, "!"] & a' \arrow[u, "\ast"']
\end{tikzcd}
\]
with $b\in \mathcal{B}$. 
An edge $e':b\rightarrow b'$ in $\mathcal{B}$ is called a counit map if there exists a diagram in $\mathcal{M}$ of the form 
\[
\begin{tikzcd}
                  & a \arrow[ld, "!"'] \arrow[d, "\ast"] \\
b \arrow[r, "e'"] & b'                                  
\end{tikzcd}
\]
with $a\in \mathcal{A}$.
\end{definition}
By the properties of the involved coCartesian and Cartesian edges, unit and counit maps are determined up to contractible choice by their domain and target, respectively.\\

As we now explain, given a functor we can encode it in a Cartesian or coCartesian fibration by applying the Grothendieck construction. Consider more generally a small $1$-category $C$ and a functor $f:C\rightarrow \on{Set}_\Delta$  taking values in $\infty$-categories. The relative nerve construction, cf.~\cite[3.2.5.2]{HTT}, associates a coCartesian fibration $\Gamma(f)\longrightarrow N(C)$  over the nerve $N(C)$ of $C$ whose fibers are equivalent to the values of the functor $f$. We will call the relative nerve construction the covariant Grothendieck construction. We call the dual version of the relative nerve construction the contravariant Grothendieck construction, it is given by the Cartesian fibration $\chi(f)=\Gamma\left((-)^{op}\circ f\right)^{op}\rightarrow N(C)^{op}$, where $(-)^{op}$ denotes the autoequivalence of $\on{Set}_\Delta$ that assigns to a simplicial set its opposite simplicial set.

Consider the case where $C=[1]$ and  $N(C)=\Delta^1$. Given $f:[1]\rightarrow \on{Set}_\Delta$ taking values in $\infty$-categories, we can identify it with a functor of $\infty$-categories denoted $\hat{f}:f(0)\longrightarrow f(1).$ Then $\Gamma(\hat{f})\coloneqq\Gamma(f)$ is spanned by the two full subcategories $f(0)=\Gamma(\hat{f})\times_{\{0\}}\Delta^1$ and $f(1)=\Gamma(\hat{f})\times_{\{1\}}\Delta^1.$ By definition of $\Gamma(\hat{f})$, there are no edges from $f(1)$ to $f(0)$ in $\Gamma(\hat{f})$. An edge $x\rightarrow y$ where $x\in f(0)$ and $y\in f(1)$ corresponds to the data of two vertices $x\in f(0),y\in f(1)$ and an edge $\hat{f}(x)\rightarrow y$ in $f(1)$. We note that the functor associated to the coCartesian fibration $\Gamma(f)\rightarrow \Delta^1$ is equivalent to $\hat{f}$. The functor $\hat{f}$ admits a right adjoint if and only if the coCartesian fibration $\Gamma(\hat{f})\rightarrow \Delta^1$ is also Cartesian, see \cite[5.2.1.3]{HTT}. The Grothendieck construction thus allows us to recover an adjunction from any of its two adjoints.

We end this section with recording a technical fact, which is best referred to as needed. Consider a map of simplicial sets $f:X\rightarrow Y$ and an $\infty$-category $\mathcal{D}$ such that all functors in $\on{Fun}(X,\mathcal{D})$ admit colimits. Then there is an adjunction $f_!:\on{Fun}(X,\mathcal{D})\leftrightarrow \on{Fun}(Y,\mathcal{D}):f^*$ between the left Kan extension functor and the pullback functor by \cite[4.3.3.7]{HTT}. A coCartesian edge in the biCartesian fibration $\chi(f^*)\rightarrow \Delta^1$ is of the form $F\rightarrow f_!(F)$ with $F\in \on{Fun}(X,\mathcal{D})$. Such an edge corresponds to the data of a left extension of $F$ by $f_!(F)$, as defined in \cite[4.3.3.1]{HTT}. That extension is a left Kan extension.

\begin{lemma}
\label{Kanextlem}
Let $f:X\rightarrow Y$ be a map of simplicial sets and let $\mathcal{D}$ be an $\infty$-category such that all functors in $\on{Fun}(X,\mathcal{D})$ admit colimits. Consider the adjunction $f_!:\on{Fun}(X,\mathcal{D})\leftrightarrow \on{Fun}(Y,\mathcal{D}):f^*$ between the left Kan extension functor and the pullback functor. An edge $F\rightarrow G$ in $\chi(f^*)$ with $F\in \on{Fun}(X,\mathcal{D})$ and $G\in \on{Fun}(Y,\mathcal{D})$ is coCartesian if and only if the induced edge $F\rightarrow f^*(G) $ is a left Kan extension.
\end{lemma}

\subsection{Kan extensions and the construction of the twist functors}\label{sec2.2}

We begin by recalling the general procedure for constructing functors via Kan extensions. We are given an $\infty$-category $\mathcal{C}$ and two simplicial sets $A'\subset A\in \on{Set}_\Delta$. We define an $\infty$-category $\mathcal{D}$ as the full subcategory of the $\infty$-category of diagrams $\on{Fun}(A,\mathcal{C})$ spanned by functors that are either left or right Kan extensions of their restriction to $A'$. We denote $\mathcal{D}'\coloneqq\on{Fun}(A',\mathcal{C})$. We find the restriction functor $\on{res}:\mathcal{D}\longrightarrow \mathcal{D}'$ to be a trivial fibration, see \cite[4.3.2.15]{HTT}, and in particular an equivalence of $\infty$-categories. A trivial fibration has the property that there exists an essentially unique section, i.e.~the space of sections 
\[\on{Fun}_{\mathcal{D}'}(\mathcal{D}',\mathcal{D})\coloneqq \on{Fun}(\mathcal{D}',\mathcal{D}) \times_{\on{Fun}(\mathcal{D}',\mathcal{D}')} \{id_{\mathcal{D}'}\}\]
is contractible. We can make the choice of one such section $F: \mathcal{D}'\rightarrow \mathcal{D}$ and have constructed an interesting functor. 

\begin{remark}\label{rem:twistdef}
We illustrate the above procedure with the construction of the twist and cotwist functors associated to an adjunction of stable $\infty$-categories in \Cref{twistconstr}, as appearing in \cite{DKSS19}. Before doing so, let us comment on equivalent way to describe the twist and cotwist functors.

Consider an adjunction $\mathcal{M}\rightarrow \Delta^1$ of stable $\infty$-categories, associated to a pair of functor $F:\mathcal{A}\leftrightarrow \mathcal{B}:G$ and choose a  unit $u$ and counit $cu$. Then the twist functor $T_\mathcal{A}$ defined in \Cref{twistconstr} is equivalent to the functor given by the cofiber of $u$ in the stable $\infty$-category $\on{Fun}(\mathcal{A},\mathcal{A})$. Dually, the cotwist functor $T_\mathcal{B}$ is equivalent to the fiber of $cu$ in the stable $\infty$-category $\on{Fun}(\mathcal{B},\mathcal{B})$. We will use the definition provided in \Cref{twistconstr} because it is better applicable in proofs.
\end{remark}

\begin{construction}\label{twistconstr}
Let $p:\mathcal{M}\rightarrow \Delta^1$ be an adjunction between stable $\infty$-categories $\mathcal{A}$ and $\mathcal{B}$. We split the construction of the twist functor into six steps, denoted {\bf a)} to {\bf f)} below.

\noindent {\bf a)} Consider the full subcategory $\mathcal{D}_1$ of $\on{Fun}_{\Delta^1}(\Delta^{1},\mathcal{M})$ spanned by functors that are a left Kan extension relative $p$ of their restriction to $\on{Fun}_{\Delta^1}(\Delta^{\{0\}},\mathcal{M})$. The vertices of $\mathcal{D}_1$ can be depicted as 
\[ a\xlongrightarrow{!}b\] 
with $a\in \mathcal{A}$ and $b\in \mathcal{B}$. By \cite[4.3.2.15]{HTT}, the restriction functor to $a$ is a trivial fibration from $\mathcal{D}_1$ to $\mathcal{A}$. 

\noindent {\bf b)} We consider $\Lambda^2_2$ as lying over $\Delta^1$, by mapping $0,1$ to $0$ and $2$ to $1$. Consider the inclusion $\Delta^1\simeq \Delta^{\{1,2\}}\subset \Lambda^2_2$ and the resulting restriction functor $\on{Fun}_{\Delta^1}(\Lambda^2_2,\mathcal{M})\rightarrow \on{Fun}_{\Delta^1}(\Delta^1,\mathcal{M})$. We define $\mathcal{D}_2$ to be the full subcategory of $\on{Fun}_{\Delta^1}(\Lambda^2_2,\mathcal{M})$ spanned by diagrams whose restriction to $\on{Fun}_{\Delta^1}(\Delta^1,\mathcal{M})$ lies in $\mathcal{D}_1$ and that are a right Kan extension relative $p$ of their restriction to $\on{Fun}_{\Delta^1}(\Delta^1,\mathcal{M})$. The vertices of $\mathcal{D}_2$ are of the form
\[ 
 \begin{tikzcd}
  a\arrow[dr, "!"]&\\
  a'\arrow[r, "\ast"]& b
 \end{tikzcd}
\] 
where $a,a'\in \mathcal{A}$ and $b\in \mathcal{B}$. The restriction functor defines a trivial fibration from $\mathcal{D}_2$ to $\mathcal{D}_1$. 

\noindent {\bf c)} We consider $\Delta^2$ as lying over $\Delta^1$, by mapping $0,1$ to $0$ and $2$ to $1$. Let $E$ denote the set of all degenerate edges of $\Delta^2$ together with the edge $\Delta^{\{1,2\}}$. The inclusion $(\Lambda^2_2,E\cap (\Lambda^2_2)_1)\subset (\Delta^2,E)$ is by \cite[3.1.1.1]{HTT} a marked anodyne morphism of marked simplicial sets, so that the restriction functor $\on{Fun}_{\Delta^1}((\Delta^2,E),\mathcal{M}^\natural)\rightarrow\on{Fun}_{\Delta^1}((\Lambda^2_2,E\cap (\Lambda^2_2)_1),\mathcal{M}^\natural)$ is a trivial fibration by \cite[3.1.3.4]{HTT}, see also \cite[3.1.1.8]{HTT} for the notation $M^\natural$. Consider the pullback of simplicial sets $\mathcal{D}_3=\on{Fun}_{\Delta^1}((\Delta^2,E),\mathcal{M}^\natural)\times_{\on{Fun}_{\Delta^1}((\Lambda^2_2,E\cap (\Lambda^2_2)_1),\mathcal{M}^\natural)}\mathcal{D}_2$. The $\infty$-category $\mathcal{D}_3$ is equivalent to the full subcategory of $\on{Fun}_{\Delta^1}(\Delta^2,\mathcal{M})$ spanned by vertices of the following form.
\[
 \begin{tikzcd}
   a\arrow[dr, "!"]\arrow[d]&\\
   a'\arrow[r, "\ast"]& b
 \end{tikzcd}
\]
The functor from $\mathcal{D}_3$ to $\mathcal{D}_2$ contained in the defining pullback diagram of $\mathcal{D}_3$ is again a trivial fibration. 

\noindent {\bf d)} We consider the simplicial set $\Delta^{\{0,1'\}}$ as lying over $\Delta^1$ via the constant map with value $0$. Let $\mathcal{D}_4$ be the full subcategory of $\on{Fun}_{\Delta^1}(\Delta^2\coprod_{\Delta^{\{0\}}} \Delta^{\{0,1'\}},\mathcal{M})$ spanned by functors that are a $p$-relative right Kan extension of their restriction to $\Delta^2$ and whose restriction to $\on{Fun}_{\Delta^1}(\Delta^2,\mathcal{M})$ is contained in $\mathcal{D}_3$. The vertices of $\mathcal{D}_4$ are diagrams of the following form.
\[ 
\begin{tikzcd}
0 & a \arrow[rd, "!"] \arrow[d] \arrow[l] &   \\
  & a' \arrow[r, "\ast"]                  & b
\end{tikzcd}
\]
We find the restriction functor to be a trivial fibration from $\mathcal{D}_4$ to $\mathcal{D}_3$. 

\noindent {\bf e)} We consider the simplicial set $\Delta^1\times\Delta^1$ lying over $\Delta^1$ with the constant map with value $0$. Consider the full subcategory $\mathcal{D}_5$ of $\on{Fun}_{\Delta^1}(\Delta^2 \coprod_{\Delta^{\{0,1\}}} \Delta^1\times\Delta^1,\mathcal{M})$ spanned by functors that are $p$-relative left Kan extensions of their restriction to $\Delta^2\coprod_{\Delta^{\{0\}}} \Delta^{\{0,1'\}}$  and whose restriction to $\on{Fun}_{\Delta^1}(\Delta^2\coprod_{\Delta^{\{0\}}} \Delta^{\{0,1'\}},\mathcal{M})$ is contained in $\mathcal{D}_4$.  The vertices of $\mathcal{D}_5$ can be depicted as follows,
\begin{equation}\label{squeq}
\begin{tikzcd}
0 \arrow[d] & a \arrow[rd, "!"] \arrow[d] \arrow[l] \arrow[ld, "\square", phantom] &   \\
a''         & a' \arrow[r, "\ast"] \arrow[l]                              & b
\end{tikzcd}
\end{equation}
 with $a,a',a''\in \mathcal{A}$ and $b\in \mathcal{B}$. The box $\square$ in the center of the commutative square in diagram \eqref{squeq} denotes that the square is biCartesian, i.e.~both pullback and pushout. The square thus describes a fiber and cofiber sequence. We find the restriction functor to be a trivial fibration from $\mathcal{D}_5$ to $\mathcal{D}_4$.
 
\noindent {\bf f)} Composing the above constructed trivial fibrations, we obtain the trivial fibration $R:\mathcal{D}_5\rightarrow \mathcal{A}$ given by the restriction functor to the vertex $a$.  We define up to contractible choice the twist functor 
\[ T_\mathcal{A}:\mathcal{A}\longrightarrow \mathcal{A} \]
as the composition of a section of $R$ with the restriction functor to $a''$ from $\mathcal{D}_5$ to $\mathcal{A}$.\\

The construction of the cotwist functor is dual. We consider the full subcategory $\mathcal{D}'$ of $\on{Fun}_{\Delta^1}(\Delta^2\coprod_{\Delta^{\{1,2\}}}\Delta^1\times\Delta^1,\mathcal{M})$ spanned by functors of the following form,
\[
\begin{tikzcd}
b'' \arrow[r] \arrow[d] \arrow[rd, "\square", phantom] & b' \arrow[d] &                                       \\
0 \arrow[r]                                            & b            & a \arrow[l, "\ast"'] \arrow[lu, "!"']
\end{tikzcd}
\]
where $a\in \mathcal{A}$ and $b,b',b''\in \mathcal{B}$. Similar to before, the restriction functor $R':\mathcal{D}'\rightarrow \mathcal{B}$ to the vertex $b$ is a trivial fibration. We define up to contractible choice the cotwist functor 
\[ T_\mathcal{B}:\mathcal{B}\longrightarrow \mathcal{B} \] 
as the composition of a section of $R'$ with the restriction functor to $b''$ from $\mathcal{D}'$ to $\mathcal{B}$. 
\end{construction} 

\begin{remark}
We will often describe appearing diagram $\infty$-categories by specifying their vertices up to equivalence, leaving their construction using Kan extensions implicit. For example, consider the setup of \Cref{twistconstr} and denote the adjoint functors associated to the biCartesian fibration $p:\mathcal{M}\rightarrow \Delta^1$ by $F:\mathcal{A}\leftrightarrow \mathcal{B}:G$. We can describe the $\infty$-category $\mathcal{D}'$ as spanned by functors of the following form, 
\[
\begin{tikzcd}
T_\mathcal{B}(b) \arrow[r] \arrow[d] \arrow[rd, "\square", phantom] & FG(a) \arrow[d] &                                          \\
0 \arrow[r]                                                         & b               & G(b) \arrow[l, "\ast"'] \arrow[lu, "!"']
\end{tikzcd}
\]
up to equivalence. This notational convention will help to remember the meaning of smaller diagram $\infty$-categories and also simplify the notation in the construction of large diagram $\infty$-categories. 
\end{remark}

\subsection{Monadic adjunctions}\label{sec1.4}

In this section we recall the theory of modules over a monad and of monadic adjunctions and describe the Kleisli $\infty$-category and stable Kleisli $\infty$-category associated to a monad. 

Before turning to monads, we first recall some concepts and notation regarding monoidal $\infty$-categories. The definition is based on the formalism of $\infty$-operads, see \cite[Section 2.1]{HA}, meaning certain functors $O^\otimes\rightarrow N(\on{Fin}_\ast)$, where $\on{Fin}_\ast$ denotes the category of finite pointed sets. The associative $\infty$-operad is denoted by $\on{Assoc}^\otimes$, for the definition see \cite[4.1.1.3]{HA}. A monoidal $\infty$-category $\mathcal{C}$ is defined to be a coCartesian fibration of $\infty$-operads $\mathcal{C}^\otimes\rightarrow \on{Assoc}^\otimes$, see \cite[4.1.1.10]{HA}. The $\infty$-category $\mathcal{C}$ arises from the $\infty$-operad $\mathcal{C}^{\otimes}\rightarrow N(\on{Fin}_\ast)$ as the fiber over $\langle 1\rangle \in N(\on{Fin}_\ast)$. The coCartesian fibration $\mathcal{C}^{\otimes}\rightarrow \on{Assoc}^\otimes$ serves to encode a monoidal product $\otimes:\mathcal{C}\times\mathcal{C}\rightarrow \mathcal{C}$, a monoidal unit and the data exhibiting coherent associativity and unitality.

Let $\mathcal{D}$ be an $\infty$-category. We can turn the $\infty$-category $\on{Fun}(\mathcal{D},\mathcal{D})$ of endofunctors into a monoidal $\infty$-category via the composition monoidal structure, see \Cref{endlem} below. The monoidal product of the composition monoidal structure is given by composition of functors. Given a monoidal $\infty$-category, there is a notion of an associative algebra object, see \cite[2.1.3.1]{HA}. If the monoidal $\infty$-category is the nerve of a monoidal $1$-category, then this notion agrees with the 1-categorical notion of an associative algebra object. An associative algebra object in the monoidal $\infty$-category $\on{Fun}(\mathcal{D},\mathcal{D})$ is called a monad on $\mathcal{D}$. We can informally express the datum of a monad $M:\mathcal{D}\rightarrow \mathcal{D}$ as follows.
\begin{itemize}
\item A multiplication map $m:M\otimes M= M^2 \rightarrow M$.
\item The data expressing the coherent associativity of the map $m$.
\item A unit map $u:id_\mathcal{D}\rightarrow M$.
\item The data expressing the unitality of $u$ and $m$.
\end{itemize}
Every adjunction $F:\mathcal{D}\leftrightarrow\mathcal{C}: G$ of $\infty$-categories determines a monad $M=GF:\mathcal{D}\rightarrow\mathcal{D}$, see \cite[4.7.3.3]{HA}. We call the monad $M$ the adjunction monad of $F\dashv G$. The multiplication map of $M$ is induced by the counit of the adjunction and the unit $id_\mathcal{D}\rightarrow M$ of the monad $M$ is equivalent to the unit map of the adjunction. Every monad arises as the adjunction monad of the associated monadic adjunction. To define the monadic adjunction, we first introduce the $\infty$-category modules over the monad (the $1$-categorical analogue is called the category of algebras over the monad or also the Eilenberg-Moore category).

One can associate to any associative algebra object $A$ in a monoidal $\infty$-category $\mathcal{C}$ an $\infty$-category $\on{LMod}_A(\mathcal{C})$ of left modules in $\mathcal{C}$ over $A$. This is however not sufficiently general for our purpose. We are interested in module objects in $\mathcal{D}$ over a monad in $\on{Fun}(\mathcal{D},\mathcal{D})$. The objects of $\on{Fun}(\mathcal{D},\mathcal{D})$ act on the objects of $\mathcal{D}$ via the evaluation of functors. This can be used to exhibit the $\infty$-category $\mathcal{D}$ as left-tensored over the monoidal $\infty$-category $\on{Fun}(\mathcal{D},\mathcal{D})$, in the sense of \cite[4.2.1.19]{HA}. Given a monad $M\in \on{Fun}(\mathcal{D},\mathcal{D})$, there is thus an associated $\infty$-category $\on{LMod}_M(\mathcal{D})$ of left modules over $M$ in $\mathcal{D}$, see \cite[4.2.1.13]{HA}.

Given a monad $M:\mathcal{D}\rightarrow \mathcal{D}$ we denote by $F:\mathcal{D}\leftrightarrow \on{LMod}_M(\mathcal{D}):G$ the monadic adjunction. The functor $F$ is the free module functor and admits a right adjoint by\cite[4.2.4.8]{HA}. Any adjunction of this form is called monadic. An adjunction is called comonadic if the opposite adjunction is monadic. A functor equivalent to the right adjoint $G$ of a monadic adjunction is called a monadic functor. Lurie’s $\infty$-categorical Barr-Beck theorem \cite[4.7.3.5]{HA} characterizes monadic functors. The theorem states that a functor $G:\mathcal{C}\rightarrow\mathcal{D}$ between $\infty$-categories is monadic if and only if the following conditions are satisfied.
\begin{itemize} 
\item $G$ admits a left adjoint.
\item The functor $G$ is conservative, i.e.~reflects isomorphism.
\item The $\infty$-category $\mathcal{C}$ admits and $G$ preserves colimits of $G$-split simplicial objects, i.e.~functors $N(\Delta)^{op}\rightarrow \mathcal{C}$ such that their composition with $G$ can be extended to a split simplicial object, as defined in \cite[4.7.2.2]{HA}.
\end{itemize}
The $\infty$-categorical Barr-Beck theorem is an essential tool in the theory of $\infty$-categories and one of its corollaries \cite[4.7.3.16]{HA} can be used to identify equivalences of $\infty$-categories. 

A module in $\on{LMod}_M(\mathcal{D})$ is called free if it is equivalent to an element in the image of the left adjoint $F$ of the monadic adjunction. The essential image of $F$ is called the Kleisli $\infty$-category which we denote by $\on{LMod}_M^{\on{free}}(\mathcal{D})$. The restriction of the monadic adjunction $F:\mathcal{D}\leftrightarrow \on{LMod}_M^{\on{free}}(\mathcal{D}):G$ is called the Kleisli adjunction. The Kleisli adjunction is the minimal adjunction whose adjunction monad is equivalent to $M$, as captured by \Cref{Kleisli}. 

\begin{proposition}\label{Kleisli}
Let $F:\mathcal{D}\leftrightarrow \mathcal{C}:G$ be an adjunction of $\infty$-categories with adjunction monad $M$. Denote by $F'':\mathcal{D}\rightarrow \on{LMod}_M^{\on{free}}(\mathcal{D})$ the free-module functor of the monadic adjunction of $M$. There exists a fully faithful functor $F':\on{LMod}_{M}^{\on{free}}(\mathcal{D})\rightarrow \mathcal{C}$ making the following diagram commute.
\begin{equation}\label{tri1}
\begin{tikzcd}
\on{LMod}_M^{\on{free}}(\mathcal{D}) \arrow[rr, "F'"] &                                               & \mathcal{C} \\
                                                     & \mathcal{D} \arrow[ru, "F"'] \arrow[lu, "F''"] &            
\end{tikzcd}
\end{equation}
In particular, there exists an equivalence of $\infty$-categories $\on{LMod}^{\on{free}}_M(\mathcal{D})\simeq \on{Im}(F)$.
\end{proposition}

\begin{proof}
We adapt the proof of \cite[4.7.3.13]{HA}. Denote the monadic functor of $M$ by $G''$. Consider the canonical functor $G':\mathcal{C}\rightarrow \on{LMod}_M(\mathcal{D})$, defined in the discussion following Proposition 4.7.3.3 in \cite{HA}. There exist an equivalence of functors $G''\circ G'\simeq G$. We define a functor 
\begin{equation}\label{fhat}\hat{F}: \on{LMod}_M(\mathcal{D})^{op}\xrightarrow{j} \on{Fun}(\on{LMod}_M(\mathcal{D}),\mathcal{S})\xrightarrow{\circ G'} \on{Fun}(\mathcal{C},\mathcal{S})\,,\end{equation}
where $j$ denotes the Yoneda embedding of $\on{LMod}_M(\mathcal{C})^{op}$. For $d\in \mathcal{D}$, there exists an equivalence in $\on{Fun}(\mathcal{C},\mathcal{S})$
\begin{equation}\label{ffeq} \hat{F}F''(d)=\on{Map}_{\on{LMod}_M(\mathcal{D})}(F''(d),G'(\mhyphen))\simeq \on{Map}_{\mathcal{D}}(d,G''G'(\mhyphen))\simeq \on{Map}_{\mathcal{C}}(F(d),\mhyphen)\,. \end{equation} 
It follows that the image of ${\on{LMod}_M^{\on{free}}(\mathcal{D})^{op}}$ under $\hat{F}$ is given by the full subcategory\linebreak \mbox{$\on{Im}(F)^{op}\subset \mathcal{C}^{op}\subset \on{Fun}(\mathcal{C},\mathcal{S})$} of functors corepresentable by an object in the essential image \mbox{of $F$}. We hence obtain a functor 
\[  F':\on{LMod}_M^{\on{free}}(\mathcal{D})\subset \on{LMod}_M(\mathcal{D})\xrightarrow{\hat{F}^{op}}j'(\on{Im}(F)^{op})^{op}\xrightarrow{(j'^{-1})^{op}} \on{Im}(F)\subset \mathcal{C}\,,\] 
where $j':\mathcal{C}^{op}\rightarrow \on{Fun}(\mathcal{C},\mathcal{S})$ denotes the Yoneda embedding of $\mathcal{C}^{op}$. The equivalence \eqref{ffeq} is functorial in $d\in \mathcal{D}$, so that we see that the diagram \eqref{tri1} commutes.

For any $d\in \mathcal{D}$, the equivalence \eqref{ffeq} defines an edge $e:F''(d)\rightarrow G'F'F''(d)$, which is adjoint to the edge $e':d'\rightarrow G''G'F'F''(d)\simeq GF(d)$ which is a unit map of the adjunction $F\dashv G$. The adjunctions $F\dashv G$ and $F''\dashv G''$ have the same associated adjunction monad $M$, so that $e'$ is also a unit map of the adjunction $F''\dashv G''$, showing that $e$ is an equivalence. Hence $G'$ restricts to a functor $G':\on{Im}(F)\rightarrow \on{LMod}_M^{\on{free}}(\mathcal{D})$, which is, by construction of $F'$, right adjoint to $F':\on{LMod}_M^{\on{free}}(\mathcal{D})\rightarrow \on{Im}(F)$ with the unit map at $d$ given by the equivalence $e$. This shows that $F'$ is fully faithful.
\end{proof}

\begin{lemma}
 Let $\mathcal{D}$ be a stable $\infty$-category and $M:\mathcal{D}\rightarrow \mathcal{D}$ a monad whose underlying endofunctor is exact. The $\infty$-category $\on{LMod}_M(\mathcal{D})$ is stable.
\end{lemma}

\begin{proof}
 This is \cite[7.1.1.4]{HA}.
\end{proof}

The Kleisli $\infty$-category $\on{LMod}_M^{\on{free}}(\mathcal{D})$ is in general not stable, even if $M$ is exact. However, if $M$ is a monad whose underlying endofunctor is an exact endofunctor of a stable $\infty$-category there is a stable version of the Kleisli $\infty$-category which is the minimal adjunction of stable $\infty$-categories whose adjunction monad is equivalent to $M$. This is captured by \Cref{stbKleisli}.

\begin{definition}~
\begin{itemize}
\item A full subcategory $\mathcal{D}$ of a stable $\infty$-category $\mathcal{C}$ is called a stable subcategory if it closed under finite limits and colimits in $\mathcal{C}$ (in particular $\mathcal{D}$ contains a zero object) and is closed under equivalences in $\mathcal{C}$.
\item Let $\mathcal{C}$ be a stable $\infty$-category and $\mathcal{C}'$ a full subcategory. The stable closure $\overline{\mathcal{C}'}$ of $\mathcal{C}'$ is the smallest stable subcategory of $\mathcal{C}$ which contains $\mathcal{C}'$.
\end{itemize}
\end{definition}

We call the stable closure $\overline{\on{LMod}_M^{\on{free}}(\mathcal{D})}$ of the Kleisli $\infty$-category in $\on{LMod}_M(\mathcal{D})$ the stable Kleisli $\infty$-category.

\begin{proposition}\label{stbKleisli}
Let $F:\mathcal{D}\leftrightarrow \mathcal{C}:G$ be an adjunction of stable $\infty$-categories with adjunction monad $M$. Denote by $F'':\mathcal{D}\rightarrow \on{LMod}_M^{\on{free}}(\mathcal{D})\subset \on{LMod}_M(\mathcal{D}))$ the free-module functor of the monadic adjunction of $M$. There exists a fully faithful functor $F':\overline{\on{LMod}_{M}^{\on{free}}(\mathcal{D})}\rightarrow \mathcal{C}$ making the following diagram commute.
\begin{equation*}
\begin{tikzcd}
\overline{\on{LMod}_M^{\on{free}}(\mathcal{D})} \arrow[rr, "F'"] &                                               & \mathcal{C} \\
                                                     & \mathcal{D} \arrow[ru, "F"'] \arrow[lu, "F''"] &            
\end{tikzcd}
\end{equation*}
The functor $F'$ induces an equivalence
\[ \overline{\on{LMod}_M^{\on{free}}(\mathcal{D})}\simeq \on{Im}(F')=\overline{\on{Im}(F)}\,.\] 
\end{proposition}

Before we can prove \Cref{stbKleisli}, we need to give a more explicit description of the stable closure of a subcategory.

\begin{notation}
Let $\mathcal{C}$ be a stable $\infty$-category and let $\mathcal{C}'\subset \mathcal{C}$ be a full subcategory such that $0\in \mathcal{C}'$ and  $\mathcal{C}'$ is closed under deloopings in $\mathcal{C}$. Denote by $\on{cl}(\mathcal{C}')$ the full subcategory of $\mathcal{C}$ spanned by objects which are equivalent to the cofiber in $\mathcal{C}$ of an edge in $\mathcal{C}'\subset \mathcal{C}$.   
\end{notation}

\begin{lemma}\label{stbcllem}
Let $\mathcal{C}$ be a stable $\infty$-category and $\mathcal{C}'\subset \mathcal{C}$ a full subcategory such that $0\in \mathcal{C}'$ and $\mathcal{C}'$ is closed under deloopings in $\mathcal{C}$. Then the stable closure of $\mathcal{C}'$ in $\mathcal{C}$ is equivalent to  the direct limit in $\on{Cat}_\infty$ of the diagram
\begin{equation}\label{dirlim} \mathcal{C}'\subset \on{cl}(\mathcal{C}')\subset \on{cl}^2(\mathcal{C}')\subset \dots \,.\end{equation}
\end{lemma}

\begin{proof}
We denote the direct limit of \eqref{dirlim} by $\mathcal{C}''$. We can consider $\mathcal{C}''$ as the full subcategory of $\mathcal{C}$ spanned by the objects of $\on{cl}^i(\mathcal{C}')$ for all $i>0$. Given a morphism $\alpha:x\rightarrow y$ in $\mathcal{C}''$, we find $i>0$ such that $\alpha$ lies in $\on{cl}^i(\mathcal{C}')$. The cofiber sequence exhibiting the cofiber of $\alpha$ in $\mathcal{C}''$ lies in $\on{cl}^{i+1}(\mathcal{C}')\subset \mathcal{C}''$ and is by construction also a cofiber sequence in $\mathcal{C}$. Using that $\mathcal{C}'$ is closed under deloopings in $\mathcal{C}$, it follows that $\on{cl}^i(\mathcal{C}')$ is also closed under deloopings for all $i>0$. It follows that $\mathcal{C}''\subset \mathcal{C}$ is a stable subcategory. Any stable subcategory $\mathcal{D}\subset \mathcal{C}$ satisfies $\on{cl}(\mathcal{D})=\mathcal{D}$. Thus for all $i>0$ we find $\on{cl}^i(\mathcal{C}')\subset \on{cl}^i(\overline{\mathcal{C}'})=\overline{\mathcal{C}'}$. Thus $\mathcal{C}''\subset \overline{\mathcal{C}'}$, which implies by definition $\mathcal{C}''=\overline{\mathcal{C}'}$.
\end{proof}

\begin{remark}\label{clrem}
 \Cref{stbcllem} can be informally summarized as showing that any object $c\in \overline{\mathcal{C}'}$ is obtained as a finitely many times repeated cofiber, starting with a finite collections of objects of $\mathcal{C}'$.
\end{remark}

\begin{proof}[Proof of \Cref{stbKleisli}]
Denote the monadic functor of $M$ by $G''$. Consider the functor
\begin{equation}\label{fhat2}\hat{F}: \on{LMod}_M(\mathcal{D})^{op}\xrightarrow{j} \on{Fun}(\on{LMod}_M(\mathcal{D}),\mathcal{S})\xrightarrow{\circ G'} \on{Fun}(\mathcal{C},\mathcal{S})\,,\end{equation} 
from the proof of \Cref{Kleisli}. Denote by $\mathcal{X}\subset \on{LMod}_M(\mathcal{D})$ the full subcategory spanned by objects whose image under $\hat{F}$ is a corepresentable functor. The functor $\hat{F}$ preserves all limits in $\on{LMod}_M(\mathcal{D})^{op}$ by \cite[5.1.3.2]{HTT} and using that limits are computed pointwise in functor categories. The full subcategory of $\on{Fun}(\mathcal{C},\mathcal{S})$ of corepresentable functors is closed under limits. This implies that $\mathcal{X}$ is closed under colimits in $\on{LMod}_M(\mathcal{D})$. The stable Kleisli $\infty$-category $\overline{\on{LMod}_M^{\on{free}}(\mathcal{D})}$ is contained in $\mathcal{X}$: this follows from $\on{LMod}^{\on{free}}_M(\mathcal{D})\subset \mathcal{X}$, that $\on{LMod}^{\on{free}}_M(\mathcal{D})$ is closed under delooping in $\on{LMod}_M(\mathcal{D})$ and \Cref{stbcllem}. As in the proof of \Cref{Kleisli}, using the inverse of the Yoneda embedding $j'$ of $\mathcal{C}^{op}$, we can thus find a functor $F':\overline{\on{LMod}_M^{\on{free}}(\mathcal{D})}\rightarrow \mathcal{C}$.

 For any $x\in \mathcal{X}$, we denote by $u_x:x\rightarrow G'F'(x)$ the edge in $\on{LMod}_M(\mathcal{D})$ being mapped to $id_{F'(x)}$ via the equivalence $\on{Map}_{\on{LMod}_M(\mathcal{D})}(x,G'F'(x))\simeq \on{Map}_{\mathcal{C}}(F'(x),F'(x))$. Consider the Cartesian fibration $p:\chi(G')\rightarrow \Delta^1$. For $x\in \mathcal{X}\subset \chi(G')\times_{\Delta^1}\Delta^{\{0\}}$, the edge $x\rightarrow F'(x)$ in $\chi(G')$ lying over $0\rightarrow 1$ and corresponding to $u_x:x\rightarrow G'F'(x)$ is $p$-coCartesian by \cite[2.4.4.3]{HTT}. By performing a construction as in the proof of \cite[5.2.2.8]{HTT}, we can produce a natural transformation $u:\mathcal{X}\rightarrow \on{Fun}(\Delta^1,\on{LMod}_M(\mathcal{D}))$ from the identity functor to $G'\circ F'$ such that $u(x)\simeq u_x$. Denote by $\mathcal{X}_0\subset \mathcal{X}$ the full subcategory consisting of $x\in \mathcal{X}$ such that $u_x$ is an equivalence. We have shown in the proof of \Cref{Kleisli} that $\mathcal{X}_0$ contains $\on{LMod}_M^{\on{free}}(\mathcal{D})$. The functor $G'$ is exact, because $G''\circ G'\simeq G$ is exact and $G''$ is conservative. The functor $G'F'$ is thus exact. It follows that $\mathcal{X}_0$ is closed under the formation of finite colimits. Thus $\overline{\on{LMod}_M^{\on{free}}(\mathcal{D})}\subset \mathcal{X}_0$. This shows that $F'$ is a fully faithful functor and thus an equivalence onto its essential image. It also follows that $\on{Im}(F')$ is a stable subcategory of $\mathcal{C}$.

Let $x\in \overline{\on{LMod}_M^{\on{free}}(\mathcal{D})}$. Then $x$ is given by a finitely many times repeated cofiber starting with a set of objects in $\on{LMod}^{\on{free}}_M(\mathcal{D})$. Thus $F'(x)\in \on{Im}(F')$ is given by a finitely many times repeated cofiber starting with a set of objects in $\on{Im}(F'|_{\on{LMod}_M^{\on{free}}(\mathcal{D})})=\on{Im}(F)$. We deduce $\on{Im}(F')\subset \overline{\on{Im}(F)}$. Using that $\on{Im}(F')$ is a stable subcategory of $\mathcal{C}$ and that $\on{Im}(F)\subset \on{Im}(F')$, it follows by definition $\overline{\on{Im}(F)}\subset \on{Im}(F')$. We have shown the desired equality $\on{Im}(F')=\overline{\on{Im}(F)}$.
\end{proof}

We end this section with a (well known) construction of monads from algebra objects in monoidal $\infty$-categories and a characterization of the monads associated to free-forget adjunctions of algebra objects.

\begin{lemma}\label{endlem}
Let $\mathcal{D}$ be an $\infty$-category. The $\infty$-category $\on{Fun}(\mathcal{D},\mathcal{D})$ is an endomorphism object for $\mathcal{D}\in \on{Cat}_\infty$ and in particular admits a monoidal structure, called the composition monoidal structure.

If $\mathcal{D}$ is left tensored over a monoidal $\infty$-category $\mathcal{C}$, then there exists an $\on{Assoc}^\otimes$-monoidal functor $i:\mathcal{C}^\otimes\rightarrow \on{Fun}(\mathcal{D},\mathcal{D})^\otimes$, in the sense of \cite[2.1.3.7]{HA}. In particular, given an associative algebra object $A\in \mathcal{C}$, $i(A)\in \on{Fun}(\mathcal{D},\mathcal{D})$ inherits the structure of a monad. We denote $i(A)$ by $A \otimes \mhyphen$\,.
\end{lemma}

\begin{proof}
We show that $\mathcal{E}\coloneqq\on{Fun}(\mathcal{D},\mathcal{D})$ is an endomorphism object for $\mathcal{D}$ in $\on{Cat}_\infty$, meaning a terminal object of the endomorphism $\infty$-category $\on{Cat}_\infty[\mathcal{D}]$ introduced in \cite[4.7.1.1]{HA}. $\on{Cat}_\infty$ is Cartesian closed, the functor  $\mathcal{D}\times\mhyphen:\on{Cat}_\infty\rightarrow \on{Cat}_\infty$ admits a right adjoint given by the functor $\on{Fun}(\mathcal{D},\mhyphen)$. Spelling out the definition, it follows that there exists an equivalence
\[\on{Cat}_\infty[\mathcal{D}]\simeq \on{Fun}_{\Delta^1}(\Delta^1,\Gamma(\mathcal{D}\times \mhyphen))\times_{\on{Fun}_{\Delta^1}(\{1\},\Gamma((\mathcal{D}\times \mhyphen))} \{\mathcal{D}\}\,.\]
Using the equivalence $\Gamma(\mathcal{D}\times \mhyphen)\simeq \chi(\on{Fun}(\mathcal{D},\mhyphen))$, it follows
\[ \on{Cat}_\infty[\mathcal{D}]\simeq \on{Cat}_{\infty/\mathcal{E}}\,.\] 
Thus $\mathcal{E}$ is a terminal object of $\on{Cat}_\infty[\mathcal{D}]$ and admits a monoidal structure by \cite[4.7.1.40, 4.1.3.19]{HA}. 

The second part of the Lemma follows directly from \cite[4.7.1.40.(2)]{HA} and \cite[4.1.3.19]{HA}. 
\end{proof}

\begin{example}
Let $\mathcal{D}$ be a monoidal $\infty$-category. Then $\mathcal{D}$ is canonically left tensored over itself. \Cref{endlem} implies that there exists an $\on{Assoc}^\otimes$-monoidal functor $\mathcal{D}^\otimes\rightarrow \on{Fun}(\mathcal{D},\mathcal{D})^\otimes$.
\end{example}

\begin{lemma}\label{algmndlem}
Let $\mathcal{D}$ be an $\infty$-category left tensored over a monoidal $\infty$-category $\mathcal{C}$. Let $A \in \mathcal{C}$ be an associative algebra object. Consider the monad $M=A \otimes \mhyphen$. Then there exists an equivalence $\on{LMod}_M(\mathcal{D})\simeq \on{LMod}_A(\mathcal{D})$ identifying the forgetful functor $\on{LMod}_A(\mathcal{D})\rightarrow \mathcal{D}$ with the monadic functor of $M$.
\end{lemma}

\begin{proof}
Let $p_1,p_2:\mathcal{O}_1^\otimes,\mathcal{O}_2^\otimes\rightarrow \mathcal{LM}^\otimes$ be coCartesian fibrations exhibiting $\mathcal{D}$ as left tensored over $\mathcal{C}$ and $\on{Fun}(\mathcal{D},\mathcal{D})$, respectively. Using that $\on{Fun}(\mathcal{D},\mathcal{D})$ is terminal in $\on{Alg}_{\mathbb{A}_\infty}(\on{Cat}_\infty)_{/\on{Fun}(\mathcal{D},\mathcal{D})}$, it follows from \cite[4.7.1.40.(2)]{HA}, that there exists an $\mathcal{LM}^\otimes$-monoidal functor $i':\mathcal{O}^\otimes_1\rightarrow \mathcal{O}^\otimes_2$. By definition $\on{LMod}_A(\mathcal{D})\subset \on{Fun}_{\mathcal{LM}^{\otimes}}(\mathcal{LM}^\otimes,\mathcal{O}_1^\otimes)$ and $\on{LMod}_M(\mathcal{D})\subset \on{Fun}_{\mathcal{LM}^{\otimes}}(\mathcal{LM}^\otimes,\mathcal{O}_2^\otimes)$. Composition with $i'$ determines a functor $f:\on{LMod}_A(\mathcal{D})\rightarrow \on{LMod}_M(\mathcal{D})$ making the following diagram commute,
\[
\begin{tikzcd}
\on{LMod}_A(\mathcal{D}) \arrow[rd, "G"'] \arrow[rr, "f"] &             & \on{LMod}_M(\mathcal{D}) \arrow[ld, "G'"] \\
                                                      & \mathcal{D} &                                      
\end{tikzcd}
\]
where $G$ and $G'$ are the forgetful functors. Denote the left adjoints of $G$ and $G'$ by $F$ and $F'$, respectively. The functor $G'$ is the monadic functor of the monad $M$. The functor $G$ is also monadic by the $\infty$-categorical Barr-Beck theorem and \cite[3.2.2.6,~3.4.4.6]{HA}.  We observe that the unit maps $id_\mathcal{D}\rightarrow GF\simeq A \otimes \mhyphen$ and $id_\mathcal{D}\rightarrow G'F'\simeq A \otimes \mhyphen$ are both equivalent to the unit map of the monad $M$. We apply \cite[4.7.3.16]{HA} and deduce that $f$ is an equivalence.
\end{proof}

\subsection{Semiorthogonal decompositions and adjunctions}\label{sec1.5} 

In this section we closely follow \cite{DKSS19} in describing semiorthogonal decompositions of stable $\infty$-categories and their relation to adjunctions of stable $\infty$-categories.

\begin{definition}\label{soddef}
Let $\mathcal{V}$ be a stable $\infty$-category and let $\mathcal{A},\mathcal{B}$ be an ordered pair of stable subcategories of $\mathcal{V}$. Denote by $\{\mathcal{A},\mathcal{B}\}$ the full subcategory of $\on{Fun}(\Delta^1,\mathcal{V})$ spanned by vertices of the form $a\rightarrow b$ with $a\in \mathcal{A}$ and $b\in \mathcal{B}$. The fiber functor $\on{fib}$, assigning to a morphism its fiber, restricts to a functor 
\begin{equation}\label{fibreseq}
\on{fib}: \{\mathcal{A},\mathcal{B}\}\longrightarrow \mathcal{V}\,.
\end{equation} 
We call the ordered pair $(\mathcal{A},\mathcal{B})$ a semiorthogonal decomposition of $\mathcal{V}$  if the functor \eqref{fibreseq} is an equivalence.
\end{definition}

We next associative to any semiorthogonal decomposition $(\mathcal{A},\mathcal{B})$ an inner fibration $p:\chi(\mathcal{A},\mathcal{B})\rightarrow \Delta^1$. The construction is similar to the Grothendieck construction. We will relate the existence of further semiorthogonal decompositions to properties of the fibration $p$.

\begin{construction}\label{chiconstr}
Let $\mathcal{V}$ be a stable $\infty$-category with a semiorthogonal decomposition $(\mathcal{A},\mathcal{B})$. We define the simplicial set $\chi(\mathcal{A},\mathcal{B})$ by defining an $n$-simplex of $\chi(\mathcal{A},\mathcal{B})$ to correspond to the following data.
\begin{itemize}
\item An $n$-simplex $j:\Delta^n\rightarrow \Delta^1$ of $\Delta^1$.
\item An $n$-simplex $\sigma:\Delta^n\rightarrow \mathcal{V}$ such that $\sigma(\Delta^{j^{-1}(0)})\subset \mathcal{A}$ and  $\sigma({\Delta^{j^{-1}(1)}})\subset \mathcal{B}$.
\end{itemize}
We define the face and degeneracy maps to act on an $n$-simplex $(j,\sigma)\in \chi(\mathcal{A},\mathcal{B})_n$ componentwise.
The forgetful map $p:\chi(\mathcal{A},\mathcal{B})\rightarrow \Delta^1$ is an inner fibration. 
In particular, we see that $\chi(\mathcal{A},\mathcal{B})$ is an $\infty$-category.
\end{construction}

The next lemma shows that in the setting of \Cref{chiconstr}, the stable $\infty$-category $\mathcal{V}$ can be recovered up to equivalence from the inner fibration $p:\chi(\mathcal{A},\mathcal{B})\rightarrow \Delta^1$.

\begin{lemma}\label{equlem}
Let $\mathcal{V}$ be a stable $\infty$-category with a semiorthogonal decomposition $(\mathcal{A},\mathcal{B})$. There exists an equivalence 
\begin{equation} \label{decompiso}
 \on{Fun}_{\Delta^1}(\Delta^1,\chi(\mathcal{A},\mathcal{B}))\simeq \mathcal{V}\,,\end{equation}
between the $\infty$-category of sections of $p:\chi(\mathcal{A},\mathcal{B})\rightarrow \Delta^1$ and $\mathcal{V}$.
\end{lemma}

\begin{proof}
The inclusions $\on{Fun}_{\Delta^1}(\Delta^1,\chi(\mathcal{A},\mathcal{B})),\{\mathcal{A},\mathcal{B}\}\subset \on{Fun}(\Delta^1,\mathcal{V})$ are fully faithful functors whose images are identical, given by the edges $a\rightarrow b$ where $a\in \mathcal{A}$ and $b\in \mathcal{B}$. Thus there exists an equivalence $\on{Fun}_{\Delta^1}(\Delta^1,\chi(\mathcal{A},\mathcal{B}))\simeq \{\mathcal{A},\mathcal{B}\}$. Using the equivalence $\{\mathcal{A},\mathcal{B}\}\simeq \mathcal{V}$ the statement follows.
\end{proof}

\begin{definition}
Let $\mathcal{V}$ be a stable $\infty$-category and $\mathcal{A}\subset \mathcal{V}$ a stable subcategory. We define 
\begin{itemize}
\item the right orthogonal $\mathcal{A}^\perp$ to be the full subcategory of $\mathcal{V}$ spanned by those vertices $x\in \mathcal{V}$ such that for all $a\in \mathcal{A}$ the mapping space $\on{Map}_\mathcal{V}(a,x)$ is contractible. 
\item the left orthogonal $\prescript{\perp}{}{\mathcal{A}}$ to be the full subcategory of $\mathcal{V}$ spanned by those vertices $x\in \mathcal{V}$ such that for all $a\in \mathcal{A}$ the mapping space $\on{Map}_\mathcal{V}(x,a)$ is contractible.
\end{itemize}
\end{definition}

\begin{lemma}\label{perplem}
 Let $\mathcal{V}$ be a stable $\infty$-category with a semiorthogonal decomposition $(\mathcal{A},\mathcal{B})$. Then $\mathcal{B}= \prescript{\perp}{}{\mathcal{A}}$ and $\mathcal{A}= \mathcal{B}^\perp$. 
\end{lemma}

\begin{proof}
We show $\mathcal{A}=\mathcal{B}^\perp$. The relation $\mathcal{B} = \prescript{\perp}{}{\mathcal{A}}$ can be shown analogously. By assumption the fiber functor  $\on{fib}: \{\mathcal{A},\mathcal{B}\}\rightarrow \mathcal{V}$ is an equivalence. Under any inverse of $\on{fib}$, the stable subcategory $\mathcal{A}$ of $\mathcal{V}$ is identified with the stable subcategory of $\{\mathcal{A},\mathcal{B}\}$ spanned by sections of the form $a\rightarrow 0$. We show that $a\rightarrow b\in \{\mathcal{A},\mathcal{B}\}\simeq \mathcal{V}$ lies in $\mathcal{B}^\perp$ if and only if $b\simeq 0$. Using that $\{\mathcal{A},\mathcal{B}\}$ is a full subcategory of $\on{Fun}(\Delta^1,\mathcal{V})$, we obtain for $0\rightarrow b',a\rightarrow b \in \{\mathcal{A},\mathcal{B}\}$ an equivalence
 \[ \on{Map}_{\{\mathcal{A},\mathcal{B}\}}(0\rightarrow b',a\rightarrow b) \simeq \on{Map}_{\on{Fun}(\Delta^1,\mathcal{V})}(0\rightarrow b', a\rightarrow b)\,. \] 
Using that $0\in \mathcal{A}$ is initial, we obtain that the restriction functor 
\[ \on{Map}_{\on{Fun}(\Delta^1,\mathcal{V})}(0\rightarrow b', a\rightarrow b)\longrightarrow \on{Map}_{\mathcal{V}}(b',b)\]
is also an equivalence. We deduce that under any inverse of $\on{fib}$ the $\infty$-category $\prescript{\perp}{}{B}$ is identified with the stable subcategory of $\{\mathcal{A},\mathcal{B}\}$ spanned by edges $a\rightarrow b$ satisfying $\on{Map}_{\mathcal{V}}(b',b)\simeq \ast$ for all $b'\in \mathcal{B}$. This shows that $a\rightarrow b\in \{\mathcal{A},\mathcal{B}\}\simeq \mathcal{V}$ lies in $\prescript{\perp}{}{\mathcal{B}}$ if and only if $b\simeq 0$. 
\end{proof}

\begin{definition}\label{cartdef}
Let $\mathcal{V}$ be a stable $\infty$-category with a semiorthogonal decomposition $(\mathcal{A},\mathcal{B})$ and consider the map $p:\chi(\mathcal{A},\mathcal{B})\rightarrow \Delta^1$. We call
\begin{itemize}
\item $(\mathcal{A},\mathcal{B})$ Cartesian if $p$ is a Cartesian fibration and coCartesian if $p$ is a coCartesian fibration.
\item an edge $a\rightarrow b$ in $\mathcal{V}$ with $a\in \mathcal{A}$ and $b\in \mathcal{B}$ an $(\mathcal{A},\mathcal{B})$-Cartesian edge if it is $p$-Cartesian and an $(\mathcal{A},\mathcal{B})$-coCartesian edge if it is $p$-coCartesian, considered as an edge in $\chi(\mathcal{A},\mathcal{B})$.
\item if $(\mathcal{A},\mathcal{B})$ is Cartesian, the functor associated to the Cartesian fibration $p$ the right gluing functor associated to $(\mathcal{A},\mathcal{B})$. 
\item if $(\mathcal{A},\mathcal{B})$ is coCartesian, the functor associated to the Cartesian fibration $p$ the left gluing functor associated to $(\mathcal{A},\mathcal{B})$.
\end{itemize}
\end{definition}
 
\begin{definition}
 Let $\mathcal{V}$ be a stable $\infty$-category and let $i:\mathcal{A}\rightarrow \mathcal{V}$ be the inclusion of a stable subcategory. Then $\mathcal{A}\subset \mathcal{V}$ is called 
\begin{itemize}
\item left admissible if $i$ admits a left adjoint,
\item right admissible if $i$ admits a right adjoint,
\item admissible if $i$ admits both left and right adjoints.
\end{itemize}
\end{definition}

The relation between Cartesian and coCartesian semiorthogonal decompositions and admissible subcategories is as follows.

\begin{proposition}\label{admprop}
 Let $\mathcal{V}$ be a stable $\infty$-category and let $\mathcal{A}\subset \mathcal{V}$ be a stable subcategory. Then the following are equivalent.
\begin{enumerate}
\item $\mathcal{A}$ is an admissible subcategory of $\mathcal{V}$.
\item $(\mathcal{A}^\perp,\mathcal{A})$ and $(\mathcal{A},\prescript{\perp}{}{\mathcal{A}})$ form semiorthogonal decompositions of $\mathcal{V}$.
\item $(\mathcal{A}^\perp,\mathcal{A})$ is a coCartesian semiorthogonal decomposition of $\mathcal{V}$.
\item $(\mathcal{A},\prescript{\perp}{}{\mathcal{A}})$ is a Cartesian semiorthogonal decomposition of $\mathcal{V}$.
\end{enumerate}
\end{proposition}

\Cref{admprop} follows from \Cref{dkssprop1} and \Cref{dkssprop2}, which we cite from \cite{DKSS19}.

\begin{proposition}\label{dkssprop1}
 Let $\mathcal{V}$ be a stable $\infty$-category and let $i:\mathcal{A}\rightarrow \mathcal{V}$ be the inclusion of a stable subcategory. Then:
\begin{enumerate}
\item The pair $(\mathcal{A},\prescript{\perp}{}{\mathcal{A}})$ forms a semiorthogonal decomposition of $\mathcal{V}$ if and only if $\mathcal{A}\subset \mathcal{V}$ is left admissible. If these conditions hold then we further have $\prescript{\perp}{}{\mathcal{A}}=\ker(p)$  for $p$ any left adjoint of the inclusion $i:\mathcal{A}\rightarrow \mathcal{V}.$ 
\item The pair $(\mathcal{A}^\perp,\mathcal{A})$ forms a semiorthogonal decomposition of $\mathcal{V}$ if and only if $\mathcal{A}\subset \mathcal{V}$ is right admissible.
  If these conditions hold then we further have ${\mathcal{A}}^\perp=\ker(q)$  for $q$ any right adjoint of the inclusion $i:\mathcal{A}\rightarrow \mathcal{V}.$ 
\end{enumerate}
\end{proposition}

\begin{proposition}\label{dkssprop2}
 Let $\mathcal{V}$ be a stable $\infty$-category with a semiorthogonal decomposition $(\mathcal{A},\mathcal{B})$.
\begin{enumerate}
\item The semiorthogonal decomposition $(\mathcal{A},\mathcal{B})$ is Cartesian if and only if the subcategory $\mathcal{A}\subset \mathcal{V}$ is right admissible. If these equivalent conditions hold, then the restriction of any right adjoint of $i:\mathcal{A}\subset \mathcal{V}$ to $\mathcal{B}$ is a right gluing functor.
\item The semiorthogonal decomposition $(\mathcal{A},\mathcal{B})$ is coCartesian if and only if the subcategory $\mathcal{B}\subset \mathcal{V}$ is left admissible. If these equivalent conditions hold, then the restriction of any left adjoint of $i:\mathcal{B}\subset \mathcal{V}$ to $\mathcal{A}$ is a left gluing functor. 
\end{enumerate}
\end{proposition}

\begin{lemma}\label{cartsect}
Let $(\mathcal{A},\mathcal{B})$ be a semiorthogonal decomposition of a stable $\infty$-category $\mathcal{V}$. Then:
\begin{enumerate}
\item Suppose $(\mathcal{A},\mathcal{B})$ is Cartesian. Under any inverse of the functor $\on{fib}:\{\mathcal{A},\mathcal{B}\}\rightarrow \mathcal{V}$, the subcategory $\mathcal{A}^\perp\subset \mathcal{V}$ is identified with the full subcategory of $\{\mathcal{A},\mathcal{B}\}$ spanned by the $(\mathcal{A},\mathcal{B})$-Cartesian edges.
\item  Suppose $(\mathcal{A},\mathcal{B})$ is coCartesian. Under any inverse of the functor $\on{fib}:\{\mathcal{A},\mathcal{B}\}\rightarrow \mathcal{V}$, the subcategory $\prescript{\perp}{}{\mathcal{B}}\subset \mathcal{V}$ is identified with the full subcategory of $\{\mathcal{A},\mathcal{B}\}$ spanned by the $(\mathcal{A},\mathcal{B})$-coCartesian edges.
\end{enumerate}
\end{lemma}

\begin{proof}
 We only show statement 1, statement 2 can be shown analogously. Using \Cref{admprop} we deduce that $\mathcal{A}\subset \mathcal{V}$ is an admissible subcategory. By \cite[A.8.20]{HA}, we obtain that $\mathcal{A}$ and $\mathcal{A}^\perp$ form a recollement of $\mathcal{V}$ as defined in \cite[A.8.1]{HA}. The description of $\mathcal{A}^\perp$ as spanned by $(\mathcal{A},\mathcal{B})$-Cartesian edges is thus provided in \cite[A.8.7]{HA}.
\end{proof}

We cite the following lemma from \cite{DKSS19}.

\begin{lemma}\label{mutlem}
Let $\mathcal{V}$ be a stable $\infty$-category and let $\mathcal{A}\subset \mathcal{V}$ be an admissible subcategory. Consider a biCartesian square 
\[
\begin{tikzcd}
  d\arrow[r]\arrow[dr, phantom, "\square"]\arrow[d]& a\arrow[d]\\
  0\arrow[r]& b
\end{tikzcd}
\]
in $\mathcal{V}$. Then the following are equivalent:
\begin{enumerate}
\item $d$ is an object of $\mathcal{A}^\perp$, $a$ is an object of $\mathcal{A}$ and $b$ is an object of $\prescript{\perp}{}{\mathcal{A}}$.
\item The edge $d\rightarrow a$ is $(\mathcal{A}^\perp,\mathcal{A})$-coCartesian.
\item The edge $a\rightarrow b$ is $(\mathcal{A},\prescript{\perp}{}{\mathcal{A}})$-Cartesian.
\end{enumerate}
\end{lemma}

\begin{remark}\label{mutrem}
The restriction functors from the full subcategory of $\on{Fun}(\Delta^1\times\Delta^1,\mathcal{V})$ spanned by diagrams satisfying the equivalent conditions 1.-3. of \Cref{mutlem} to the vertices $d$ or $b$ constitute trivial fibrations. By choosing a section of one of these trivial fibrations we hence get an essentially unique equivalence $m_\mathcal{A}: \mathcal{A}^\perp\rightarrow \prescript{\perp}{}{\mathcal{A}},$ which we call the mutation around $\mathcal{A}$. We observe that precomposition with $m_\mathcal{A}$ identifies the right gluing functor $\prescript{\perp}{}{\mathcal{A}}\rightarrow \mathcal{A}$ of the semiorthogonal decomposition $(\mathcal{A},\prescript{\perp}{}{\mathcal{A}})$ with the left gluing functor $\mathcal{A}^\perp\rightarrow \mathcal{A}$ of the semiorthogonal decomposition $(\mathcal{A}^\perp,\mathcal{A})$.
\end{remark}

The next lemma associates to any adjunction of stable $\infty$-categories a biCartesian semiorthogonal decomposition.

\begin{lemma}\label{adjsodlem}
Let $p:\mathcal{M}\rightarrow \Delta^1$ be an adjunction of stable $\infty$-categories associated to a pair of functors $F:\mathcal{Y}\leftrightarrow \mathcal{X}:G$ and consider the stable $\infty$-category $\mathcal{V}\coloneqq\on{Fun}_{\Delta^1}(\Delta^1,\mathcal{M})$ of sections of $p$. Denote by $\on{p-Ran}_{\{i\}\rightarrow \Delta^1}\mathcal{M}$ and $\on{p-Lan}_{\{i\}\rightarrow \Delta^1}\mathcal{M}$ the full subcategories of $\mathcal{V}$ spanned by functors that are right and left Kan extensions relative $p$, respectively, of their restriction to $\on{Fun}_{\Delta^1}(\{i\},\mathcal{M})\subset \mathcal{M}.$ 
\begin{enumerate}
\item The following four $\infty$-categories 
\begin{align*}
 \mathcal{A}\coloneqq\on{p-Ran}_{\{1\}\rightarrow \Delta^1}\mathcal{M} \quad\quad&  \mathcal{B}\coloneqq\on{p-Ran}_{\{0\}\rightarrow \Delta^1}\mathcal{M}\\
 \mathcal{C}\coloneqq\on{p-Lan}_{\{1\}\rightarrow \Delta^1}\mathcal{M} \quad\quad&  \mathcal{D}\coloneqq\on{p-Lan}_{\{0\}\rightarrow \Delta^1}\mathcal{M}
\end{align*}
are stable subcategories of $\mathcal{V}$. 
\end{enumerate}
The restrictions functors induce equivalences
\begin{equation}\label{idabc} 
\mathcal{A},\mathcal{C}\simeq \mathcal{X}\, ,\quad \mathcal{B},\mathcal{D}\simeq \mathcal{Y}\,.
\end{equation}
\begin{enumerate}\setcounter{enumi}{1}
\item There are coCartesian semiorthogonal decompositions $(\mathcal{A},\mathcal{B}),\,(\mathcal{B},\mathcal{C})$ and Cartesian semi-orthogonal decompositions $(\mathcal{B},\mathcal{C}),\,(\mathcal{C},\mathcal{D})$ of $\mathcal{V}$. 
\item Up to mutations and the equivalences \eqref{idabc}, the left and right gluing functor of $(\mathcal{A},\mathcal{B})$ and $(\mathcal{B},\mathcal{C})$, respectively, can be identified with $G$ and the left and right gluing functor of $(\mathcal{B},\mathcal{C})$ and $(\mathcal{C},\mathcal{D})$, respectively, can be identified with $F$.
\end{enumerate}
\end{lemma}

\begin{proof}
For statement 1, we note that the full subcategories $\mathcal{A},\mathcal{B},\mathcal{C},\mathcal{D}$ of $\mathcal{V}$ are closed under equivalences in $\mathcal{V}$ and contain a zero object of $\mathcal{V}$. The closure of the subcategories of $\mathcal{V}$ under the formation of fibers and cofibers in $\mathcal{V}$ follows from the observation that the Kan extension functors \eqref{idabc} are exact. We obtain that the subcategories are stable subcategories.

We now show statements 2 and 3. The functor $G$ is exact. By \cite[A.8.7]{HA} we obtain that $\mathcal{B}, \mathcal{A}$ forms a recollement of $\mathcal{V}$. \cite[A.8.20]{HA} implies that $\mathcal{B}\subset \mathcal{V}$ is an admissible subcategory. This shows by \Cref{admprop} that $(\mathcal{A},\mathcal{B})$ forms a coCartesian and $(\mathcal{B},\mathcal{C})$ forms a Cartesian semiorthogonal decompositions of $\mathcal{V}$. The right gluing functor of $(\mathcal{B},\mathcal{C})$, denoted $G'$, is by \Cref{dkssprop2} given by the restriction of a right adjoint of the inclusion $i:\mathcal{B}\subset \mathcal{V}$ to $\mathcal{C}$. Consider the coCartesian fibration $\Gamma(i)\rightarrow \Delta^1$. An edge $e:\Delta^1\rightarrow \Gamma(i)$ lying over $0\rightarrow 1$ corresponds to a diagram $\Delta^1\times\Delta^1\rightarrow \mathcal{M}$ 
\begin{equation}\label{pbdiag1}
\begin{tikzcd}
y'\arrow[r]\arrow[d]&0\arrow[d]\\
y\arrow[r]&x
\end{tikzcd}
\end{equation}
such that $y,y'\in \mathcal{Y}$ and $x\in \mathcal{X}$. The edge $e$ is Cartesian if and only if the diagram \eqref{pbdiag1} is a pullback diagram. \cite[4.3.1.10]{HTT} implies that the diagram \eqref{pbdiag1} is a pullback diagram if and only if the vertex $y'$ is a choice of fiber of the edge $y\rightarrow G(x)$. Hence the right adjoint of $i$ restricted to $\mathcal{C}$ is equivalent to the functor $G':\mathcal{C}\simeq \mathcal{X}\xrightarrow{G[-1]}\mathcal{Y}\simeq \mathcal{B}$. The functor $G'$ admits a left adjoint $F':\mathcal{B}\simeq \mathcal{Y}\xrightarrow{F[1]}\mathcal{X}\simeq \mathcal{C}$. The Cartesian fibration $\chi(\mathcal{B},\mathcal{C})\rightarrow \Delta^1$ classifies $G'$. We obtain that the fibration $\chi(\mathcal{B},\mathcal{C})\rightarrow \Delta^1$ is also coCartesian, classifying the functor $F'$. It follows that the semiorthogonal decomposition $(\mathcal{B},\mathcal{C})$ is coCartesian. The $\infty$-category $\mathcal{D}$ is by \cite[4.3.1.4]{HTT} spanned by $p$-coCartesian edges. We observe that the equivalence $\on{Fun}_{\Delta^1}(\Delta^1,\chi(\mathcal{A},\mathcal{B}))\simeq \mathcal{V}$ of \Cref{equlem} maps the $\chi(\mathcal{B},\mathcal{C})$-coCartesian edges to $p$-coCartesian edges. It follows from \Cref{cartsect} that $\mathcal{D}=\prescript{\perp}{}{\mathcal{C}}$ and from \Cref{admprop} that the pair $(\mathcal{C},\mathcal{D})$ forms a Cartesian semiorthogonal decomposition of $\mathcal{V}$. Composing the left gluing functor $G'$ of $(\mathcal{B},\mathcal{C})$ with the inverse of the mutation equivalence around $\mathcal{B}$ yields by \Cref{mutrem} the right gluing functor of $(\mathcal{A},\mathcal{B})$. Similarly, the right gluing functor of $(\mathcal{C},\mathcal{D})$ is obtained from the left gluing functor $F'$ of $(\mathcal{B},\mathcal{C})$ by precomposition with the mutation equivalence around $\mathcal{C}$.
\end{proof}

\begin{remark}\label{sodrem}
Let $\mathcal{M}\rightarrow \Delta^1$ be an adjunction of stable $\infty$-categories. Abusing notation, we will denote sections $s\in \mathcal{V}=\on{Fun}_{\Delta^1}(\Delta^1,\mathcal{M})$ by their restrictions $s=\left(s|_0,s|_1\right).$ To fully specify any $s\in \mathcal{V}$ involves the additional data of a map $s|_0\rightarrow G(s|_1)$. The set of vertices of the resulting subcategories of $\mathcal{V}$ from \Cref{adjsodlem} can be depicted up to equivalence as follows.
\[\mathcal{A}_0=\{\left(G(x),x\right)\}\quad\quad \mathcal{B}_0=\{\left(y,0\right)\} \quad\quad  \mathcal{C}_0=\{\left(0,x\right)\}\quad\quad \mathcal{D}_0=\{\left(y,F(y)\right)\}\]
The vertices $(G(x),x)$ in $\mathcal{A}$ have the property that the associated edge $G(x)\rightarrow G(x)$ is an equivalence. The vertices $(y,F(y))$ in $\mathcal{D}$ have the property that the associated edge $y\rightarrow GF(y)$ is a unit map of the adjunction $F\dashv G$.
\end{remark}

\Cref{adjsodlem} implies that the datum of a biCartesian semiorthogonal decomposition is equivalent to the datum of an adjunction of stable $\infty$-categories. We can also describe adjoint triples or longer sequences of adjoint functors by sequences of biCartesian semiorthogonal decompositions.

\begin{lemma}\label{adjsodlem2}
Let $\mathcal{M}\rightarrow \Delta^1$ be an adjunction of stable $\infty$-categories associated to a pair of functors $F:\mathcal{Y}\leftrightarrow \mathcal{X}: G$ such that $F$ admits a right adjoint $E$. Consider the associated semiorthogonal decompositions $(\mathcal{A},\mathcal{B}),(\mathcal{B},\mathcal{C}),(\mathcal{C},\mathcal{D})$ of $\mathcal{V}=\on{Fun}_{\Delta^1}(\Delta^1,\mathcal{M})$ from \Cref{adjsodlem} and denote $\mathcal{E}=\prescript{\perp}{}{\mathcal{D}}$.
\begin{enumerate} 
\item The semiorthogonal decomposition $(\mathcal{C},\mathcal{D})$  is also coCartesian and $(\mathcal{D},\mathcal{E})$ forms a Cartesian semiorthogonal decomposition of $\mathcal{V}$.
\item Let $x\in \mathcal{X}$. There is a $(\mathcal{C},\mathcal{D})$-coCartesian edge $e:\left(0,x\right)\xlongrightarrow{!}\left(E(x),FE(x)\right)$ in $\mathcal{V}$, such that the restriction to the second component is a unit map of the adjunction $E\dashv F$.
\end{enumerate}
\end{lemma}

\begin{proof}
We denote by $F'$ the functor classified by the Cartesian fibration $\chi(\mathcal{C},\mathcal{D})\rightarrow \Delta^1$. The functor $F'$ is equivalent to the composition of the functor $F$ with mutation equivalences. By assumption, the functor $F'$ admits a left adjoint, denoted $E'$. It follows that $\chi(\mathcal{C},\mathcal{D})\rightarrow \Delta^1$ is also coCartesian. Using \Cref{admprop}, we see that $(\mathcal{D},\mathcal{E})$ forms a Cartesian semiorthogonal decomposition of $\mathcal{V}$. To obtain the description of the coCartesian edge $e$, consider the following diagram in $\mathcal{V}$.
\[
 \begin{tikzcd}
  (0,FE(x))\arrow[d, "\ast"]&\\
  (E(x),FE(x))&(0,x)\arrow[l, "!"]\arrow[ul]
 \end{tikzcd}
\]
The edge $(0,x)\rightarrow (0,FE(x))$ is a unit map of the adjunction $E'\dashv F'$, showing that the restriction to the second component $x\rightarrow FE(x)$ is a unit map of the adjunction $E\dashv F$. 
\end{proof}

\section{Spherical adjunctions}\label{sec3}

In this section we study spherical adjunctions in the setting of stable $\infty$-categories. Because of the added generality over spherical adjunctions of dg-categories, basic properties of spherical adjunction known in the latter setting need to be proven again. A treatment of spherical adjunctions in the setting of stable $\infty$-categories also appears in \cite{DKSS19}, so that we can use all properties proven there. We show further basic properties of spherical adjunctions of stable $\infty$-categories, most notably the 2/4 property of spherical adjunctions in \Cref{sec2.3}.

We begin with the definition of a spherical adjunction of stable $\infty$-categories.
 
\begin{definition}
An adjunction $p:\mathcal{M}\rightarrow \Delta^1$ of stable $\infty$-categories $\mathcal{A}$ and $\mathcal{B}$ is called spherical if the associated twist functor $T_\mathcal{A}$ and the associated cotwist functor $T_\mathcal{B}$, as defined in \Cref{twistconstr}, are equivalences.
\end{definition}

An immediate consequence of the definition of the twist functors is that they commute with the adjoints.

\begin{lemma}\label{commlem}
Let $F:\mathcal{A}\leftrightarrow\mathcal{B}:G$ be an adjunction of stable $\infty$-categories, with twist functor $T_\mathcal{A}$ and cotwist functor $T_\mathcal{B}$.  There exist equivalences of functors
$  FT_\mathcal{A}\simeq T_\mathcal{B} F$ and $ T_\mathcal{A}G\simeq G T_\mathcal{B}.$
\end{lemma}

\begin{proof}
We construct only the equivalence $FT_\mathcal{A}\simeq T_\mathcal{B}F$, the second equivalence can be constructed analogously. Consider the full subcategory $\mathcal{D}\subset \on{Fun}((\Delta^1)^4\coprod_{\Delta^2}\Delta^4,\Gamma(F))$ spanned by diagrams of the following form.
\[
 \begin{tikzcd}
                                    & F(a) \arrow[rr, "\simeq"]                                             &                                                               & F(a) \arrow[rr]                                   &  & 0                            \\
                                    &                                                                       &                                                               &                                                   &  &                              \\
                                    & F(a) \arrow[uu, "\simeq"] \arrow[rr, "F(u_a)"] \arrow[rruu, "\simeq"] &                                                               & FGF(a) \arrow[uu, "cu_{F(a)}"'] \arrow[rr]        &  & FT_\mathcal{C}(a) \arrow[uu] \\
a \arrow[ru, "!"] \arrow[rr, "u_a", near end] &                                                                       & GF(a) \arrow[lu, "\ast"] \arrow[ru, "!"] \arrow[ruuu, "\ast"] &                                                   &  &                              \\
                                    & 0 \arrow[uu] \arrow[rr]                                               &                                                               & T_\mathcal{D}F(a) \arrow[uu] \arrow[rr, "\simeq"] &  & b \arrow[uu, "\simeq"]     
 \end{tikzcd}
\]
In the 3x3-square in the above diagram, all rows and columns are extended to biCartesian squares, which we do not depict. We also do not depict all edges of the $4$-simplex with the vertices $a,GF(a),F(a),FGF(a),F(a)$. The edge $u_a$ is a unit map of the adjunction $F\dashv G$ at $a$ and the edge $cu_{F(a)}$ is a counit map of the adjunction $F\dashv G$ at $F(a)$. Due to the involved Cartesian and coCartesian edges, equivalences and biCartesian squares, we obtain that the restriction functor $\on{res}:\mathcal{D}\rightarrow \mathcal{A}$ to the vertex $a$ is a trivial fibration.  Choosing a section of $\on{res}$ and composing with the restriction functor $\mathcal{D}\rightarrow \on{Fun}(\Delta^1,\mathcal{A})$ to the edge $T_\mathcal{D}F(a)\xrightarrow{\simeq} FT_\mathcal{C}(a)$ provides us with a natural equivalence between the functors $T_\mathcal{D}F$ and $FT_\mathcal{C}$. 
\end{proof}

\subsection{Spherical adjunctions and \texorpdfstring{$4$}{4}-periodic semiorthogonal decompositions}\label{sec4.2}

\begin{definition}
Let $\mathcal{V}$ be a stable $\infty$-category with a biCartesian semiorthogonal decomposition $(\mathcal{A},\mathcal{B})$ as defined in \Cref{sec1.5}. We call $(\mathcal{A},\mathcal{B})$ $4$-periodic, if $(\prescript{\perp}{}{\mathcal{B}}, \mathcal{A}^\perp)$ forms a semiorthogonal decomposition of $\mathcal{V}$. 
\end{definition}

The relationship between $4$-periodic semiorthogonal decompositions and spherical adjunctions is due to \cite{HLS16} and was extended to stable $\infty$-categories in \cite{DKSS19}. We cite the next proposition from \cite{DKSS19}, describing this relationship. 

\begin{proposition}\label{4pedprop1}
Let $\mathcal{V}$ be a stable $\infty$-category with a biCartesian semiorthogonal decomposition $(\mathcal{A},\mathcal{B})$. The semiorthogonal decomposition $(\mathcal{A},\mathcal{B})$ is $4$-periodic if and only if the adjunction \mbox{$\chi(\mathcal{A},\mathcal{B})\rightarrow \Delta^1$} is spherical.
\end{proposition}

\begin{corollary}\label{lrcor}
 Let $E\dashv F\dashv G$ be an adjoint triple of functors between stable $\infty$-categories. Then the adjunction $E\dashv F$ is spherical if and only if the adjunction $F\dashv G$ is spherical.
\end{corollary}

\begin{proof}
The adjoint triple $E\dashv F \dashv G$ corresponds by \Cref{adjsodlem} and \Cref{adjsodlem2} to a sequence of semiorthogonal decompositions $(\mathcal{A},\mathcal{B}),(\mathcal{B},\mathcal{C}),(\mathcal{C},\mathcal{D}),(\mathcal{D},\mathcal{E})$ of $\on{Fun}_{\Delta^1}(\Delta^1,\Gamma(F))$. \Cref{4pedprop1} implies the following.
\begin{itemize}
\item The adjunction $E\dashv F$ is spherical if and only if the semiorthogonal decomposition $(\mathcal{C},\mathcal{D})$ is $4$-periodic. 
\item The adjunction $F\dashv G$ is spherical if and only if the semiorthogonal decomposition $(\mathcal{B},\mathcal{C})$ is $4$-periodic. 
\end{itemize}
The statement follows from the observation that $(\mathcal{C},\mathcal{D})$ is $4$-periodic if and only if $(\mathcal{B},\mathcal{C})$ is $4$-periodic.
\end{proof}

We will further need the following proposition shown in \cite{DKSS19}. 

\begin{proposition}\label{4pedprop2}
Let $F:\mathcal{A}\leftrightarrow \mathcal{B}:G$ be a spherical adjunction. Then $F$ admits a further left adjoint $E$ and $G$ admits a further right adjoint $H$ which satisfy $E\simeq T_\mathcal{A}^{-1} G$ and $H\simeq F T_\mathcal{A}^{-1}$.
\end{proposition} 

In the following construction we start with an adjoint triple of functors of stable $\infty$-categories and investigate in more detail the condition of $4$-periodicity of the corresponding semiorthogonal decompositions.

\begin{construction}\label{Kconstr} 
Consider an adjunction $\mathcal{M}\rightarrow \Delta^1$ of stable $\infty$-categories, classifying the adjoint pair of functors $F:\mathcal{Y}\leftrightarrow \mathcal{X}:G$. Assume further that $F$ admits a left adjoint $E$. We denote $\mathcal{V}=\on{Fun}_{\Delta^1}(\Delta^1,\mathcal{M})$ and $\mathcal{A},\mathcal{B},\mathcal{C},\mathcal{D}$ the stable subcategories of $\mathcal{V}$ defined in \Cref{adjsodlem}. We further denote $\mathcal{E}=\prescript{\perp}{}{\mathcal{D}}$. \Cref{adjsodlem2} implies that $(\mathcal{D},\mathcal{E})$ also forms a semiorthogonal decomposition of $\mathcal{V}$. We denote the cotwist functor of the adjunction $F\dashv G$ by $T_\mathcal{X}$ and the twist functor of the adjunction $E\dashv F$ by $T_\mathcal{X}'$. We use in the following the notation introduced in \Cref{sodrem}.

Consider the full subcategory $\mathcal{K}_1$ of $\on{Fun}(\Delta^1\times\Delta^1,\mathcal{V})$ spanned by biCartesian squares 
\[ 
 \begin{tikzcd}
  \left(0,x\right)\arrow[r, "!"]\arrow[dr, phantom, "\square"]\arrow[d]& \left(E(x),FE(x)\right)\arrow[d, "\ast"]\\
  \left(0,0\right)\arrow[r]& \left(E(x),T_\mathcal{X}'(x)\right)
 \end{tikzcd}
\]
such that the top edge is $(\mathcal{C},\mathcal{D})$-coCartesian and the right edge is $(\mathcal{D},\mathcal{E})$-Cartesian. Using \Cref{mutlem}, we see that the restriction functor to the vertex $(E(x),T_\mathcal{X}'(x))$ is a trivial fibration from $\mathcal{K}_1$ to $\mathcal{E}$. All diagrams in $\mathcal{K}_1$ have by \Cref{adjsodlem2} the property that the restriction to the edge $x\rightarrow FE(x)$ is a unit map of the adjunction $E\dashv F$.

Consider the full subcategory $\mathcal{K}_2$ of $\on{Fun}(\Delta^1\times\Delta^1,\mathcal{V})$ spanned by diagrams of the form 
 \[  \begin{tikzcd}
\left(0,T_\mathcal{X}(x)\right)\arrow[r]\arrow[dr, phantom, "\square"]\arrow[d]& \left(G(x),FG(x)\right)\arrow[d]\\
\left(0,0\right)\arrow[r]& \left(G(x),x\right)
\end{tikzcd}\]
such that the restriction to $\on{Fun}(\Delta^2,\chi(F))$ is of the following form. 
\[\begin{tikzcd}
G(x)\arrow[r, "!"]\arrow[dr, "*"]& FG(x)\arrow[d]\\
& F
  \end{tikzcd}
\]
We observe that the restriction functor to the vertex $(G(x),x)$ is a trivial fibration from $\mathcal{K}_2$ to $\mathcal{A}$.

If $\mathcal{K}_1=\mathcal{K}_2$, we immediately obtain that $\mathcal{A}=\mathcal{E}$. This implies that the semiorthogonal decomposition $(\mathcal{A},\mathcal{B})$ is $4$-periodic and that the adjunction $\mathcal{M}\rightarrow \Delta^1$ is spherical. The converse is also true, as we will show \Cref{uculem3}.
\end{construction}

\subsection{Unit and counit maps of spherical adjunctions}
In this section we study the properties of the unit and counit maps of spherical adjunctions. The following lemma is due to \cite{DKSS19}.

\begin{lemma}\label{uculem1}
Let $\mathcal{V}$ be a stable $\infty$-category with a $4$-periodic semiorthogonal decomposition $(\mathcal{C},\mathcal{D})$. Consider a diagram $\Delta^2\times\Delta^2\rightarrow \mathcal{V}$ of the form 
\[ 
 \begin{tikzcd}
z\arrow[dr,phantom,"\square"] \arrow[r] \arrow[d] & b\arrow[dr,phantom,"\square"] \arrow[d, "!"] \arrow[r]     & 0 \arrow[d]  \\
c'\arrow[dr,phantom,"\square"] \arrow[d] \arrow[r] & c \arrow[dr,phantom,"\square"]\arrow[d] \arrow[r, "\ast"] & d \arrow[d] \\
0 \arrow[r]           & c'' \arrow[r]                  & z'          
\end{tikzcd}
\]
satisfying
\begin{itemize}
\item $c,c',c''\in \mathcal{C}$, $d\in \mathcal{D}$ and $b\in \mathcal{C}^\perp$,
\item the edge $b\rightarrow c$ is $(\mathcal{C}^\perp,\mathcal{C})$-coCartesian and the edge $c\rightarrow d$ is $(\mathcal{C},\mathcal{D})$-Cartesian.
\end{itemize}
Then the following statements are equivalent. 
\begin{enumerate} 
\item The edge $b\rightarrow c''$ is $(\mathcal{C}^\perp,\mathcal{C})$-Cartesian.  
\item The edge $c'\rightarrow d$ is $(\mathcal{C},\mathcal{D})$-coCartesian.
\end{enumerate}
\end{lemma}

\begin{proof}
By \Cref{mutlem} the edge $b\rightarrow c''$ is $(\mathcal{C}^\perp,\mathcal{C})$-Cartesian if and only if $z\in \mathcal{C}^{\perp\perp}$ and the edge $c'\rightarrow d$ is $(\mathcal{C},\mathcal{D})$-coCartesian if and only if $z'\in \prescript{\perp}{}{\mathcal{D}}$. We obtain that statement 1 and statement 2 are equivalent using the equivalence $z[1]\simeq z'$ and that by $4$-periodicity $\mathcal{C}^{\perp\perp}=\prescript{\perp}{}{\mathcal{D}}$.
\end{proof}

\begin{remark}\label{ucurem}
Let $F:\mathcal{Y}\leftrightarrow:\mathcal{X}:G$ be a spherical adjunction of stable $\infty$-categories. Denote the left adjoint of $F$ by $E$. Consider the following fiber and cofiber sequence in $\mathcal{X}$. 
\[ 
 \begin{tikzcd}
  x' \arrow[d]\arrow[dr, phantom, "\square"] \arrow[r] & x \arrow[d] \\
  0 \arrow[r]           & x''          
 \end{tikzcd}
\]
\Cref{uculem1} implies that the following two conditions are equivalent.
\begin{enumerate}
\item The edge $x'\rightarrow x$ is a unit map of the adjunction $E\dashv F$.
\item The edge $x\rightarrow x''$ is a counit map of the adjunction $F\dashv G$.
\end{enumerate}
\end{remark}

The next lemma shows, that in the setting of \Cref{ucurem} the other pair of unit and counit maps is also related.

\begin{lemma}\label{uculem2}
Let $F:\mathcal{Y}\leftrightarrow\mathcal{X}:G$ be a spherical adjunction of stable $\infty$-categories. Denote the left adjoint of $F$ by $E$. Consider the following fiber and cofiber sequence in $\mathcal{Y}$.
\[ 
 \begin{tikzcd}
  y' \arrow[d]\arrow[dr, phantom, "\square"] \arrow[r] & y \arrow[d] \\
  0 \arrow[r]           & y''          
 \end{tikzcd}
\]
The following two are equivalent.
\begin{enumerate}
\item The edge $y'\rightarrow y$ is a unit map of the adjunction $F\dashv G$. 
\item The edge $y\rightarrow y''$ is counit map of the adjunction $E\dashv F$.
\end{enumerate}
\end{lemma}

\begin{proof}
Denote the right adjoint of $G$ by $H$. By \Cref{ucurem}, we obtain that the edge $y'\rightarrow y$ is unit map of the adjunction $F\dashv G$ if and only if the edge $y\rightarrow y''$ is counit map of the adjunction $G\dashv H$. Denote by $T_\mathcal{X}$ the cotwist functor of the adjunction $F\dashv G$. By \Cref{4pedprop2}, we find equivalences $G\simeq ET_{\mathcal{X}}$ and $H\simeq T_\mathcal{X}^{-1}F$. This implies that any counit map of the adjunction $G\simeq ET_{\mathcal{X}}\dashv H\simeq T_\mathcal{X}^{-1}F$ is also a counit map of the adjunction $E\dashv F$ and vice versa. The equivalence of statements 1 and 2 follows.
\end{proof}

\begin{lemma}\label{uculem3}
Consider an adjunction $\mathcal{M}\rightarrow \Delta^1$ of stable $\infty$-categories, associated to the pair of adjoint functors $F:\mathcal{Y}\leftrightarrow \mathcal{X}:G$. Assume further that $F$ admits a left adjoint $E$. Let $\mathcal{V}=\on{Fun}_{\Delta^1}(\Delta^1,\mathcal{M})$. Denote by ${T}_\mathcal{X}$ the cotwist functor of the adjunction $F\dashv G$. Then the following are equivalent.
\begin{enumerate}
\item For every $x\in \mathcal{X}$ there exists an equivalence $G(x)\simeq E\circ T_\mathcal{X}(x)$ and given a fiber and cofiber sequence in $\mathcal{X}$
\[
\begin{tikzcd} x'\arrow[dr, phantom, "\square"]\arrow[r]\arrow[d]& x \arrow[d]\\ 
  0\arrow[r]& x''
\end{tikzcd}
\] 
the edge $x'\rightarrow x$ is a unit map of the adjunction $E\dashv F$ if and only if the edge $x\rightarrow x''$ is a counit map of the adjunction $F\dashv G$. 
\item The full subcategories $\mathcal{K}_1,\mathcal{K}_2\subset \on{Fun}(\Delta^1\times\Delta^1,\mathcal{V})$ introduced in \Cref{Kconstr} are identical.
\item The adjunction $F\dashv G$ is spherical.
\end{enumerate}
\end{lemma}

\begin{proof}
It is shown in \Cref{Kconstr} that statement 2 implies statement 3. Assume statement 3. By \Cref{4pedprop2} we find $G\simeq E\circ T_{\mathcal{X}}$. Statement 1 thus follows from \Cref{ucurem}.

We now show that statement 1 implies statement 2. We observe that statement 1 directly implies that the cotwist functor $T_\mathcal{X}$ of the adjunction $F\dashv G$ is inverse to the twist functor $T_\mathcal{X}'$ of the adjunction $E\dashv F$. Consider a diagram $\Delta^1\times\Delta^1\rightarrow \mathcal{V}$ corresponding to a vertex in $\mathcal{K}_2$, depicted as follows.
\begin{equation}\label{k2elmdiag}  
 \begin{tikzcd}
  \left(0,T_\mathcal{X}(x)\right)\arrow[r]\arrow[dr, phantom, "\square"]\arrow[d]& \left(G(x),FG(x)\right)\arrow[d]\\
  \left(0,0\right)\arrow[r]& \left(G(x),x\right)
 \end{tikzcd}
\end{equation}
The restriction of \eqref{k2elmdiag} to the edge $T_\mathcal{X}(x)\rightarrow FG(x)$ is by assumption a unit map of the adjunction $E\dashv F$ at $T_\mathcal{X}(x)$. Consider the biCartesian semiorthogonal decomposition $(\mathcal{C},\mathcal{D})$ used in \Cref{Kconstr} and the $(\mathcal{C},\mathcal{D})$-coCartesian edge $e_1:(0,T_\mathcal{X}(x))\rightarrow (ET_\mathcal{X}(x),FET_\mathcal{X}(x))$. This edge corresponds to the following diagram in $\Gamma(F)$
\[
 \begin{tikzcd}
  0 \arrow[d] \arrow[r]  & ET_\mathcal{X}(x) \arrow[d, "!"] \\
  T_\mathcal{X}(x) \arrow[r, "u"] & FET_\mathcal{X}(x)            
 \end{tikzcd}
\] 
where the edge $u$ is a unit map of the adjunction $E\dashv F$. 
The restriction $e_2:(0,T_\mathcal{X}(x))\rightarrow (G(x),FG(x))$ of diagram \eqref{k2elmdiag} corresponds to the following diagram in $\Gamma(F)$.
\[
 \begin{tikzcd}
  0 \arrow[d] \arrow[r]  & G(x) \arrow[d, "!"] \\
  T_\mathcal{X}(x) \arrow[r, "u"] & FG(x)            
 \end{tikzcd}
\]
We use the equivalence 
\[ET_\mathcal{X}(x)\simeq GT_\mathcal{X}T^{-1}_\mathcal{X}(x)\simeq G(x)\]
to extend the diagram corresponding to $e_2$ to the following diagram. 
\[
 \begin{tikzcd}
  0 \arrow[d] \arrow[r]  & G(x) \arrow[d, "!"] \arrow[r, "\simeq"] & ET_\mathcal{X}(x) \arrow[d, "!"] \\
  T_\mathcal{X}(x) \arrow[r, "u"] & FG(x) \arrow[r, "\simeq"]             & FET_\mathcal{X}(x)          
 \end{tikzcd}
\]
The composed edge $T_\mathcal{X}(x)\rightarrow FG(x)\simeq FET_\mathcal{X}(x)$ remains a unit map of the adjunction \mbox{$E\dashv F$}, so that we can produce an equivalence between $e_1$ and $e_2$. This shows that $e_2$ is also $(\mathcal{C},\mathcal{D})$-coCartesian so that using \Cref{mutlem} we obtain that the diagram \eqref{k2elmdiag} lies in $\mathcal{K}_1$. It follows $\mathcal{K}_2\subset \mathcal{K}_1$. The restriction functor $\on{Fun}(\Delta^1\times\Delta^1,\mathcal{V})\rightarrow \on{Fun}(\Delta^{\{0,0\}},\mathcal{V})$ restricts to a trivial fibration $\mathcal{K}_1\rightarrow \mathcal{C}$. This implies that two elements in $\mathcal{K}_1$ are equivalent if and only if their restrictions to $\Delta^{\{0,0\}}$ are equivalent. Using that $T_\mathcal{X}$ is an equivalence, we find for every $(0,x')\in \mathcal{C}$ an element in $\mathcal{K}_2$ whose restriction to $\Delta^{\{0,0\}}$ is given by $(0,x')$. Using that $\mathcal{K}_2$ is closed under equivalences, we obtain $\mathcal{K}_1\subset\mathcal{K}_2$. This implies $\mathcal{K}_1=\mathcal{K}_2$.
\end{proof}

\subsection{The 2/4 property}\label{sec2.3}
In this section we prove the 2/4 property of spherical adjunctions of \cite{AL17} in the setting of $\infty$-categories. A key tool in the proof is the description of the sphericalness of an adjunction is terms of the 4-periodicity of the corresponding biCartesian semiorthogonal decomposition. A further important ingredient in the proof is the relation between the sphericalness of an adjunction and the properties of the appearing unit and counit maps, see \Cref{uculem3}.

\begin{lemma}\label{2/4lem1}
Let $F:\mathcal{Y}\leftrightarrow\mathcal{X}:G$ be an adjunction of stable $\infty$-categories such that $F$ admits a left adjoint $E$. Denote the cotwist functor of $F\dashv G$ by $T_\mathcal{X}$ and the twist functor of $E\dashv F$ by $T_\mathcal{X}'$. The adjunction $F\dashv G$ is spherical if and only if the following two conditions are satisfied.
\begin{enumerate}
\item For any $x\in \mathcal{X}$, consider the edge $\alpha:\Delta^1\rightarrow \mathcal{Y}$ contained in the biCartesian square in $\mathcal{Y}$
\[
 \begin{tikzcd}
  G(x) \arrow[d, "{G(u')}"'] \arrow[r] \arrow[rd, "\square", phantom] & 0 \arrow[d] \\
  GFE(x) \arrow[r, "\alpha"']                       & GT_\mathcal{X}'(x) 
 \end{tikzcd}
\] where $u'$ is the unit map of the adjunction $E\dashv F$. The composition 
\[ E(x)\xrightarrow{u} GFE(x)\xrightarrow{\alpha} GT_\mathcal{X}'(x)\] 
of the unit map $u$ of the adjunction $F\dashv G$ and $\alpha$ is an equivalence.
\item For any $x\in \mathcal{X}$, consider the edge $\beta:\Delta^1\rightarrow \mathcal{Y}$ contained in the biCartesian square in $\mathcal{Y}$ 
\[ 
 \begin{tikzcd}
  ET_\mathcal{X}(x) \arrow[d, "\beta"'] \arrow[r] \arrow[rd, "\square", phantom] & 0 \arrow[d] \\
  EFG(x) \arrow[r, "E(cu)"']                                     & E(x)        
 \end{tikzcd}
\] 
where $cu$ is the counit map $cu$ of the adjunction $F\dashv G$. The composition 
\[ET_\mathcal{X}(x)\xrightarrow{\beta}EFG(x)\xrightarrow{cu'}G(x)\] 
of $\beta$ and the counit map $cu'$ of the adjunction $E\dashv F$ is an equivalence.
\end{enumerate}
\end{lemma}

\begin{proof}
Consider the adjunction $\chi(G)\rightarrow \Delta^1$. We obtain by the \Cref{adjsodlem,adjsodlem2} semiorthogonal decompositions $(\mathcal{A},\mathcal{B}),(\mathcal{B},\mathcal{C}),(\mathcal{C},\mathcal{D}),(\mathcal{D},\mathcal{E})$ of $\mathcal{V}=\on{Fun}_{\Delta^1}(\Delta^1,\chi(G))$. We show that condition 1 is equivalent to $\mathcal{E}\subset \mathcal{A}$. We then present a dual version of the argument below, showing that condition 2 is equivalent to $\mathcal{A}\subset \mathcal{E}$. Together, these two statements imply the Lemma.

Let $x\in \mathcal{X}$ and consider the vertex $(E(x),T_\mathcal{X}'(x))\in \mathcal{E}$ contained in a diagram $D:(\Delta^1)^{\times 2}\rightarrow \mathcal{V}$ of the following form.
\[  
 \begin{tikzcd}
  \left(0,x\right)\arrow[r, "!"]\arrow[dr, phantom, "\square"]\arrow[d]& \left(E(x),FE(x)\right)\arrow[d, "\ast"]\\
  \left(0,0\right)\arrow[r]& \left(E(x),T_\mathcal{X}'(x)\right)
 \end{tikzcd}
\] 
The diagram $D$ corresponds to a vertex in $\mathcal{K}_1$ as introduced in \Cref{Kconstr}. The datum of the functor $D$ is equivalent to the datum of a functor $D_1:\Delta^1\times(\Delta^1)^{\times 2}\rightarrow \chi(G)$. Consider the restriction $D'$ of $D_1$ to $\{1\}\times (\Delta^1)^{\times 2}\rightarrow \chi(G)$. Denote the left Kan extension relative $\chi(G)\rightarrow \Delta^1$ of $D'$ to $\Delta^1\times (\Delta^1)^{\times 2}$ by $D_2$. We can depict the vertices of $D_2$ as follows. 
\[ 
 \begin{tikzcd}
  {(G(x),x)} \arrow[d] \arrow[r] \arrow[rd, "\square", phantom] & {(GFE(x),FE(x))} \arrow[d] \\
  {(0,0)} \arrow[r]                                      & {(GT_\mathcal{X}'(x),T_\mathcal{X}'(x))}              
 \end{tikzcd}
\]
By the properties of the Kan extension, we find a map $D_1\rightarrow D_2$ whose restriction to $\Delta^1\times (\{1\}\times\Delta^1)$ can be depicted as follows.
\begin{equation}\label{diag6'}
 \begin{tikzcd}
  (E(x),FE(x))\arrow[r, dotted]\arrow[d]& (GFE(x),FE(x))\arrow[d]\\
  (E(x),T_\mathcal{X}'(x))\arrow[r, dotted, "\phi"]& (GT_\mathcal{X}'(x),T_\mathcal{X}'(x)) 
 \end{tikzcd}
\end{equation}
The dotted edges represent (part of) the map $D_1\rightarrow D_2$. The edge $\phi$ is an equivalence if and only if $(E(x),T_\mathcal{X}'(x))$ lies in $\mathcal{A}$. The second component of $\phi$ is by construction an equivalence and diagram \eqref{diag6'} shows that the first component of $\phi$ is an equivalence if and only if condition 1 is satisfied for this particular $x\in \mathcal{X}$. This shows that condition 1 is equivalent to $\mathcal{E}\subset \mathcal{A}.$

\Cref{equlem} provides us with an equivalence $\mathcal{V}'\coloneqq\on{Fun}_{\Delta^1}(\Delta^1,\Gamma(E))\simeq \mathcal{V}$. We denote the essential image of the stable subcategories $\mathcal{A},\mathcal{B},\mathcal{C},\mathcal{D},\mathcal{E}\subset \mathcal{V}$ under the equivalence $\mathcal{V}\simeq \mathcal{V}'$ by $\mathcal{A}',\mathcal{B}',\mathcal{C}',\mathcal{D}',\mathcal{E}'$. We note that $\mathcal{A}'\subset \mathcal{E}'$ if and only if $\mathcal{A}\subset \mathcal{E}$. We show that condition 2 is equivalent to $\mathcal{A}'\subset \mathcal{E}'$. We denote the elements of $\mathcal{V}'$ by pairs $(x,y)$ such that $x\in \mathcal{X}$ and $y\in \mathcal{Y}$, suppressing the map $E(x)\rightarrow y$ contained in the data of $(x,y)\in \mathcal{V}'$. The right gluing functor of the biCartesian decomposition $(\mathcal{B},\mathcal{C})$ can be identified with $G$, the right gluing functor of the biCartesian semiorthogonal decomposition $(\mathcal{C},\mathcal{D})$ can thus be identified with $F$. Using $\Gamma(E)\simeq \chi(F)$, we obtain a description of $\mathcal{B}',\mathcal{C}',\mathcal{D}',\mathcal{E}'$ as in \Cref{sodrem}:
\begin{itemize}
\item $\mathcal{B}'$ is spanned by elements of the form $(F(y),y)$ such that the edge $EF(y)\rightarrow y$ is a counit map of the adjunction $E\dashv F$.
\item $\mathcal{C}'$ is spanned by elements of the form $(x,0)$.
\item $\mathcal{D}'$ is spanned by elements of the form $(0,y)$.
\item $\mathcal{E}'$ is spanned by elements $(x,E(x))$, such that the edge $E(x)\rightarrow E(x)$ is an equivalence.
\end{itemize}
Applying \Cref{mutlem} we obtain that $\mathcal{A}'$ is spanned by the vertices $(T_\mathcal{X}(x),G(x))$, fitting into a diagram in $\mathcal{V}'$
\begin{equation}\label{fdiag}
 \begin{tikzcd}
  {(T_\mathcal{X}(x),G(x))} \arrow[d] \arrow[r, "!"] \arrow[rd, "\square", phantom] & {(FG(x),G(x))} \arrow[d, "\ast"] \\
  {(0,0)} \arrow[r]                                                     & {(x,0)}                       
 \end{tikzcd}
\end{equation}
where the top edge is $(\mathcal{A}',\mathcal{B}')$-coCartesian and the right edge is $(\mathcal{B}',\mathcal{C}')$-Cartesian. Consider the biCartesian square 
in $\mathcal{X}$,
\begin{equation}\label{bicartsq1}
 \begin{tikzcd}
  T_\mathcal{X}(x) \arrow[d] \arrow[r] & FG(x) \arrow[d, "cu"] \\
  0 \arrow[r]                          & x                        
 \end{tikzcd}
\end{equation} 
where $cu$ is the counit map of the adjunction $F\dashv G$. We extend \eqref{bicartsq1} via right Kan extension relative $\Gamma(E)\rightarrow \Delta^1$ to the following diagram. 
\begin{equation}\label{rkandiag} 
 \begin{tikzcd}
  {(T_\mathcal{X}(x),ET_\mathcal{X}(x))} \arrow[d] \arrow[r] & {(FG(x),EFE(x))} \arrow[d] \\
  {(0,0)} \arrow[r]                                           & {(x,E(x))}                        
 \end{tikzcd}
\end{equation}
We obtain a map from \eqref{rkandiag} to a diagram of the form \eqref{fdiag}. Restricting the diagram, we obtain the following diagram.
\[
 \begin{tikzcd}
  {(T_\mathcal{X}(x),ET_\mathcal{X}(x))} \arrow[d] \arrow[r, "\phi", dotted] & {(T_\mathcal{X}(x),G(x))} \arrow[d] \\
  {(FG(x),EFE(x))} \arrow[r, dotted]                                & {(FG(x),G(x))}                  
 \end{tikzcd}
\]
We find that $\phi$ is an equivalence if and only if $(T_\mathcal{X},G(x))\in \mathcal{E}'$. The first component of $\phi$ is by construction an equivalence and the second component is an equivalence if and only if condition 2 is satisfied for this $x\in \mathcal{X}$. This shows that $\mathcal{A}'\subset \mathcal{E}'$ if and only if condition 2 holds.
\end{proof}

\begin{lemma}\label{2/4lem2}
Let $F:\mathcal{Y}\leftrightarrow\mathcal{X}:G$ be an adjunction of stable $\infty$-categories such that $F$ admits a left adjoint $E$. Denote the twist and cotwist functors of $F\dashv G$ by $T_\mathcal{Y}$ and $T_\mathcal{X}$, respectively, and the twist and cotwist functors of $E\dashv F$ by $T_\mathcal{X}'$ and $T_\mathcal{Y}'$ respectively. Let $x\in \mathcal{X}$ and consider the following diagrams in $\mathcal{Y}$.
\begin{equation}\label{24sq}
 \begin{tikzcd}
  {y[-1]} \arrow[d] \arrow[r] \arrow[rd, "\square", phantom]  & G(x) \arrow[d, "G(u')"] \arrow[r] \arrow[rd, "\square", phantom]               & 0 \arrow[d]        \\
  E(x) \arrow[r, "u"] \arrow[d] \arrow[rd, "\square", phantom] & GFE(x) \arrow[d, "\alpha'"] \arrow[r, "\alpha"] \arrow[rd, "\square", phantom] & GT_\mathcal{X}'(x) \arrow[d] \\
  0 \arrow[r]                                         & T_\mathcal{Y}E(x) \arrow[r]                                                     & y                
 \end{tikzcd} \quad\quad 
 \begin{tikzcd}
  {y'[-1]} \arrow[d] \arrow[r] \arrow[rd, "\square", phantom]                     & ET_\mathcal{X}(x) \arrow[d, "\beta"] \arrow[r] \arrow[rd, "\square", phantom] & 0 \arrow[d]    \\
  T_\mathcal{Y}'G(x) \arrow[r, "\beta'"] \arrow[d] \arrow[rd, "\square", phantom] & EFG(x) \arrow[d, "cu'"] \arrow[r, "E(cu)"] \arrow[rd, "\square", phantom]     & E(x) \arrow[d] \\
  0 \arrow[r]                                                            & G(x) \arrow[r]                                                       & y'            
 \end{tikzcd}
\end{equation}
The edges $u$ and $cu$ are unit and counit maps of the adjunction $F\dashv G$ and the edges $u'$ and $cu'$ are unit and counit maps of the adjunction $E\dashv F$. Then:
\begin{enumerate}
\item The composition $e_1:G(x)\xrightarrow{G(u')}GFE(x)\xrightarrow{\alpha'}T_\mathcal{Y}E(x)$ is an equivalence if and only if the composition $e_2:E(x)\xrightarrow{u}GFE(x)\xrightarrow{\alpha} GT_\mathcal{X}'(x)$ is an equivalence.
\item The composition $e_3:ET_\mathcal{X}(x)\xrightarrow{\beta}EFG(x)\xrightarrow{cu'} G(x)$ is an equivalence if and only if the composition $e_4:T_\mathcal{Y}'G(x)\xrightarrow{\beta'}EFG(x)\xrightarrow{E(cu)} E(x)$ is an equivalence.
\end{enumerate}
\end{lemma}

\begin{proof}
The pasting law for pullbacks implies that the edge $e_1$ is an equivalence if and only if $y[-1]\simeq 0$ and further that the edge $e_2$ is an equivalence if and only $y\simeq 0$. This shows \mbox{statement 1}. Statement 2 is shown analogously.
\end{proof}

\begin{remark}\label{2/4rem}
In the setting of \Cref{2/4lem2} there exists natural transformations 
\[ \eta_1:G\rightarrow T_{\mathcal{Y}}E\,, \quad \quad \eta_2:E\rightarrow GT_{\mathcal{X}'}\,,\]
\[ \eta_3:ET_{\mathcal{X}}\rightarrow G\,, \quad \quad \eta_4:T_{\mathcal{Y}'}G \rightarrow E\,,\]
such that evaluating $\eta_i$ on $x\in \mathcal{X}$ yields the edge $e_i$.
Using that a natural transformation is an equivalence if and only if it is pointwise an equivalence, we see that \Cref{2/4lem2} implies that $\eta_1$ is an equivalence if and only if $\eta_2$ is an equivalence and that $\eta_3$ is an equivalence if and only if $\eta_4$ is an equivalence.
\end{remark}

\begin{theorem}[The 2/4 property of spherical adjunctions]\label{2/4prop}
Let $F:\mathcal{Y}\leftrightarrow\mathcal{X}:G$ be an adjunction of stable $\infty$-categories and assume that $F$ admits a left adjoint $E$. Denote the twist and cotwist functors of $F\dashv G$ by $T_\mathcal{Y}$ and $T_\mathcal{X}$, respectively, and the twist and cotwist functors of $E\dashv F$ by $T_\mathcal{X}'$ and $T_\mathcal{Y}'$, respectively. The adjunction $F\dashv G$ is spherical if and only if any two of the following four conditions are satisfied.
\begin{enumerate}
\item The twist functor $T_\mathcal{X}'$ and cotwist functor $T_\mathcal{X}$ are inverse equivalences.
\item The twist functor $T_\mathcal{Y}$ and the cotwist functor $T_\mathcal{Y}'$ are inverse equivalences.
\item The natural transformations $\eta_1$ and $\eta_2$ introduced in \Cref{2/4rem} are equivalences.
\item The natural transformations $\eta_3$ and $\eta_4$ introduced in \Cref{2/4rem} are equivalences.
\end{enumerate}
\end{theorem}

\begin{proof}
If $F\dashv G$ is a spherical adjunction, we deduce conditions 1 and 2 by \Cref{ucurem} and \Cref{uculem2} and conditions 3 and 4 by \Cref{2/4lem1,2/4lem2}. Conditions 1 and 2 imply the sphericalness of the adjunction $F\dashv G$ by definition. Conditions 3 and 4 imply the sphericalness of the adjunction $F\dashv G$ by \Cref{2/4lem1}.

Assume conditions 1 and 3. The proof of \Cref{2/4lem1} shows that condition 3 is equivalent to $ \mathcal{E}\subset \mathcal{A}$ with the notation used there. Let $(G(x),x)\in \mathcal{A}$. As discussed in \Cref{Kconstr}, the $\infty$-category $\mathcal{E}$ is spanned by objects $(E(x'),T_\mathcal{X}'(x'))\in \mathcal{V}$ fitting into a diagram  
\[  \begin{tikzcd}
\left(0,x'\right)\arrow[r, "!"]\arrow[dr, phantom, "\square"]\arrow[d]& \left(E(x'),FE(x')\right)\arrow[d, "\ast"]\\
\left(0,0\right)\arrow[r]& \left(E(x'),T_\mathcal{X}'(x')\right)
\end{tikzcd}\]
corresponding to a vertex in $\mathcal{K}_1$. Thus for $x'=T_\mathcal{X}(x)$ we obtain that $(ET_\mathcal{X}(x),T_\mathcal{X}'T_\mathcal{X}(x))\simeq (ET_\mathcal{X}(x),x)\in \mathcal{E}\subset \mathcal{A}$. Using that the restriction functor to the second component is a trivial fibration from $\mathcal{A}$ to $\mathcal{X}$, we obtain that $(G(x),x)\simeq (ET_\mathcal{X}(x),x)\in \mathcal{E}$. Thus $\mathcal{E}=\mathcal{A}$ and the adjunction $F\dashv G$ is spherical by \Cref{4pedprop1}.

Showing that the conditions 1 and 4 imply the sphericalness of the adjunction $F\dashv G$ is analogous to showing that the conditions 1 and 3 imply the sphericalness.

Assume conditions 2 and 3. The proof of \Cref{2/4lem1} shows that condition 3 is equivalent to $ \mathcal{E}\subset \mathcal{A}$ with the notation used there. Using that $T_\mathcal{Y}$ is an equivalence, we can compose the adjunctions $E\dashv F$ and $T_\mathcal{Y}\dashv T_\mathcal{Y}^{-1}$ to obtain an adjunction $G\simeq T_\mathcal{Y}E\dashv FT_\mathcal{Y}'=:H$. We want to show that statement 1 of \Cref{uculem3} is satisfied for the adjunction $G\dashv H$ to deduce that $G\dashv H$ is spherical. \Cref{lrcor} then implies that $F\dashv G$ is also spherical. The counit of the adjunction $G\dashv H$ is equivalent to the counit of the adjunction $E\dashv F$, so that the cotwist functors of the two adjunctions are equivalent. This shows that the first part of statement 1 holds. \Cref{adjsodlem2} implies that $(\mathcal{A}^\perp,\mathcal{A})$ forms a semiorthogonal decomposition of $\mathcal{V}$. The condition $\mathcal{E}\subset \mathcal{A}$ implies $\mathcal{A}^\perp\subset \mathcal{E}^\perp=\mathcal{D}$. Consider any diagram $\Delta^2\times\Delta^2\rightarrow \mathcal{V}$ of the following form,
\begin{equation}\label{22diag}
 \begin{tikzcd}
  z \arrow[r]\arrow[dr,phantom,"\square"]\arrow[d] & a\arrow[dr,phantom,"\square"] \arrow[d, "!"] \arrow[r]     & 0 \arrow[d]  \\
  b'\arrow[dr,phantom,"\square"] \arrow[d] \arrow[r] & b\arrow[dr,phantom,"\square"] \arrow[d] \arrow[r, "\ast"] & c \arrow[d] \\
  0 \arrow[r]           & b'' \arrow[r]                  & z'          
 \end{tikzcd}
\end{equation}
satisfying
\begin{enumerate}
\item[1)] $a\in \mathcal{A}$, $b,b',b''\in \mathcal{B}$ and $c\in \mathcal{C}$,
\item[2)] the edge $a\rightarrow b$ is $(\mathcal{A},\mathcal{B})$-coCartesian,
\item[3)] the edge $b\rightarrow c$ is $(\mathcal{B},\mathcal{C})$-Cartesian,
\item[4)] the edge $a\rightarrow b''$ is $(\mathcal{A},\mathcal{B})$-Cartesian.
\end{enumerate}
Condition 4) implies by \Cref{mutlem} that $z\in \prescript{\perp}{}{\mathcal{A}} \subset \mathcal{D}$. Thus $z'\simeq z[1]\in \mathcal{D}$. By \Cref{mutlem}, we obtain that the edge $b'\rightarrow c$ is  $(\mathcal{B},\mathcal{C})$-coCartesian. We define two $\infty$-categories of diagrams in $\mathcal{V}$. Let $\mathcal{W}$ be the full subcategory of $\on{Fun}(\Delta^2\times\Delta^2,\mathcal{V})$ spanned by diagrams of the form \eqref{22diag} satisfying 1) to 4) and let $\mathcal{W}'$ be the full subcategory of $\on{Fun}(\Delta^1\times \Delta^1,\mathcal{V})$ spanned by functors of the form \eqref{22diag} satisfying $1),2),3)$ and that the edge $b'\rightarrow c$ is $(\mathcal{B},\mathcal{C})$-coCartesian.  We have shown that $\mathcal{W}\subset \mathcal{W}'$. Consider the restriction functors $\on{res}_{b'},\on{res}_{b''}:\mathcal{W}\rightarrow \mathcal{B}$ to the vertices $b'$ and $b''$. By the properties of the involved Cartesian and coCartesian edges and the biCartesian squares, we obtain that $\on{res}_{b''}$ is a trivial fibration. Choosing a section of $\on{res}_{b''}$ and composing with $\on{res}_{b'}$ constitutes a choice of cotwist functor $T'_\mathcal{Y}$, which is an equivalence by assumption. Thus $\on{res}_{b'}$ is also a trivial fibration. We observe that the restriction functor to $b'$ is also a trivial fibration from $\mathcal{W}'$ to $\mathcal{B}$. Using that $\mathcal{W}$ and $\mathcal{W}'$ are closed under equivalences in $\on{Fun}(\Delta^1\times\Delta^1,\mathcal{V})$ it follows that $\mathcal{W}=\mathcal{W}'$. We have shown that the second part of statement 1 of \Cref{uculem3} is fulfilled and deduce that the adjunction $F\dashv G$ is spherical.

Showing that the conditions 2 and 4 imply the sphericalness of the adjunction $F\dashv G$ is analogous to showing that the conditions 2 and 3 imply the sphericalness.
\end{proof}

\begin{corollary}\label{rescor}
Let $F:\mathcal{D}\leftrightarrow \mathcal{C}:G$ be an adjunction of stable $\infty$-categories. Assume that $F$ admits a left adjoint $E$. Let $\mathcal{D}'\subset \mathcal{D}$ be a stable subcategory such that $\on{Im}(E),\on{Im}(G)\subset \mathcal{D}'$. Then:
\begin{enumerate}
\item The adjunctions $E\dashv F$ and $F\dashv G$ restrict to adjunctions $E:\mathcal{C}\leftrightarrow \mathcal{D}':F'$ and \mbox{$F':\mathcal{D}'\leftrightarrow \mathcal{C}:G$,} where $F'$ is the restriction of the functor $F$.
\item The adjunction $F\dashv G$ is spherical if and only if the adjunction $F'\dashv G$ is spherical.
\end{enumerate}
\end{corollary}

\begin{proof}
Statement 1 follows directly from the assumptions $\on{Im}(E),\on{Im}(G)\subset \mathcal{D}'$. 

We now show statement 2. The units of the adjunctions $E\dashv F$ and $E\dashv F'$ are equivalent. Similarly, the counits of the adjunctions $F\dashv G$ and $F'\dashv G$ are also equivalent. The counits of the adjunction $E\dashv F$ restrict to the counit of the adjunction $E\dashv F'$ and similarly the unit of the adjunction $F\dashv G$ restricts to the unit of the adjunction $F'\dashv G$. We thus obtain that the natural transformation $\eta_1$ of \Cref{2/4rem} associated to the adjunctions $E\dashv F\dashv G$ is equivalent to the natural transformation $\eta_1'$ associated to the adjunctions $E\dashv F'\dashv G$. Similarly, the natural transformation $\eta_3$ of \Cref{2/4rem} associated to the adjunctions $E\dashv F\dashv G$ is equivalent to the natural transformation $\eta_3'$ associated to the adjunctions $E\dashv F'\dashv G$. Thus $\eta_1$ and $\eta_3$ are equivalences if and only if  $\eta_1'$ and $\eta_3'$ are equivalences. \Cref{2/4prop} implies that the adjunction $F\dashv G$ is spherical if and only if the adjunction $F'\dashv G$ is spherical.
\end{proof}

\section{Local systems on spheres}\label{sec4} 
Let $\mathcal{D}$ be a stable $\infty$-category. We denote by $S^n$ the singular set of the topological $n$-sphere. The pullback functor $f^*$ along the map $f:S^n\rightarrow \ast$ is part of an adjunction 
\begin{equation}\label{sphadj1} 
f^*:\mathcal{D}\longleftrightarrow \on{Fun}(S^n,\mathcal{D}):f_* 
\end{equation}
between $\mathcal{D}$ and the stable $\infty$-category of local systems on the $n$-sphere with values in $\mathcal{D}$. We show in \Cref{sec3.1} that the adjunction \eqref{sphadj1} is spherical. Generalizing the spherical adjunction \eqref{sphadj1} to a relative context, we show in \Cref{sphfibtwist} that for any spherical fibration $f:X\rightarrow Y$ between Kan complexes there is a spherical adjunction \mbox{$f^*:\on{Fun}(Y,\mathcal{D})\longleftrightarrow \on{Fun}(X,\mathcal{D}):f_*.$} If we further assume  that $\mathcal{D}=\on{Sp}$ is the stable $\infty$-category of spectra, or any other good symmetric monoidal and stable $\infty$-category, we can endow the $\infty$-categories $\on{Fun}(X,\mathcal{D})$ and $\on{Fun}(Y,\mathcal{D})$ with a symmetric monoidal structure, called the pointwise monoidal structure. We provide an explicit description of the resulting twist functor $T_{\on{Fun}(Y,\mathcal{D})}:\on{Fun}(Y,\mathcal{D})\rightarrow \on{Fun}(Y,\mathcal{D})$ in terms of the symmetric monoidal product with a local system $\zeta\in \on{Fun}(Y,\mathcal{D})$ in \Cref{sec3.3}. The value of $\zeta$ at a point $y\in Y$ can be interpreted as the reduced homology the fiber of $f$ over $y$.

\subsection{Twist along a sphere}\label{sec3.1}

We fix a stable $\infty$-category $\mathcal{D}$. Given a simplicial set $Z$, we denote by $Z^\triangleright$ the join $Z\ast \Delta^0$. We recursively define $\infty$-categories by
\[ P_0\coloneqq S^0=\Delta^0 \amalg \Delta^0\,,\]
\[ P_n\coloneqq P_{n-1}^\triangleright \coprod_{P_{n-1}} P_{n-1}^\triangleright\,.\]
We define recursively a labeling of the vertices of $P_n$ by denoting the vertices which are not contained in $P_{n-1}$ by $2n+1$ and $2n+2$. Denote by $g^*:\mathcal{D}\rightarrow \on{Fun}(P_n,\mathcal{D})$ the pullback functor along the map of simplicial sets $g:P_n\rightarrow \ast$. The simplicial set $P_n$ is finite and $\mathcal{D}$ admits as a stable $\infty$-category all finite limits and colimits. By \cite[4.3.3.7]{HTT}, we hence obtain adjunctions 
\begin{equation}\label{adj1}
g^*: \mathcal{D}\longleftrightarrow \on{Fun}(P_n,\mathcal{D}): g_*=\lim\,,
\end{equation}
\begin{equation*}
g_!=\on{colim}: \text{Fun}(P_n,\mathcal{D})\longleftrightarrow \mathcal{D}: g^*\,,
\end{equation*}
where $g_!$ and $g_*$ are the limit and colimit functors, respectively. We will show in this section that the adjunction $g^*\dashv g_*$ is spherical.

The $\infty$-category $P_n$ is equivalent to the $\infty$-category of exit paths of a stratification of the $n$-sphere, as defined in \cite[A.6.2]{HA}, where the strata are nested spheres, one in each dimension $d\leq n$. For example for $n=2$, this corresponds to the $2$-sphere, stratified by a circle and two points on that circle. The functor category $\on{Fun}(P_n,\mathcal{D})$ is equivalent to the $\infty$-category of constructible sheaves with values in $\mathcal{D}$ with respect to the stratification of $S^n$ if $\mathcal{D}$ is a compactly generated $\infty$-category, see \cite[Section 8.6]{Tan19}.

Let $f:S^n\rightarrow \ast$ and let $f^*:\mathcal{D}\rightarrow \on{Fun}(S^n,\mathcal{D})$ be the corresponding pullback functor. If we can show that $\mathcal{D}$ admits $S^n$-indexed limits and colimits, we can again apply \cite[4.3.3.7]{HTT} to show that $f^*$ admits left and right adjoints. This follows from the following Lemma.

\begin{lemma}\label{cofinallemma}
There exists a final and cofinal map $P_n \rightarrow S^n$.
\end{lemma}

\begin{proof}
Denote the topological $n$-sphere by $S^n_{top}$, so that $S^n=\on{Sing}(S^n_{top})$.
Using the Quillen equivalence between the Kan model structure on $\on{Set}_\Delta$ and the Quillen model structure on $\on{Top}$, we can obtain from a choice of weak homotopy equivalence $|P_n|\simeq S^n_{top}$ a weak homotopy equivalence $e:P_n \rightarrow S^n$. The finality and cofinality thus follow from \cite[4.1.2.6]{HTT}.
\end{proof}

We thus conclude that $f^*$ admits left and right adjoints $f_!$ and $f_*$, given by mapping a functor $S^n\rightarrow \mathcal{D}$ to its colimits, respectively, its limit. The next Lemma shows that $\on{Fun}(S^n,\mathcal{D})$ is a full subcategory of $\on{Fun}(P_n,\mathcal{D})$ and that the adjunction $g^*\dashv g_*$ restricts to the adjunction $f^*\dashv f_*$.

\begin{lemma}\label{lem:fgres}
There exists a fully faithful functor $\on{Fun}(S^n,\mathcal{D})\rightarrow \on{Fun}(P_n,\mathcal{D})$ making the following diagrams commute.
\begin{equation}\label{eq:fgcomm}
\begin{tikzcd}
{\on{Fun}(S^n,\mathcal{D})} \arrow[rr] &                                                  & {\on{Fun}(P_n,\mathcal{D})} \\
                                       & \mathcal{D} \arrow[lu, "f^*"] \arrow[ru, "g^*"'] &                            
\end{tikzcd}
\quad\quad
\begin{tikzcd}
{\on{Fun}(S^n,\mathcal{D})} \arrow[rr] \arrow[rd, "f_*"'] &             & {\on{Fun}(P_n,\mathcal{D})} \arrow[ld, "g_*"] \\
                                                          & \mathcal{D} &                                              
\end{tikzcd}
\end{equation}
\end{lemma}

\begin{proof}
Suppose that $n=0$. Then $S^0=P_0=\Delta^0 \amalg \Delta^0$ and the assertion is clear. We proceed by induction on $n$. 

Suppose that there is a map $P_n\rightarrow S^n$ whose pullback functor $\on{Fun}(S^n,\mathcal{D})\rightarrow \on{Fun}(P_n,\mathcal{D})$ is fully faithful and makes the diagrams \eqref{eq:fgcomm} commute. Let $a:P_n\rightarrow P_n^\triangleright$ be the inclusion and $b:P_n^\triangleright\rightarrow \ast$. The pullback functors assemble into a commutative diagram 
\[
\begin{tikzcd}
\mathcal{D} \arrow[r, "b^*"] \arrow[rd, "g^*"] \arrow[d, "f^*"] & {\on{Fun}(P_n^\triangleright,\mathcal{D})} \arrow[d, "a^*"] \\
{\on{Fun}(S^n,\mathcal{D})} \arrow[r]                           & {\on{Fun}(P_n,\mathcal{D})}                                
\end{tikzcd}
\]
where the upper triangle commutes because $b\circ a=g$ and the lower triangle commutes by the induction assumption. The pushout diagrams of $S^{n+1}\simeq \Delta^0 \amalg_{S^n}\Delta^0$ and $P^{n+1}\simeq P_n^\triangleright\amalg_{P_n}P_n^\triangleright$ are mapped by $\on{Fun}(\mhyphen,\mathcal{D})$ to pullback diagrams in $\on{Cat}_\infty$, which assemble into the following commutative diagram.
\[
\begin{tikzcd}[row sep=small]
{\on{Fun}(S^{n+1},\mathcal{D})} \arrow[rr, hook, "e^*"] \arrow[dd] \arrow[rd] \arrow[rddd, "\lrcorner", phantom] &                                                 & {\on{Fun}(P_{n+1},\mathcal{D})} \arrow[dd] \arrow[rd] \arrow[rddd, "\lrcorner", phantom] &                                                              \\
                                                                                                          & \mathcal{D} \arrow[rr, "b^*", near start] \arrow[dd, "f^*"', near start] &                                                                                          & {\on{Fun}(P_n^\triangleright,\mathcal{D})} \arrow[dd, "a^*"] \\
\mathcal{D} \arrow[rd, "f^*", near end] \arrow[rr, "b^*", near end]                                                           &                                                 & {\on{Fun}(P_n^\triangleright,\mathcal{D})} \arrow[rd, "a^*"]                             &                                                              \\
                                                                                                          & {\on{Fun}(S^n,\mathcal{D})} \arrow[rr, hook]    &                                                                                          & {\on{Fun}(P_n,\mathcal{D})}                                 
\end{tikzcd}
\]
Since $P_{n}^\triangleright$ is contractible, the functor $b^*$ is an equivalence onto its image given by (up to equivalence) constant functors $P_{n}^\triangleright\rightarrow  \mathcal{D}$. It follows that the resulting functor $e^*:\on{Fun}(S^{n+1},\mathcal{D})\rightarrow \on{Fun}(P_{n+1},\mathcal{D})$ is also an equivalence onto its image. The functor $e^*$ is given by the pullback functor along a weak equivalence $e:P_{n+1}\rightarrow S_{n+1}$, which is as in \Cref{cofinallemma} final. Since $e\circ f=g$, it follows that $g^*=e^*\circ f^*$, meaning that the left diagram of \eqref{eq:fgcomm} commutes. The commutativity of the right diagram of \eqref{eq:fgcomm} follows from $e$ being final.
\end{proof}

If we show that the adjunction $g^*\dashv g_*$ is spherical, it follows by \Cref{rescor}, that the adjunction $f^*\dashv f_*$ is also spherical. The advantage of treating the adjunction $g^*\dashv g_*$ over treating the adjunction $f^*\dashv f_*$ is, that the former is more accessible to direct computations because $P_n$ is a finite poset.

To show that the adjunction $g^*\dashv g_*$ is spherical, we use the conditions 3 and 4 of  the 2/4 property of spherical adjunctions. We begin in \Cref{sphlem1} by describing the limits and colimits of constant local systems and the arising unit and counit maps. The idea of the proof is to use the decomposition results for limits and colimits, see \cite[4.2.3.10]{HTT}.

\begin{lemma}\label{sphlem1}
Let $n\geq 0$ and consider the functor $g^*:\mathcal{D}\rightarrow \on{Fun}(P_n,\mathcal{D})$ with left and right adjoints $g_!$ and $g_*$.
\begin{enumerate}
\item There exists an equivalence of functors $g_!g^*\simeq id_{\mathcal{D}}\oplus id_{\mathcal{D}}[n]$ in $\on{Fun}(\mathcal{D},\mathcal{D})$ such that composition with the counit map $g_!g^*\rightarrow id_\mathcal{D}$ of the adjunction $g_!\dashv g^*$ yields the projection from the direct sum.
\item There exists an equivalence of functors $g_*g^*\simeq id_{\mathcal{D}}\oplus id_{\mathcal{D}}[-n]$ in $\on{Fun}(\mathcal{D},\mathcal{D})$ such that composition with the unit map $id_\mathcal{D}\rightarrow g_*g^*$ of the adjunction $g^*\dashv g_*$ yields the inclusion into the direct sum.
\end{enumerate}
\end{lemma}

\begin{proof}
We prove only statement 1, statement 2 can be shown analogously. Assume $n=0$. It is immediate that $g_!g^*\simeq id_\mathcal{D}\oplus id_{\mathcal{D}}$. The counit map is given by $id_{\mathcal{D}}\oplus id_{\mathcal{D}}\xrightarrow{(id,id)}id_{\mathcal{D}}$ which is equivalent to the projection, using the equivalence $id_{\mathcal{D}}\oplus id_{\mathcal{D}} \xrightarrow{A} id_{\mathcal{D}}\oplus id_{\mathcal{D}}$ where $A=\begin{pmatrix}id & -id \\ 0 & id \end{pmatrix}$ with inverse  $A^{-1}=\begin{pmatrix}id & id \\ 0 & id \end{pmatrix}$. 

We introduce some notation: For $m>0$ we denote $[m]=\{1,\dots,m\}$. Given $X\subset [2n+2]$, we denote the full subcategory of $P_n$ generated by the elements of $X$ by $\langle X\rangle$. We label the following inclusion functors $i_{n-1}:P_{n-1}\rightarrow P_n$, $j_{n-1}:\langle [2n+1]\rangle\rightarrow P_n$ and $k_{n-1}:\langle [2n]\cup \{2n+2\}\rangle\rightarrow P_n$. We denote by $i^*_{n-1},j^*_{n-1},k^*_{n-1}$ the respective pullback functors. We denote by $\on{colim}$ the colimit functors, irrespective of their domain.

For $n>0$, we continue by induction. We show that there exists an equivalence of functors $g_!g^*\simeq id_{\mathcal{D}}\oplus id_{\mathcal{D}}[n]$, such that the following diagram commutes.
\[
\begin{tikzcd}
g_!g^* \arrow[r, "\simeq"] \arrow[rd, "cu_n"'] & {id_\mathcal{D}\oplus id_\mathcal{D}[n]} \arrow[d, "{(id,0)}"] \\
                                              & id_\mathcal{D}                                                
\end{tikzcd}
\]
 
Assume that the above statement has been shown for $n-1$. Via Kan extension we produce the following diagram in $\on{Fun}(\mathcal{D},\mathcal{D})$.
\begin{equation}\label{fdc}
\begin{tikzcd}
{\on{colim}\circ\, i_{n-1}^*\circ g^*} \arrow[r] \arrow[d] & {\on{colim}\circ\, k^*_{n-1}\circ g^*} \arrow[d] \\
{\on{colim}\circ\, j_{n-1}^* \circ g^*} \arrow[r, "b"]          & g_! g^*                             
\end{tikzcd}
\end{equation}
The top horizontal and left vertical maps can be identified with the counit of the adjunction $\mathcal{D}\leftrightarrow \on{Fun}(P_{n-1},\mathcal{D})$.  The induction assumption implies that the diagram \eqref{fdc} is equivalent to the following diagram in $\on{Fun}(\mathcal{D},\mathcal{D})$.
\begin{equation}\label{sq4}
\begin{tikzcd}
{id_{\mathcal{D}}\oplus id_{\mathcal{D}}[n-1]} \arrow[r, "{(id\text{,}0)}"] \arrow[d, "{(id\text{,}0)}"'] & id_{\mathcal{D}} \arrow[d]  \\
id_{\mathcal{D}} \arrow[r]                                                                                                           & g_!g^*
\end{tikzcd}
\end{equation} 
Evaluating the diagram \eqref{fdc} at a vertex $d\in \mathcal{D}$ yields a diagram equivalent to the pushout diagram obtained from the decomposition of the colimit of the functor $g^*(d)\in \on{Fun}(P_n,\mathcal{D})$ along the decomposition
\[ P_{n-1}\subset \langle [2n+1]\rangle,\langle [2n]\cup \{2n+2\}\rangle\subset P_n\,.\]
This implies that the diagrams \eqref{fdc} and \eqref{sq4} are biCartesian squares. We can thus find an equivalence of diagrams between the diagram \eqref{sq4} and the following diagram.
\[\begin{tikzcd}
id_{\mathcal{D}}\oplus id_{\mathcal{D}}[n-1]\arrow[dr, phantom, "\square"]\arrow[r, "(id\text{,}0)"]\arrow[d, "(id\text{,}0)"]& id_{\mathcal{D}}\arrow[d, "{(id,0)}"]\\
id_{\mathcal{D}}\arrow[r, "{(id,0)}"]& id_{\mathcal{D}}\oplus id_{\mathcal{D}}[n]
\end{tikzcd}\] 

Using that $\langle [2n+1]\rangle$ is contractible, it follows that the composite map $\on{colim}\circ\, j^{n-1}\circ g^*\xrightarrow{b} g_!g^*\xrightarrow{cu_n} id_\mathcal{D}$ is an equivalence. There thus exists a commutative diagram of the form 
\[
\begin{tikzcd}
g_!g^* \arrow[r, "\simeq"] \arrow[rd, "cu_n"'] & {id_\mathcal{D}\oplus id_\mathcal{D}[n]} \arrow[d, "{(c,c')}"] \arrow[r, "A"] & {id_\mathcal{D}\oplus id_\mathcal{D}[n]} \arrow[d, "{(id,0)}"] \\
                                              & id_\mathcal{D} \arrow[r, "id"]                                                 & id_\mathcal{D}                                                
\end{tikzcd}
\]
with $c,c'$ some undetermined maps and $A= \begin{pmatrix}c^{-1} & c' \\ 0 & id \end{pmatrix}$ an equivalence, completing the induction.
\end{proof}

\begin{remark}\label{sphrem1}
\Cref{sphlem1} implies that the twist functor $T_\mathcal{D}$ of the adjunction $g^*\dashv g_*$ satisfies $T_\mathcal{D}\simeq [-n]$. The twist functor of the adjunction $f^*\dashv f_*$ is equivalent to the twist functor of the adjunction $g^*\dashv g_*$ and thus also equivalent to $[-n]$.
\end{remark}

\begin{lemma}\label{sphlem2}
Every functor in $\on{Fun}(P_n,\mathcal{D})$ is a pullback of functors whose value is zero on $2n+1\in (P_n)_0$ or $2n+2\in (P_n)_0$.
\end{lemma}

\begin{proof}
Let $X\in\on{Fun}(P_n,\mathcal{D})$. Denote by $X_{n+1},X_{n+2},X_{n+1,n+2}\in \on{Fun}(P_n,\mathcal{D})$ the functors which are identical to $X$ except for their value at $2n+1$, $2n+2$ and $\{2n+1,2n+2\}$, respectively, where their value is $0$. The functors $X_{2n+1},X_{2n+2},X_{2n+1,2n+2}$ can be described as right Kan extensions of restrictions of $X$. The functor $X$ is equivalent to a pullback $X_{2n+1}\times_{X_{2n+1,2n+2}} X_{2n+2}$ in $\on{Fun}(P_n,\mathcal{D})$. 
\end{proof}

\begin{lemma}\label{sphlem3} 
Let $F\in \on{Fun}(P_n,\mathcal{D})$ and consider the counit map $cu:g^*g_*(F)\rightarrow F$ of the adjunction $g^*\dashv g_*$. There exists an equivalence $g_!(F)\oplus g_*(F)\simeq g_!g^*g_*(F)$ such that the composite with $g_!(cu):g_!g^*g_*(F)\rightarrow g_!(F)$ yields the projection from the direct sum.
\end{lemma}

\begin{proof}
By \Cref{sphlem2} it suffices to consider the case where $F$ is zero on $2n+1$ or $2n+2$. The statement for a general functor then follows using that $g_*$, $g^*$ and $g_!$ are exact. We restrict in the following to the case where $F(2n+2)=0$. The case $F(2n+1)=0$ is completely analogous.

For $n=0$, we consider a functor $F\in\on{Fun}(P_0,\mathcal{D})$ with $F(2)=0$. We find equivalences $g_!(F)\simeq g_*(F)\simeq F(1)$. This implies $g_!g^*g_*(F)\simeq g_*(F)\oplus g_*(F)\simeq F(1)\oplus F(1)$. We decompose the colimit of the diagram $P_0\rightarrow \on{Fun}(\Delta^1,\mathcal{D})$ corresponding to the natural transformation of functors $g^*g_*(F)\xrightarrow{cu} F$ along the decomposition $P_0=\ast\amalg \ast$ to obtain the following commutative diagram in $\mathcal{D}$.
\[
\begin{tikzcd}[column sep=small]
F(1) \arrow[rd, "{(id,0)}"] \arrow[rrrrrr, "id", dashed, bend left=15] &                                                    & F(1) \arrow[ld, "{(0,id)}"'] \arrow[rrrrrr, dashed, bend left=15] &  & ~~~ &  & F(1) \arrow[rd, "id"] &      & 0 \arrow[ld] \\
                                                                       & F(1)\oplus F(1) \arrow[rrrrrr, "{(id,0)}", dashed] &                                                                  &  & ~~~ &  &                       & F(1) &             
\end{tikzcd}
\]
This shows that the map $F(1)\oplus F(1)\simeq g_!g^*g_*(F)\xrightarrow{g_!(cu)} g_*(F) \simeq F(1)$ is equivalent to the projection onto the first direct summand. 

For $n\geq 1$ we proceed by induction. Fix $n$ and assume the lemma has been shown for all functors in $\on{Fun}(P_{n-1},\mathcal{D})$. Consider $F\in \on{Fun}(P_n,\mathcal{D})$ with the property, that $F(2n+2)=0$. We use the notation introduced in the proof of \Cref{sphlem1} in the following. We apply the decomposition of colimits with the decomposition
\[ P_{n-1}\subset \langle[2n+1]\rangle,\langle[2n]\cup\{2n+2\}\rangle \subset P_n\]
to the diagram $P_n\rightarrow \on{Fun}(\Delta^1,\mathcal{D})$ corresponding to the natural transformation $g^*g_*(F)\xrightarrow{cu} F$. The resulting diagram in $\mathcal{D}$ is up to equivalence of the following form.
\begin{equation}\label{diag4}
\begin{tikzcd}[column sep=small, row sep=small]
                               & {g_*(F)[n-1]\oplus g_*(F)} \arrow[rrrrr, "e_{n-1}", dotted] \arrow[ld, "{(0,id)}"'] \arrow[rd, "{(0,id)}"] \arrow[dd, "\square", phantom] &                               &  &  &                    & \on{colim}F|_{P_{n-1}} \arrow[dd, "\square", phantom] \arrow[ld] \arrow[rd] &              \\
g_*(F) \arrow[rd, "{(0,id)}"'] &                                                                                                                                           & g_*(F) \arrow[ld, "{(0,id)}"] &  &  & F(2n+1) \arrow[rd] &                                                                                       & 0 \arrow[ld] \\
                               & {g_*(F)[n]\oplus g_*(F)} \arrow[rrrrr, "{e_n=(h_0,h_1)}", dotted]                                                                         &                               &  &  &                    & \on{colim} F                                                                          &             
\end{tikzcd}
\end{equation} 
We have not depicted all edges. If $n=1$, the edge $e_{n-1}$ is given by $g_*(F)\oplus g_*(F)\xrightarrow{(a,a)}F(1)\simeq \on{colim}F|_{P_0}$, where $a:g_*(F)\rightarrow F(1)$ is the map contained in the limit cone of $F$. It follows that $h_1$ is zero. If $n\geq 2$, the edge $e_{n-1}$ factors as 
\[g_*(F)[n-1]\oplus g_*(F) \rightarrow \lim F|_{P_{n-1}}[n-1]\oplus \lim F|_{P_{n-1}}\xrightarrow{(h_0',h_1')} \on{colim} F|_{P_{n-1}}\,,\] 
where by the induction $h_0'$ is an equivalence and $h_1'$ is zero. It follows that $e_{n-1}$ restricted to $g_*(F)$ is zero and thus that $h_1$ is also zero. To complete the induction step, we need to show that for any $n\geq 1$ the map $h_0$ is an equivalence. For that it suffices to show that the following square contained in the diagram \eqref{diag4} is biCartesian. 
\begin{equation}\label{sq1}
\begin{tikzcd}
{g_*(F)[n-1]\oplus g_*(F)} \arrow[d, "{(0,id)}"] \arrow[r, "e_{n-1}"] & \on{colim} F|_{P_{n-1}} \arrow[d] \\
g_*(F) \arrow[r]                                                       & F(2n+1)                                
\end{tikzcd}
\end{equation}

The decomposition of colimits organizes the colimit cones of the restrictions of the functors $g^*g_*(F)$ and $F$ into a diagram 
\[D:Z\coloneqq P_{n-1}^\triangleright\times \Delta^1\times \Delta^1 \coprod_{P_{n-1}^\triangleright\times \Delta^1\times\{1\}} \langle [2n+1]\rangle^\triangleright\times \Delta^1\times\{1\}\rightarrow \mathcal{D}\,,\] 
i.e.~$D$ restricts to the colimit cones of the functors $g^*g_*(F)|_{P_{n-1}},F|_{P_{n-1}},g^*g_*(F)|_{\langle [2n+1]\rangle}\text{ and }F|_{\langle [2n+1]\rangle}$ on 
\begin{align}\label{comps} 
\begin{split}
& P_{n-1}^\triangleright\times \{0\}\times \{0\},~P^\triangleright_{n-1}\times \{1\}\times\{0\},~\langle [2n+1]\rangle^\triangleright\times \{0\}\times\{1\}\\
\text{and }& \langle [2n+1]\rangle^\triangleright\times \{1\}\times\{1\}\,,
\end{split}
\end{align}
respectively. Furthermore, the restriction of $D$ to the 'tips of the cones', i.e.~the simplicial subset $\Delta^1\times\Delta^1\hookrightarrow Z$ mapping to the joined $0$-simplicies, yields the diagram \eqref{diag4}. By the involved universal properties, the diagram $D$ is determined up to equivalence by its restriction to 
\begin{equation}\label{eq29} 
P_{n-1}\times \Delta^1\times \Delta^1 \coprod_{P_{n-1}\times \Delta^1\times\{1\}}\langle [2n+1]\rangle\times \Delta^1\times\{1\}
\end{equation}
and the restriction to each component in \eqref{comps}. Note that \eqref{eq29} is a left Kan extension of its restriction to 
\begin{equation}\label{eq30}
P_{n-1}\times \left(\Delta^{\{(0,0),(0,1)\}}\amalg_{\Delta^{\{(0,0)\}}}\Delta^{\{(0,0),(1,0)\}}\right)
\end{equation} and that each restriction to a component in \eqref{comps} is a left Kan extension of its restriction to the complement of the tip of the colimit cone. Using that left Kan extensions commute with each other, it follows that $D$ is a left Kan extension of its restriction to \eqref{eq29} and that \eqref{sq1} is a pushout square and thus biCartesian.
\end{proof}

\begin{remark}\label{dualrem} 
There exists an equivalence $P_n\simeq P^{op}_n$ mapping $i$ to $2n+3-i$. Replacing $\mathcal{D}$ with $\mathcal{D}^{op}$ in the proof of \Cref{sphlem2} thus implies the following: Let $F\in \on{Fun}(P_n,\mathcal{D})$ and consider the edge $g_*(F)\xrightarrow{g_*(u)}g_*g^*g_!(F)$ obtained from applying $g_*$ to the unit map $u:F\rightarrow g^*g_!(F)$ of the adjunction $g_!\dashv g^*$. There exists an equivalence $g_*g^*g_!(F)\simeq g_!(F)\oplus g_*(F)$ such that precomposition with $g_*(u)$ yields up to equivalence the inclusion of $g_*(F)$ into the direct sum.
\end{remark}

\begin{remark}\label{sphrem2}
We note that in the construction of the equivalences in the \Cref{sphlem1,sphlem2} and \Cref{dualrem} the identical decompositions of $P_n$ were used, so that we obtain a compatibility statement between the equivalences as follows. The composition of the equivalences 
\[ g_!(F)\oplus g_!(F)[n]\simeq g_*g^*g_!(F)\simeq g_!(F)\oplus g_*(F)\] 
from the \Cref{sphlem1,sphlem2} restrict to equivalences $g_!(F)\simeq g_!(F)$ and $g_!(F)[n]\simeq g_*(F)$ on the first and second factors, respectively.
\end{remark}

\begin{proposition}\label{sphprop}
Let $\mathcal{D}$ be a stable $\infty$-category and $n\geq 0$. Consider the pullback functor $g^*:\mathcal{D}\rightarrow \on{Fun}(P_n,\mathcal{D})$ with right adjoint $g_*$. The adjunction $g^*\dashv g_*$ is spherical.
\end{proposition}

\begin{proof}
We apply \Cref{2/4lem1} to show the sphericalness. We denote the cotwist functor of $g^*\dashv g_*$ by $T_\mathcal{X}$ and the twist functor of $g_!\dashv g^*$ by $T_\mathcal{X}'$. Let $F\in \on{Fun}(P_n,\mathcal{D})$ and let $u':F\rightarrow g^*g_!(F)$ be the unit map of the adjunction $g_!\dashv g^*$. \Cref{dualrem} shows that in the biCartesian square in $\mathcal{D}$
\[
\begin{tikzcd}
g_*(F) \arrow[d, "{g_*(u')}"'] \arrow[r] \arrow[rd, "\square", phantom] & 0 \arrow[d] \\
g_*g^*g_!(F) \arrow[r, "\alpha"']                       & g_*T_\mathcal{X}'(F) 
\end{tikzcd}
\]
the edge $g_*(u')$ is equivalent to the edge $g_*(F)\xrightarrow{(0,id)} g_!(F)\oplus g_*(F)$. We obtain that $\alpha$ is equivalent to the edge $g_!(F)\oplus g_*(F)\xrightarrow{(id,0)} g_!(F)$. The unit map $u:g_!(F)\rightarrow g_*g^*g_!(F)$ of the adjunction $g^*\dashv g_*$ is by \Cref{sphlem1} equivalent to the edge $g_!(F)\xrightarrow{(id,0)} g_!(F)\oplus g_*(F)$. The equivalences $g_*g^*g_!(F)\simeq g_!(F)\oplus g^*(F)$ in the description of the unit map and $\alpha$ are compatible as discussed in \Cref{sphrem2}. We thus obtain that the composition $\alpha\circ u:g_!(F)\rightarrow g_*T_\mathcal{X}'(F)$ is an equivalence. Thus condition 1. of \Cref{2/4lem1} is satisfied.

Let $F\in \on{Fun}(P_n,\mathcal{D})$ and let $cu':g^*g_*(F)\rightarrow F$ be the counit map of the adjunction $g^*\dashv g_*$. Consider the following biCartesian square in $\mathcal{D}$.
\[ 
\begin{tikzcd}
g_!T_\mathcal{X}(F) \arrow[d, "\beta"'] \arrow[r] \arrow[rd, "\square", phantom] & 0 \arrow[d] \\
g_!g^*g_*(F) \arrow[r, "g_!(cu')"']                                     & g_!(F)        
\end{tikzcd}
\] 
By \Cref{sphlem2} we find that $g_!(cu')$ is equivalent to the edge $g_!(F)\oplus g_*(F)\xrightarrow{(id,0)} g_!(F)$. We thus obtain that $\beta$ is equivalent to the edge $g_*(F)\xrightarrow{(0,id)}g_!(F)\oplus g_*(F)$. By \Cref{sphlem1} we find that the counit map $cu:g_!g^*g_*(F)\rightarrow g_*(F)$ is equivalent to the edge $g_!(F)\oplus g_*(F)\xrightarrow{(0,id)}g_*(F)$. Using \Cref{sphrem2}, we find that the composition $cu\circ\beta:g_!T_\mathcal{X}(F)\rightarrow g_*(F)$ is an equivalence and that condition 2. of \Cref{2/4lem1} is satisfied.
\end{proof}

\begin{corollary}
Let $\mathcal{D}$ be a stable $\infty$-category and $n\geq 0$. Consider the pullback functor\linebreak  $f^*:\mathcal{D}\rightarrow \on{Fun}(S^n,\mathcal{D})$ with right adjoint $f_*$. The adjunction $f^*\dashv f_*$ is spherical.
\end{corollary}

\begin{proof}
The adjunction $f^*\dashv f_*$ arises by \Cref{lem:fgres} as the restriction of the spherical adjunction $g^*\dashv g_*$. We apply \Cref{rescor} (with $F=f_*$) to deduce the sphericalness. For that, we need to show that $\on{Im}(g^*),\on{Im}(g^{**})\subset \on{Fun}(S^n,\mathcal{D})$, with $g^{**}$ the right adjoint of $g_*$. The inclusion $\on{Im}(g^*)\subset \on{Fun}(S^n,\mathcal{D})$ follows from the commutativity of the left diagram in \eqref{eq:fgcomm}. By \Cref{4pedprop2,sphrem1}, we have $g^{**}\simeq g^*\circ [n]$, so that $\on{Im}(g^{**})=\on{Im}(g^*)$, concluding the proof.
\end{proof}

\begin{remark}
Let $\mathcal{D}=\on{Sp}$ be the stable $\infty$-category of spectra. Let $E\in \on{Sp}$. The homotopy groups of the spectra $f_!f^*(E)$ and $f_*f^*(E)$ describe the homology and cohomology groups of the $n$-sphere with values in the spectrum $E$, respectively, i.e. 
\[ \pi_i(f_!f^*(E))\simeq H_i(X,E)\text{ and }\pi_{-i}(f_*f^*(E))\simeq H^i(X,E)\,.\] 
We thus consider $f_*,f_!:\on{Fun}(S^n,\on{Sp})\rightarrow \on{Sp}$ as the homology and cohomology functors for local systems of spectra on $S^n$. Using \Cref{sphrem1} and the 2/4 property, we find an equivalence  $[n]\circ f_* \simeq f_!$. The sphericalness of the adjunction $f^*\dashv f_*$ hence implies Poincaré duality for local systems on the $n$-sphere with values in spectra.
\end{remark}

We end this section with a conjecture for a possible generalization of the spherical adjunction $g^*\dashv g_*$. Consider a good stratification $A$ of the $n$-sphere. Denote by $\on{Sing}^A(S^n)$ the $\infty$-category of exit paths, see \cite[A.6.2]{HA}. In \cite[Section 8.6]{Tan19}, building on \cite[Appendix A]{HA}, it is shown that if $\mathcal{D}$ is a compactly generated $\infty$-category, then the $\infty$-category of functors $\on{Fun}(\on{Sing}^A(S^n),\mathcal{D})$ embeds fully faithfully into the $\infty$-category $\on{Shv}(S^n,\mathcal{D})$ of sheaves on the $n$-sphere with values in $\mathcal{D}$. The essential image of the embedding is given by constructible sheaves with respect to the stratification $A$. We have shown in \Cref{sphprop} that for a specific stratification of the $n$-sphere the pullback-limit adjunction $\mathcal{D}\leftrightarrow \on{Fun}(\on{Sing}^A(S^n),\mathcal{D})$ is spherical. We conjecture that for any good stratification the pullback-limit adjunction is spherical and arises as the restriction of a spherical adjunction involving sheaves on the $n$-sphere.

\begin{conjecture}
Let $\mathcal{D}$ be a compactly generated $\infty$-category and $n\geq 0$. Then the pullback-pushforward adjunction $\mathcal{D}\leftrightarrow \on{Shv}(S^n,\mathcal{D})$ is spherical.
\end{conjecture}

\subsection{Twist along a spherical fibration}\label{sphfibtwist} 

Let $\mathcal{D}$ be a stable $\infty$-category, $X$ and $Y$ Kan complexes and $f:X\rightarrow Y$ be a Kan fibration such that for all $y\in Y$ the fiber satisfies 
\[f^{-1}(y)\coloneqq \{y\}\times_{Y}X\simeq S^n\,.\]
We refer to such an $f$ as a spherical fibration. The pullback functor $f^*:\text{Fun}(Y,\mathcal{D})\longrightarrow \text{Fun}(X,\mathcal{D})$ admits left and right adjoints $f_!,f_*$, given by left and right Kan extension, as follows from the next Lemma and \cite[4.3.3.7]{HTT}.

\begin{lemma}\label{extlem}
Let $f:X\rightarrow Y$ be a spherical fibration, $F\in\on{Fun}(Y,\mathcal{D})$ and $y\in \mathcal{Y}$. Then the left Kan extension $f_!(F)$ and the right Kan extension $f_*(F)$ exist and satisfy 
\[f_!(F)(y)\simeq \underset{{f^{-1}(y)}}{\on{colim}}F\simeq \underset{S^n}{\on{colim}}F\]
and 
\[ f_*(F)(y)\simeq \lim_{f^{-1}(y)}F\simeq \lim_{S^n}F\,.\]
\end{lemma}

\begin{proof}
Using that $f$ is a Kan fibration, the statement follows from \cite[4.3.3.10]{HTT}.
\end{proof}

We now proof the sphericalness of the adjunction $f^*\dashv f_*$. The proof is essentially a relative version of the proof of \Cref{sphprop}.

\begin{proposition}\label{relsphprop}
Let $f:X\rightarrow Y$ be a spherical fibration and $\mathcal{D}$ a stable $\infty$-category. The adjunction $f^*:\on{Fun}(Y,\mathcal{D})\leftrightarrow \on{Fun}(X,\mathcal{D}):f_*$ is spherical.
\end{proposition}

\begin{proof}
We apply \Cref{2/4lem1} to show the sphericalness. Denote the twist functor of $f_!\dashv f^*$ by $T_\mathcal{X}'$. Let $F\in \on{Fun}(X,\mathcal{D})$ and let $u':F\rightarrow f^*f_!(F)$ be the unit map of the adjunction $f_!\dashv f^*$. The unit map $u'$ has the property that the restriction $u'|_{f^{-1}(y)}:F|_{f^{-1}(y)}\rightarrow f^*f_!(F)|_{f^{-1}(y)}$ to the fiber $f^{-1}(y)\simeq S^n$ of any $y\in Y$ is equivalent to the unit map of the adjunction $h_!:\on{Fun}(f^{-1}(y),\mathcal{D})\leftrightarrow\mathcal{D}:h^*$ where $h:f^{-1}(y)\rightarrow \ast$. By \Cref{sphlem2}, we find that in the biCartesian square in $\on{Fun}(Y,\mathcal{D})$
\[
 \begin{tikzcd}
  f_*(F) \arrow[d, "{f_*(u')}"'] \arrow[r] \arrow[rd, "\square", phantom] & 0 \arrow[d] \\
  f_*f^*f_!(F) \arrow[r, "\alpha"']                       & f_*T_\mathcal{X}'(F) 
 \end{tikzcd}
\]
the restriction of the edge $f_*(u')$ to any $y\in Y$ is equivalent to the edge $f_*(F|_{f^{-1}(y)})\xrightarrow{(0,id)} f_!(F|_{f^{-1}(y)})\oplus f_*(F|_{f^{-1}(y)})$. We thus obtain that the restriction of $\alpha$ to any $y\in Y$ is equivalent to the edge $f_!(F|_{f^{-1}(y)})\oplus f_*(F|_{f^{-1}(y)})\xrightarrow{(id,0)} f_!(F|_{f^{-1}(y)})$. The unit map $u:f_!(F)\rightarrow f_*f^*f_!(F)$ of $f^*\dashv f_!$ is by \Cref{Kanextlem} a right Kan extension. By \cite[4.3.3.10]{HTT} and  \Cref{sphlem1} we obtain that the restriction of $u$ to any $y\in \mathcal{Y}$ is equivalent to the edge $f_!(F|_{f^{-1}(y)})\xrightarrow{(id,0)} f_!(F|_{f^{-1}(y)})\oplus f_*(F|_{f^{-1}(y)})$.  We thus obtain that the composition $\alpha\circ u:f_!(F)\rightarrow f_*T_\mathcal{X}'(F)$ restricts on every $y\in Y$ to an equivalence and is hence a natural equivalence. We have shown that condition 1 of \Cref{2/4lem1} is satisfied.

Condition 2 of \Cref{2/4lem1} is shown analogously and can also be compared to the second part of the proof of \Cref{sphprop}.
\end{proof}

\subsection{Local systems with values in a symmetric monoidal \texorpdfstring{$\infty$}{infinity}-category}\label{sec3.3}

Let $\mathcal{C}$ be a symmetric monoidal and stable $\infty$-category such that the monoidal product preserves colimits in both variables. We describe in this section the twist functor of the spherical adjunction of \Cref{relsphprop} in terms of the monoidal product with a local system $\zeta$. 

\begin{lemma}
 Let $Z$ be a simplicial set and $q:\mathcal{C}^\otimes\rightarrow \on{Fin}_\ast$ a symmetric monoidal $\infty$-category. Then $\on{Fun}(Z,\mathcal{C})$ can be endowed with the structure of a symmetric monoidal $\infty$-category, with total space $ \on{Fun}(Z,\mathcal{C})^\otimes \coloneqq \on{Fun}(Z,\mathcal{C}^\otimes)\times_{\on{Fun}(Z,\on{Fin}_\ast)}\on{Fin}_\ast$. An edge in $\on{Fun}(Z,\mathcal{C})^\otimes$ is coCartesian if and only if its restriction to each vertex of $Z$ yields a $q$-coCartesian edge.
\end{lemma}

\begin{proof}
The map $\on{Fun}(Z,\mathcal{C}^\otimes)\rightarrow \on{Fun}(Z,\on{Fin}_\ast)$ is by \cite[3.1.2.1]{HTT} a coCartesian fibration whose coCartesian edges are given by edges whose restriction to each vertex of $Z$ yields a coCartesian edge in $\mathcal{C}^\otimes$. Consider the pullback $\on{Fun}(Z,\mathcal{C})^\otimes=\on{Fun}(Z,\mathcal{C}^\otimes)\times_{\on{Fun}(Z,\on{Fin}_\ast)}\on{Fin}_\ast$. The induced functor $p:\on{Fun}(Z,\mathcal{C})^\otimes\rightarrow \on{Fin}_\ast$ is also a coCartesian fibration. We note that the coCartesian edges of the fibration $p$ are also given by edges whose restriction to each vertex in $Z$ yields a coCartesian edge in $\mathcal{C}^\otimes$. Using that $\mathcal{C}$ is symmetric monoidal, it follows that the coCartesian fibration $p:\on{Fun}(Z,\mathcal{C})^\otimes \rightarrow \on{Fin}_\ast$ is also symmetric monoidal.
\end{proof}

\begin{lemma}
\label{mndsphlem}
Let $\mathcal{C}$ be a stable symmetric monoidal $\infty$-category and let $f:X\rightarrow Y$ be a map of simplicial sets. The pullback functor $f^*:\on{Fun}(Y,\mathcal{C})\rightarrow \on{Fun}(X,\mathcal{C})$ can be extended to a symmetric monoidal functor 
\[ (f^*)^\otimes:\on{Fun}(Y,\mathcal{C})^\otimes\longrightarrow \on{Fun}(X,\mathcal{C})^\otimes\,,\]
as defined in \cite[2.1.3.7]{HA}. The functor $(f^*)^\otimes$ admits a right adjoint  
\[ (f_*)^\otimes:\on{Fun}(X,\mathcal{C})^\otimes \longrightarrow \on{Fun}(Y,\mathcal{C})^\otimes\,,\]
whose restriction to $\on{Fun}(X,\mathcal{C})^\otimes_{<n>}\simeq \on{Fun}(X,\mathcal{C})^{\times n}$ is given by applying $f_*$ to each component. 
\end{lemma}

\begin{proof}
Consider the pullback functor $\alpha:\on{Fun}(Y,\mathcal{C}^\otimes)\longrightarrow \on{Fun}(X,\mathcal{C}^\otimes)$ along $f$. The restriction of $\alpha$ to $\on{Fun}(Y,\mathcal{C})^\otimes$ factors through the inclusion $\on{Fun}(X,\mathcal{C})^\otimes\subset \on{Fun}(X,\mathcal{C}^\otimes)$. The restriction of $\alpha$ to $\on{Fun}(Y,\mathcal{C})^\otimes_{\langle 1\rangle}\simeq \on{Fun}(Y,\mathcal{C})$ is equivalent to $f^*$. To show that the resulting functor $(f^*)^\otimes=\alpha|_{\on{Fun}(Y,\mathcal{C})^\otimes}:\on{Fun}(Y,\mathcal{C})^\otimes\rightarrow \on{Fun}(X,\mathcal{C})^\otimes$ is symmetric monoidal, we need to show that it preserves coCartesian edges. An edge $y\rightarrow y'$ in $\on{Fun}(Y,\mathcal{C})^\otimes$ is coCartesian with respect to the pointwise monoidal structure if and only if all its restrictions to vertices in $Y$ are coCartesian edges in $\mathcal{C}^\otimes$. Using the analogous characterization of the coCartesian edges in $\on{Fun}(X,\mathcal{C})^\otimes$, it is apparent that $f^*$ preserves coCartesian edges.

 The description of the right adjoint of $(f_*)^\otimes$ follows from the theory of relative adjunctions, see \cite[7.3.2.7]{HA}.
\end{proof}

\begin{construction}\label{mndsphcon}
Let $\mathcal{C}$ be a stable $\infty$-category and let $f:X\rightarrow Y$ be a Kan fibration. Consider the pullback functor $f^*:\on{Fun}(Y,\mathcal{C})\rightarrow \on{Fun}(X,\mathcal{C})$. Let $y \in Y$. We denote by $h^*:\mathcal{C}\rightarrow \on{Fun}(f^{-1}(y),\mathcal{C})$ the pullback along the map $f^{-1}(y)\rightarrow \Delta^0$. There is a natural transformation $\eta$ between the functors $f^*,h^*:\Delta^1\rightarrow \on{Set}_\Delta$ corresponding to the following commutative diagram in $\on{Set}_\Delta$, where the vertical edges are given by the evaluation functors.
\[
\begin{tikzcd}
{\on{Fun}(Y,\mathcal{C})} \arrow[d, "ev_y"] \arrow[r, "f^*"]& {\on{Fun}(X,\mathcal{C})}  \arrow[d, "ev_{f^{-1}(y)}"] \\
\mathcal{C}                             \arrow[r, "h^*"]& {\on{Fun}(f^{-1}(y),\mathcal{C})}                    
\end{tikzcd}
\] 
One checks that the diagram commutes in the $1$-category $\on{Set}_\Delta$, using the explicit description of the pullback and evaluation functors as maps between simplicial sets. The natural transformation $\eta$ induces a functor $\alpha:\Gamma(f^*)\rightarrow \Gamma(h^*)$ between the Grothendieck constructions. By \Cref{Kanextlem} an edge $E\rightarrow F$ lying over $0\rightarrow 1$ in $\Gamma(f^*)$ is Cartesian if and only if it is a right Kan extension. \cite[4.3.3.10]{HTT} shows that being a right Kan extension is a local property, namely this is the case if and only if for all $y\in Y$, the restricted map $f^*(E)|_{f^{-1}(y)}\rightarrow F|_{f^{-1}(y)}$ is a right Kan extension. We obtain that an edge in $\Gamma(f^*)$ lying over $0\rightarrow 1$ is Cartesian if and only if for all $y\in Y$ its restrictions in $\Gamma(h^*)$ is Cartesian.
\end{construction}

\begin{notation}
Let $\mathcal{C},Y$ be simplicial sets. We denote $\mathcal{C}^Y\coloneqq\on{Fun}(Y,\mathcal{C}).$
\end{notation}

\begin{proposition}\label{mndsphprop}
Let $\mathcal{C}$ be a symmetric monoidal and stable $\infty$-category and let $f:X\rightarrow Y$ be a spherical fibration. Denote by $T_{\mathcal{C}^Y}$ the twist functor of the adjunction $f^*:\mathcal{C}^Y\leftrightarrow \mathcal{C}^X:f_*$. Let $\zeta=\on{cof}(1_Y\xrightarrow{u} f_*f^*(1_Y))\in \mathcal{C}^Y$ where $u$ is a unit map of the adjunction $f^*\dashv f_*$. There exists an equivalence of endofunctors $T_{\mathcal{C}^Y}\simeq  \mhyphen \otimes \zeta$ of $\mathcal{C}^Y$.
\end{proposition}

The idea of the proof is to construct a natural transformation $\mhyphen\otimes \zeta \rightarrow T_{\mathcal{C}^Y}$ via Kan extension which is pointwise an equivalence.

\begin{proof}[Proof of \Cref{mndsphprop}]
Consider the the $\infty$-category $\mathcal{D}_1$ spanned by (commutative) diagrams in $\Gamma(f^*)$ as on the left in \eqref{2diag},
\begin{equation}\label{2diag}
\begin{tikzcd}
E\arrow[rr, "\simeq"]&& E\\
E \arrow[u, "\simeq"]\arrow[rr, "\simeq"] \arrow[rd, "\simeq"] &                                                & E \arrow[ld, "\simeq"] \arrow[dd, "!"] \arrow[u, "\simeq"]\\
                                            & E \arrow[d, "!"] \arrow[rd, "!"]               &                                        \\
                                            & f^*(E) \arrow[r, "\simeq"] \arrow[d, "\simeq"] & f^*(E) \arrow[ld, "\simeq"]            \\
                                            & f^*(E)                                         &                                       
\end{tikzcd}\qquad \begin{tikzcd}
0\arrow[rr]&&\zeta\\
1_Y \arrow[u]\arrow[urr, "\square", phantom]\arrow[rr] \arrow[rd, "!"] &                                    & f_*(1_X) \arrow[ld, "\ast"] \arrow[dd, "!"]\arrow[u] \\
                               & 1_X \arrow[d, "\simeq"] \arrow[rd] &                                             \\
                               & 1_X \arrow[r] \arrow[d, "\simeq"]  & f^*f_*(1_X) \arrow[ld]                      \\
                               & 1_X                                &                                            
\end{tikzcd}
\end{equation}
where $E\in \mathcal{C}^Y$ denotes any vertex. The restriction functor to the initial vertex (the lower left vertex in the upper square) is a trivial fibration from $\mathcal{D}_1$ to $\mathcal{C}^Y$. This follows from the fact that a vertex in $\mathcal{D}_1$ is the repeated Kan extension of its restriction to the initial vertex. Consider the $\infty$-category $\mathcal{D}_2$ spanned by diagrams in $\Gamma(f^*)$ as on the right in \eqref{2diag}. The restriction functor to the initial vertex induces a trivial fibration from $\mathcal{D}_2$ to the full subcategory $\langle 1_Y\rangle$ of $\mathcal{C}^Y$ spanned by the unit objects of the monoidal structure. We obtain a trivial fibration from the product $\infty$-category $\mathcal{D}_1\times\mathcal{D}_2$ to $\mathcal{C}^Y\times \langle 1_Y\rangle$. We note that $\mathcal{C}^Y\times \langle 1_Y\rangle\subset \mathcal{C}^Y\times \mathcal{C}^Y\subset \Gamma((f^*)^\otimes)_{\langle 2\rangle}.$ The product $\infty$-category $\mathcal{D}_1\times\mathcal{D}_2$ however does not include into any $\infty$-category of diagrams in $\Gamma((f^*)^\otimes)_{\langle 2\rangle}$. This is because of the vertices in the product diagram being mapped to elements in $\mathcal{C}^Y\times\mathcal{C}^X$. Given an element of $\mathcal{D}_1\times\mathcal{D}_2$, it can however be restricted to a diagram in $\Gamma((f^*)^\otimes)_{\langle 2\rangle}$ as follows. Consider the $\infty$-category $\mathcal{D}_3$ of diagrams in $\Gamma((f^*)^{\otimes})_{\langle 2\rangle}$ of the following form. 
\begin{equation}\label{diag3}
 \begin{tikzcd}
{(E,0)} \arrow[r]                       & {(E,\zeta)}                        \\
{(E,1_Y)} \arrow[u] \arrow[r] \arrow[d] & {(E,f_*(1_Y))} \arrow[u] \arrow[d] \\
{(f^*(E),1_X)} \arrow[r] \arrow[d]      & {(f^*(E),f^*f_*(1_X))} \arrow[ld]  \\
{(f^*(E),1_X)}                          &                                   
\end{tikzcd}
\end{equation}
The components of the edges are given by the corresponding edges in the diagrams \eqref{2diag}. The $\infty$-category $\mathcal{D}_3$ is a full subcategory of the functor category $\on{Fun}(Z,\Gamma((f^*)^{\otimes})_{\langle 2\rangle})$ with \[Z=\left(\Delta^1\times\Delta^1\right)\coprod_{\Delta^1} \left((\Delta^1\times\Delta^1)\amalg_{\Delta^1}\Delta^2\right)\,.\] The restriction functor maps $\mathcal{D}_1\times\mathcal{D}_2$ to $\mathcal{D}_3$. Consider the full subcategory $\mathcal{D}_4$ of the $\infty$-category $\on{Fun}(\Delta^1\times Z,\Gamma((f^*)^{\otimes}))$ spanned by diagrams whose restriction to $\Delta^{\{0\}}\times Z$ lies in $\mathcal{D}_3$ and that are a left Kan extension relative $\Gamma((f^*)^{\otimes})\rightarrow \on{Fin}_\ast$ of their restriction to $\Delta^{\{0\}}\times Z$. For the relative left Kan extension we used the map $\Delta^1\times Z\rightarrow \on{Fin}_\ast, (0,z)\mapsto \langle 2\rangle, (1,z)\mapsto \langle 1\rangle$. We note that the restriction functor $\mathcal{D}_4\rightarrow \mathcal{D}_3$ is a trivial fibration. The restrictions of the vertices of $\mathcal{D}_4$ to $\on{Fun}(\Delta^{\{1\}}\times Z,\Gamma((f^*)^{\otimes}))$  have the following form.
\begin{equation}\label{diag25}
 \begin{tikzcd}
E\otimes 0\arrow[r]& E\otimes \zeta\\
E\otimes 1_Y \arrow[d, "!"] \arrow[u]\arrow[ur, phantom, "\square"]\arrow[r]           & E\otimes f_*(1_X) \arrow[u]\arrow[d, "!"]     \\
f^*(E)\otimes 1_X \arrow[r] \arrow[d, "\simeq"] & f^*(E)\otimes f^*f_*(1_X) \arrow[ld] \\
f^*(E)\otimes 1_X                               &                                     
\end{tikzcd}
\end{equation}
The functor $f^*$ is monoidal by \Cref{mndsphlem}. We use this to deduce that the so labeled edges in \eqref{diag25} are coCartesian. We use that the monoidal product preserves colimits in the second entry to deduce that the upper square is biCartesian. 
\Cref{mndsphcon} shows that the composite edge \[\alpha:E\otimes f_*(1_X)\xrightarrow{!} f^*(E)\otimes f^*f_*(1_X) \rightarrow f^*(E)\otimes 1_X\] is a Cartesian edge in $\Gamma(f^*)$ if and only if all restrictions to points are Cartesian, which can be directly checked. Consider the $\infty$-category $\mathcal{D}_5$ spanned by diagrams of the following form.
\begin{equation}\label{findiag3}\begin{tikzcd}
              & E\otimes 0 \arrow[dr, phantom, "\square"]\arrow[r] \arrow[ld, "\simeq"']                                         & E\otimes \zeta                                                                   \\
0 \arrow[dddr, phantom, "\square", xshift=-3ex]\arrow[ddd] & E\otimes 1_Y \arrow[d, "!"] \arrow[r] \arrow[ddd, bend right=50] \arrow[l] \arrow[u] & E\otimes f_*(1_X) \arrow[d, "!"'] \arrow[lddd, "\simeq", dotted, bend left=50] \arrow[u] \\
              & f^*(E)\otimes 1_X \arrow[r] \arrow[d, "\simeq"]                                   & f^*(E)\otimes f^*f_*(1_X) \arrow[ld]                                             \\
              & f^*(E)\otimes 1_X                                                                 &                                                                                  \\
T_{\mathcal{C}^Y}(E)        & f_*f^*(E) \arrow[u, "\ast"'] \arrow[l]                                            &                                                                                 
\end{tikzcd}
\end{equation}
The label of the left biCartesian square refers to the square containing the bent arrow. The dotted edge is part of the diagram $\mathcal{D}_5$, the dotting is only for better readability. Consider also the $\infty$-category of diagrams $\mathcal{D}_6$ of the form \eqref{findiag3}, with an added edge $E\otimes \zeta \xrightarrow{\simeq}  T_{\mathcal{C}^Y}(E)$ completing a cube containing the two biCartesian squares in the diagram and thus exhibiting an equivalence between the two biCartesian squares. We observe that the restriction functor induces is a trivial fibration from $\mathcal{D}_6$ to $\mathcal{D}_5$ using that the two squares are pushout. We further observe that the functor from $\mathcal{D}_4'\coloneqq\mathcal{D}_4\times_{\on{Fun}(\Delta^{\{1\}}\times Z,\Gamma((f^*)^{\otimes})}\mathcal{D}_6$ to $\mathcal{D}_4$ contained in the defining pullback diagram is a trivial fibration. We can compose with the trivial fibration $\mathcal{D}_4\rightarrow \mathcal{D}_3$ to obtain a trivial fibration $\mathcal{D}_4'\rightarrow \mathcal{D}_3$. We obtain a trivial fibration from $\mathcal{D}_7\coloneqq\mathcal{D}_4'\times_{\mathcal{D}_3}\mathcal{D}_1\times \mathcal{D}_2$ to $\mathcal{D}_1\times\mathcal{D}_2$. We note that the projection functor $\mathcal{C}^Y\times \langle 1_Y\rangle\rightarrow \mathcal{C}^Y$ is also a trivial fibration. In total we obtain a trivial fibration 
\begin{equation}\label{trivfib7} 
\mathcal{D}_7\rightarrow \mathcal{D}_1\times\mathcal{D}_2\rightarrow \mathcal{C}^Y\times \langle 1_Y\rangle\rightarrow \mathcal{C}^Y\,.
\end{equation}
The functors $T_{\mathcal{C}^Y}$ and $\mhyphen\otimes \zeta$ are obtained by choosing a section of \eqref{trivfib7} and composing with the restriction functor to the vertex $E\otimes \zeta$ and $T_{\mathcal{C}^Y}(E)$ in diagram \eqref{findiag3}, respectively. The composition of a section of \eqref{trivfib7} with the restriction functor to the edge $E\otimes\zeta\xrightarrow{\simeq} T_{\mathcal{C}^Y}(E)$ thus describes the desired natural equivalence.
\end{proof}

\section{Spherical monadic adjunctions}\label{sec5}
We recall the notions monad and monadic adjunction in \Cref{sec1.4}. In this section we investigate spherical monadic adjunctions. The following theorem characterizes the sphericalness of a monadic adjunction in terms of the properties of the monad.

\begin{theorem}\label{sphmndthm}
Let $\mathcal{D}$ be a stable $\infty$-category and let $M:\mathcal{D}\rightarrow \mathcal{D}$ be a monad with unit $u:\on{id}_{\mathcal{D}}\rightarrow M$. Consider the endofunctor $T_{\mathcal{D}}=\on{cof}(\on{id}_{\mathcal{D}}\xrightarrow{u}M)\in \on{Fun}(\mathcal{D},\mathcal{D})$. The following conditions are equivalent.
\begin{enumerate}
\item The endofunctor $T_{\mathcal{D}}$ is an equivalence and the unit $u$ satisfies $T_\mathcal{D}u\simeq uT_\mathcal{D}$.
\item The monadic adjunction $F:\mathcal{D}\leftrightarrow \on{LMod}_M(\mathcal{D}):G$ is spherical.
\end{enumerate}
\end{theorem}

\begin{proof}
Before we begin, we note that the endofunctor $T_{\mathcal{D}}$ is equivalent to the twist functor of the adjunction $F\dashv G$. 

We start by showing that condition 1 implies condition 2. We denote the cotwist functor of the adjunction $F\dashv G$ by $T_\mathcal{C}$. \Cref{commlem} shows that the following diagram commutes.
 \[
  \begin{tikzcd}
\on{LMod}_M(\mathcal{D}) \arrow[rd, "T_\mathcal{D}G"'] \arrow[rr, "T_{\mathcal{C}}"] &             & \on{LMod}_M(\mathcal{D}) \arrow[ld, "G"] \\
                                                                                & \mathcal{D} &                                    
\end{tikzcd}
 \]
We observe that the left adjoint of $T_\mathcal{D}G$ is given by $FT_\mathcal{D}^{-1}$. We apply \cite[4.7.3.5,\ 4.7.3.16]{HA} and deduce that the following two conditions imply that $T_\mathcal{C}$ is an equivalence.
\begin{enumerate}
\item[1)] The functor $T_\mathcal{D}G$ is monadic.
\item[2)] For every $d\in \mathcal{D}$, the unit map $d\rightarrow T_\mathcal{D}GFT^{-1}_\mathcal{D}(d)\simeq GT_\mathcal{C}FT^{-1}_\mathcal{D}$ of the adjunction \mbox{$FT^{-1}_\mathcal{D}\dashv T_\mathcal{D}G$} induces via the adjunction $F\dashv G$ an equivalence $F(d)\rightarrow T_\mathcal{C}FT_\mathcal{D}^{-1}(d)$.
\end{enumerate}
Consider the following commutative diagram.
\[
 \begin{tikzcd}
                                                                 & \mathcal{D} \arrow[rd, "T_\mathcal{D}^{-1}"] &             \\
\on{LMod}_M(\mathcal{D}) \arrow[rr, "G"] \arrow[ru, "T_\mathcal{D}G"] &                                              & \mathcal{D}
\end{tikzcd}
\] The functor $T_\mathcal{D}^{-1}$ is an equivalence and hence conservative. We apply \cite[4.7.3.22]{HA} to deduce that $T_\mathcal{D}G$ is monadic. This shows 1).\\
For 2), consider the functor of $1$-categories $\alpha:[2]\rightarrow \on{Set}_\Delta$ corresponding to the composable functors $ \mathcal{D}\xrightarrow{T_\mathcal{D}^{-1}}\mathcal{D}\xrightarrow{F}\on{LMod}_M(\mathcal{D})$. We obtain a biCartesian fibration $p:\Gamma(\alpha)\rightarrow \Delta^2$. Let $d\in \mathcal{D}\simeq p^{-1}([0])$. Via Kan extension, we can produce the following diagram in $\Gamma(\alpha)$.
\[
 \begin{tikzcd}
  d \arrow[r, "!"] \arrow[d, "\simeq"']                                        & T_\mathcal{D}^{-1}(d) \arrow[d, "u'"] \arrow[r, "!"] & FT_\mathcal{D}^{-1}(d) \\
  T_\mathcal{D}T^{-1}_\mathcal{D}(d) \arrow[d, "T_\mathcal{D}(u')"'] \arrow[ru, "\ast"] & GFT_\mathcal{D}^{-1}(d) \arrow[ru, "\ast"]                  &                        \\
  T_\mathcal{D}GFT_\mathcal{D}^{-1}(d) \arrow[ru, "\ast"]                              &                                                             &                       
 \end{tikzcd}
\]
The map $u':T_\mathcal{D}^{-1}(d)\rightarrow GFT_\mathcal{D}^{-1}(d)$ is a unit map of the adjunction $F\dashv G$ and the map $v:d\xrightarrow{\simeq} T_\mathcal{D}T_\mathcal{D}^{-1}(d)\xrightarrow{T_\mathcal{D}(u')} T_\mathcal{D}GFT_\mathcal{D}^{-1}(d)$ is a unit map of the adjunction $FT_\mathcal{D}^{-1}\dashv T_\mathcal{D}G$. The map $v$ induces via the adjunction $F\dashv G$ an equivalence $F(d)\rightarrow T_\mathcal{D}FT_\mathcal{D}^{-1}(d)$ if $v$ is unit map of the adjunction $F\dashv G$, which follows from $uT_{\mathcal{D}}\simeq T_{\mathcal{D}}u$. We deduce that 2) is satisfied and the cotwist functor $T_\mathcal{C}$ is an equivalence. The monadic adjunction is thus spherical and we have shown condition 2.

We next show that condition 2 implies condition 1. If $F\dashv G$ is spherical, then by definition, it holds that $T_{\mathcal{D}}$ is an equivalence. Consider the unit $v:id_\mathcal{D}\rightarrow T_\mathcal{D}GFT^{-1}_\mathcal{D}\simeq GF$ of the adjunction $FT^{-1}_\mathcal{D}\dashv T_\mathcal{D}G.$ We observe that $v\simeq T_\mathcal{D}u'T^{-1}_\mathcal{D}$, where $u'$ is the unit of the adjunction $F\dashv G$. The cotwist functor $T_\mathcal{C}$ of $F\dashv G$ is by assumption an equivalence. Using \Cref{commlem}, we obtain equivalences $GT_\mathcal{C}\simeq T_\mathcal{D}G$ and $FT_\mathcal{D}\simeq T_\mathcal{C}F$ and thus $F T_\mathcal{D}^{-1}\simeq T_\mathcal{C}^{-1} F$. This shows that the unit $v'$ of the adjunction $T^{-1}_\mathcal{C}F\dashv GT_\mathcal{C}$ is equivalent to $v$. We observe that $v'$ is equivalent to $u'$. It thus follows that $u'$ is equivalent to $v$ and condition 1 is fulfilled.
\end{proof}

\begin{remark}
The proof of \Cref{thm1} shows that it suffices to check the commutativity $uT_\mathcal{D}\simeq T_{\mathcal{D}}u$ pointwise.
\end{remark}

\Cref{sphmndthm} can be seen as extending a part of the discussion in \cite[Section 3.2]{Seg18}, which focuses on Kleisli adjunctions. \Cref{sphmndthm} further allows us to extend the main result of \cite{Seg18} from the setting of pretriangulated dg-categories to stable $\infty$-categories.

\begin{corollary}\label{cor:sph}
Let $\mathcal{D}$ be a stable $\infty$-category and let $T:\mathcal{D}\rightarrow \mathcal{D}$ be an autoequivalence. Then $T$ arises as the twist functor of a spherical adjunction. This adjunction can further be chosen to be monadic or a stable Kleisli adjunction. 
\end{corollary}

\begin{proof}
Consider the endofunctor $M=id_\mathcal{D}\oplus T$. By \cite[Section 7.3.4]{HA}, we can equip $M$ with the structure of a monad, called the square-zero extension monad. The unit map of $M$ is given by the inclusion $u:id_\mathcal{D}\xrightarrow{(id,0)} id_\mathcal{D}\oplus T=M$. The twist functor $T_\mathcal{D}$ of the monadic adjunction $F:\mathcal{D}\leftrightarrow \on{LMod}_M(\mathcal{D}):G$ is thus equivalent to $T$. The unit $u$ clearly commutes with $T_\mathcal{D}$. The adjunction $F\dashv G$ is spherical by \Cref{sphmndthm}. The stable Kleisli adjunction of $M$ is a restriction of the spherical monadic adjunction and thus also spherical with twist functor equivalent to $T$.
\end{proof}

\begin{remark}
Many examples of spherical adjunctions are monadic or comonadic, e.g.~the spherical adjunction described in \Cref{relsphprop} is comonadic. As we argue in the following, for monadicity it suffices in good cases that one of the involved functor is conservative. Let $\mathcal{C}$ and $\mathcal{D}$ be stable $\infty$-categories that admit sufficient limits of cosimplicial objects and colimits of simplicial objects. Consider a spherical adjunction $F:\mathcal{C}\leftrightarrow \mathcal{D}: G$. The functors $F$ and $G$ admit by \Cref{4pedprop2} further left and right adjoints, which implies that $F$ and $G$ preserve all limits and colimits. The $\infty$-categorical Barr-Beck theorem thus implies that the adjunction $F\dashv G$ is monadic if and only if the functor $G$ is conservative and comonadic if and only if the functor $F$ is conservative.

Denote the right adjoint of $G$ by $H$. Under the above assumptions, if the adjunction $F\dashv G$ is monadic, then the adjunction $G\dashv H$ is comonadic and vice versa. Thus, if sufficient limits and colimits exist, spherical  monadic  adjunctions determine spherical comonadic adjunctions and vice versa.
\end{remark}

Consider an adjunction $F:\mathcal{D}\leftrightarrow \mathcal{C}:G$ of stable $\infty$-categories. Let $M\simeq GF$ be the adjunction monad. As shown in \Cref{stbKleisli}, there exists a fully faithful functor $F':\overline{\on{LMod}_M^{\on{free}}(\mathcal{D})}\rightarrow \mathcal{C}$ from the stable Kleisli $\infty$-category of $M$ to $\mathcal{C}$. The essential image of $F'$ is identical to the stable closure of the essential image of $F$ in $\mathcal{C}$. Combining \Cref{sphmndthm} and \Cref{rescor}, we obtain the following characterization of the sphericalness of the adjunction $F\dashv G$.

\begin{proposition}\label{sphmndprop}
Let $F:\mathcal{D} \leftrightarrow \mathcal{C}:G$ be an adjunction of stable $\infty$-categories. The adjunction $F\dashv G$ is spherical if and only if the following four conditions are satisfied.
\begin{enumerate}
\item The twist functor $T_\mathcal{D}$ of the adjunction $F\dashv G$ is an autoequivalence.
\item Let $u:id_\mathcal{D}\rightarrow GF$ be the unit of $F\dashv G$. There exists an equivalence $u T_\mathcal{D}\simeq T_\mathcal{D} u$.
\item The functor $G$ admits a right adjoint $H$.
\item The essential image $\on{Im}(H)$ is contained in the stable closure $\overline{\on{Im}(F)}$ of $\on{Im}(F)\subset \mathcal{C}$.
\end{enumerate}
\end{proposition}

\begin{proof}
Denote by $M$ the adjunction monad of the adjunction $F\dashv G$ with monadic adjunction $F'':\mathcal{D}\leftrightarrow \on{LMod}_M(\mathcal{D}):G''$. The equivalence $F':\overline{\on{LMod}_M^{\on{free}}(\mathcal{D})}\simeq \overline{\on{Im}(F)}$ constructed in \Cref{stbKleisli} identifies the functors $F$ and $F''$, i.e.~$F'\circ F''\simeq F$. Denote by  $T_{\on{LMod}_M(\mathcal{D})}$ and $T_{\mathcal{C}}$ the cotwist functors of the adjunctions $F\dashv G$ and $F''\dashv G''$, respectively. The equivalence $F'$ also identifies the restrictions of the cotwists, i.e.~$F'\circ T_{\on{LMod}_{M}(\mathcal{D})}|_{\overline{\on{LMod}_M^{\on{free}}(\mathcal{D})}} \circ F'^{-1}\simeq T_{\mathcal{C}}|_{\overline{\on{Im}(F)}}$.

Assume that the adjunction $F\dashv G$ is spherical. Condition 1 is immediate. Condition 2 follows from \Cref{sphmndthm} and the observation that the unit maps of the monadic adjunction of the adjunction monad $GF$ and the unit maps of the adjunction $F\dashv G$ are equivalent. \Cref{2/4prop} shows that the functor $G$ admits a right adjoint and there exists an equivalence $H\simeq FT^{-1}_\mathcal{D}$. In particular, we find $\on{Im}(H)=\on{Im}(F)$. This shows conditions 3 and 4. 

Assume conditions 1 to 4 are satisfied. Then the monadic adjunction $F''\dashv G''$ is spherical by \Cref{sphmndthm} and thus the restriction of $T_{\on{LMod}_M(\mathcal{D})}$ to $\overline{\on{LMod}_M^{\on{free}}(\mathcal{D})}$ is an equivalence. It follows that the restriction of $T_\mathcal{C}$ to $\overline{\on{Im}(F)}$ is an equivalence. Using condition 4, we can apply \Cref{rescor} to deduce that the adjunction $F\dashv G$ is spherical.  
\end{proof}

Consider a symmetric monoidal and stable $\infty$-category $\mathcal{C}$. Given an associative algebra object $A\in \mathcal{C}$, we describe in \Cref{endlem} a monad $A\otimes \mhyphen:\mathcal{C}\rightarrow \mathcal{C}$. The characterization of the sphericalness of the monadic adjunction of $A\otimes \mhyphen$ simplifies by the symmetry of the monoidal structure. Namely, the monadic adjunction is spherical if and only if the twist functor is an autoequivalence.

\begin{proposition}\label{symmndprop}
Let $\mathcal{C}^\otimes\rightarrow \on{Assoc}^\otimes$ be an $\on{Assoc}^\otimes$-monoidal stable $\infty$-category, such that the composite with $\on{Assoc}^\otimes\rightarrow \on{Fin}_\ast$ exhibits $\mathcal{C}^\otimes$ as a symmetric monoidal $\infty$-category. Assume that the monoidal product $\mhyphen \otimes \mhyphen:\mathcal{C}\times\mathcal{C}\rightarrow \mathcal{C}$ is exact in both entries. Let $A\in \mathcal{C}^\otimes$ be an associative algebra object and consider the free-forget adjunction $F:\mathcal{C}\leftrightarrow \on{LMod}_{A}(\mathcal{C}):G$. The adjunction $F\dashv G$ is spherical if and only if the twist functor $T_\mathcal{C}$ is an autoequivalence.
\end{proposition}

\begin{proof}
By \Cref{algmndlem}, the adjunction $F\dashv G$ is monadic with adjunction monad given by $M=i(A)=A\otimes\mhyphen$. Let $u$ be the unit of $A$ with cofiber $T$ in $\mathcal{C}$. The unit $u'$ of the monad $M$ is given by $\mhyphen \otimes u:id_\mathcal{D}\rightarrow M$. By the exactness of the monoidal product, the twist $T_{\mathcal{C}}$ is equivalent to $\mhyphen\otimes\on{cof}(u)$. Using that $\mathcal{C}^\otimes$ is symmetric monoidal, we obtain 
\[u'T_{\mathcal{C}}\simeq \mhyphen \otimes u \otimes \on{cof}(u) \simeq \mhyphen \otimes \on{cof}(u)\otimes u\simeq T_{\mathcal{C}}u'\,.\] 
The sphericalness is thus equivalent to $T_{\mathcal{C}}$ being an equivalence by \Cref{sphmndthm}.
\end{proof}

\subsection{Recovering a spherical adjunction from a section of the twist functor is not possible}\label{sec4.1}

Consider a spherical monadic adjunction $F:\mathcal{D}\leftrightarrow \mathcal{C}:G$ with adjunction monad $GF$. The section $s$ of the twist functor $T_\mathcal{D}$ is the natural transformation contained in the following fiber and cofiber sequence in the stable $\infty$-category $\on{Fun}(\mathcal{D},\mathcal{D})$.
\[
\begin{tikzcd}
{T_{\mathcal{D}}[-1]} \arrow[d, "s"] \arrow[r] \arrow[rd, "\square", phantom] & 0 \arrow[d] \\
id_{\mathcal{D}} \arrow[r, "u"]                                 & GF         
\end{tikzcd}
\]
We show that the adjunction $F\dashv G$ and the adjunction monad $GF$ cannot be recovered from the section $s$.

\begin{example}\label{ex1}
Let $k$ be a field. Consider the morphism of $k$-algebras $\nu:k[x]\rightarrow k[y],~x\mapsto y^2$. Geometrically, this corresponds to the map $\mathbb{A}^1\rightarrow \mathbb{A}^1,~x\mapsto x^2,$ i.e.~a branched covering. Consider the adjunction of stable $\infty$-categories 
\begin{equation}\label{mndadj1}
\mhyphen \otimes_{k[x]} k[y] =\nu_!:\mathcal{D}(k[x])\longleftrightarrow \mathcal{D}(k[y]):\nu^*=\on{RHom}_{k[y]}(k[y],\mhyphen)
\end{equation}
where $\mathcal{D}(k[x])$ denotes the unbounded derived $\infty$-category of $k[x]$. \Cref{algmndlem} implies that the adjunction \eqref{mndadj1} is monadic with the adjunction monad given by $M=k[y]\otimes_{k[x]}\mhyphen\,$. The $k[x]$-module $k[y]$ is isomorphic to the free module $k[x]\oplus k[x]$, by mapping monomials of even degree to the first component and monomials of odd degree to the second component. The underlying endofunctor of $M$ is thus equivalent to $id_{\mathcal{D}(k[x])}\oplus id_{\mathcal{D}(k[x])}$. The unit of the monad $M$ is given by the inclusion $id_{\mathcal{D}(k[x])}\xrightarrow{(id,0)} id_{\mathcal{D}(k[x])}\oplus id_{\mathcal{D}(k[x])}\simeq M$. It follows that the twist functor is equivalent to the identity functor. \Cref{symmndprop} implies that the adjunction $\nu_!\dashv \nu^*$ is spherical. The multiplication map of $M$ at $k[x]$ is the map $m:M^2(k[x])\simeq k[x]^{\oplus 4}\rightarrow k[x]^{\oplus 2}\simeq M(k[x])$, described by multiplication with the following matrix. 
\[
 \begin{pmatrix}
   id_{k[x]} & 0 & 0 & id_{k[x]} \\
   0 & id_{k[x]} & id_{k[x]} & 0
 \end{pmatrix}
\]
\end{example}

\begin{example}\label{ex2}
Let $k$ be a field. Consider the morphism of $k$-algebras $\mu:k[x]\rightarrow k[x,\epsilon]\coloneqq k[x,\epsilon]/(\epsilon^2),~ x\mapsto x$. Geometrically, this corresponds to the map $\mathbb{A}^1\times \on{Spec}(k[\epsilon])\rightarrow \mathbb{A}^1,~(x_1,x_2\epsilon) \mapsto x_1$.  Consider the adjunction of stable $\infty$-categories  
\begin{equation}\label{mndadj2}
\mhyphen\otimes_{k[x]} k[x,\epsilon]=\eta_!:\mathcal{D}(k[x])\longleftrightarrow \mathcal{D}(k[x,\epsilon]):\eta^*=\on{RHom}_{k[x,\epsilon]}(k[x,\epsilon],\mhyphen)\,.
\end{equation}
Again, \Cref{algmndlem} implies that the adjunction \eqref{mndadj2} is monadic with the adjunction monad given by $N=k[x,\epsilon]\otimes_{k[x]} \mhyphen\,$.
The $k[x]$-module $k[x,\epsilon]$ is equivalent to the direct sum $k[x]\oplus k[x]$ of free $k[x]$-modules. The underlying endofunctor of $N$ is thus equivalent to $id_{\mathcal{D}(k[x])}\oplus id_{\mathcal{D}(k[x])}$. The unit of the monad $N$ is given by the inclusion $id_{\mathcal{D}(k[x])}\xrightarrow{(id,0)} id_{\mathcal{D}(k[x])}\oplus id_{\mathcal{D}(k[x])}\simeq M$. It follows that the twist functor is equivalent to the identity functor. \Cref{symmndprop} implies that the adjunction $\eta_!\dashv \eta^*$ is spherical. 
The multiplication map of $M$ at $k[x]$ is the map $m:M^2(k[x])\simeq k[x]^{\oplus 4}\rightarrow k[x]^{\oplus 2}\simeq M(k[x])$, described by multiplication with the following matrix. 
\[
 \begin{pmatrix}
   id_{k[x]} & 0 & 0 & 0 \\
   0 & id_{k[x]} & id_{k[x]} & 0
 \end{pmatrix}
\]
\end{example}

\begin{remark}
The endofunctors underlying the monads $M,N:\mathcal{D}(k[x])\rightarrow \mathcal{D}(k[x])$ from the \Cref{ex1,ex2} are equivalent. The unit maps of $M$ and $N$ and the sections of the respective twist functors are also equivalent. However, $M$ is not equivalent as a monad to $N$ because $k[y]$ is not equivalent as a $k[x]$-algebra to $k[x,\epsilon]$.
\end{remark}

The following simplification of the \Cref{ex1,ex2} was suggested to us by Ed Segal.

\begin{example}
 Let $k$ be a field. Let $k\oplus k$ be the product $k$-algebra and $k[\epsilon]\coloneqq k[\epsilon]/(\epsilon^2)$ be the square-zero $k$-algebra. Then the monadic adjunctions of the monads $M=\mhyphen \otimes_{k}(k\oplus k):\mathcal{D}(k)\rightarrow \mathcal{D}(k)$ and $N=\mhyphen \otimes_k k[\epsilon]:\mathcal{D}(k)\rightarrow \mathcal{D}(k)$ are spherical with twist functor equivalent to the identity. There exists an equivalence of underlying endofunctors $M\simeq N$. Furthermore, the units $id_{\mathcal{D}(k)}\rightarrow M,N$ and thus the sections of the twist functors are equivalent. However $M$ and $N$ are not equivalent as monads.  
\end{example}

\bibliography{biblio} 
\bibliographystyle{alpha}

\end{document}